\documentclass{amsart}
\usepackage{amsfonts,amsthm,amsmath}
\usepackage{amssymb}
\usepackage{amscd}
\usepackage[utf8]{inputenc}
\usepackage[a4paper]{geometry}
\usepackage{enumerate}
\usepackage[usenames,dvipsnames]{color}
\usepackage{tikz}
\usepackage{soul}
\usepackage{array}
\usepackage{array,tabularx}

\DeclareMathOperator{\Coeff}{Coeff}
\DeclareMathOperator*{\Res}{Res}
\newcommand\note[1]{\mbox{}\marginpar{ \scriptsize\raggedright
\hspace{1pt}\color{red} #1}}

\numberwithin{equation}{section}
\numberwithin{equation}{subsection}

\theoremstyle{plain}

\newtheorem{theorem}[equation]{Theorem}
\newtheorem{lemma}[equation]{Lemma}
\newtheorem{proposition}[equation]{Proposition}

\newtheorem{corollary}[equation]{Corollary}

\newtheorem{conjecture}[equation]{Conjecture}
\newtheorem{thm}[equation]{Theorem}

\theoremstyle{definition}
\newtheorem{example}[equation]{Example}
\newtheorem{remark}[equation]{Remark}

\newtheorem{definition}[equation]{Definition}

\newtheorem{bekezdes}[equation]{}

\newcommand{\ce}[1]{\left\lceil #1 \right\rceil}

\def\C{\mathbb C}
\def\Q{\mathbb Q}

\def\Z{\mathbb Z}

\def\im{{\rm im}}

\newcommand{\calw}{{\mathcal W}}

\newcommand{\calv}{{\mathcal V}}

\newcommand{\calm}{{\mathcal M}}

\newcommand{\cali}{{\mathcal I}}
\newcommand{\calj}{{\mathcal J}}

\newcommand{\calF}{{\mathcal F}}
\newcommand{\calO}{{\mathcal O}}\newcommand{\calU}{{\mathcal U}}

\newcommand{\calS}{{\mathcal S}}

\newcommand{\calP}{\mathcal{P}}\newcommand{\calL}{\mathcal{L}}

\newcommand{\s}{r}
\newcommand{\tX}{\widetilde{X}}

\newcommand{\reg}{{\rm reg}}
\newcommand{\cX}{{\mathcal X}}
\newcommand{\cO}{{\mathcal O}}
\newcommand{\bP}{{\mathbb P}}
\newcommand*{\linebundle}{\mathcal{L}}
\newcommand{\bC}{{\mathbb C}}
\newcommand{\cF}{{\mathcal F}}

\newcommand{\eca}{{\rm ECa}}
\newcommand{\pic}{{\rm Pic}}

\newcommand{\m}{\mathfrak{m}}\newcommand{\fr}{\mathfrak{r}}
\newcommand{\mfl}{\mathfrak{L}}

\newcommand{\bt}{{\mathbf t}}

\newcommand{\frsw}{\mathfrak{sw}}

\newcommand{\bZ}{{\mathbb{Z}}}
\newcommand{\bQ}{{\mathbb{Q}}}

\author{J\'anos Nagy}
\address{Central European University, Dept. of Mathematics,  Budapest, Hungary}
\email{nagy\textunderscore janos@phd.ceu.edu}

\author{Andr\'as N\'emethi}
\address{Alfr\'ed R\'enyi Institute of Mathematics,
Hungarian Academy of Sciences,
Re\'altanoda utca 13-15, H-1053, Budapest, Hungary \newline
 \hspace*{4mm} ELTE - University of Budapest, Dept. of Geometry, Budapest, Hungary \newline \hspace*{4mm}
BCAM - Basque Center for Applied Math.,
Mazarredo, 14 E48009 Bilbao, Basque Country – Spain}
\email{nemethi.andras@renyi.mta.hu }

\title{The Abel map for surface singularities \\
I. Generalities and  examples}

\begin{document}



\begin{abstract}
Let  $(X,o)$ be a complex normal surface singularity. We
fix one of its  good resolutions  $\tX\to X$,   an effective cycle
$Z$ supported on the reduced
exceptional curve,  and any possible (first Chern) class $l'\in H^2(\tX,\Z)$.
With these data we define   the variety  $\eca^{l'}(Z)$
of those effective Cartier
divisors $D$ supported on $Z$ which determine a line bundles $\calO_Z(D)$ with first Chern class $l'$.
Furthermore, we consider the affine space $\pic^{l'}(Z)\subset H^1(\calO_Z^*)$ of isomorphism classes of holomorphic line bundles with Chern class $l'$ and the Abel map  $c^{l'}(Z):\eca^{l'}(Z)\to \pic^{l'}(Z)$.
The present manuscript develops the major properties of this map, and links them with the determination of
the cohomology groups $H^1(Z,\calL)$,
 where we might vary the analytic structure $(X,o)$ (supported on a fixed topological type/resolution graph) and we also vary the possible line bundles $\calL\in \pic^{l'}(Z)$.
The case of generic line bundles of $\pic^{l'}(Z)$ and generic line bundles of the image of the Abel map
will have priority roles. Rewriting the Abel map via Laufer duality based on integration of forms on
divisors, we can make explicit the Abel map and its tangent map. The case of
superisolated and weighted homogeneous singularities are exemplified with several details.

The theory has similar goals (but rather different techniques)
as the theory of Abel map or  Brill--Noether theory of reduced smooth projective curves.
\end{abstract}

\maketitle

\linespread{1.2}


\pagestyle{myheadings} \markboth{{\normalsize  J. Nagy, A. N\'emethi}} {{\normalsize Abel maps }}


\section{Introduction}\label{s:intr}
In this introduction we plan to provide the major ideas and some of the major results
without  technical details. The presentation will
 automatically provide the structure of the article as well.

The study of the Abel map of projective irreducible smooth curves was a crucial tool in the classical
algebraic geometry  and it remained so in the modern theory as well. Though in this work we will
use very little  from  this theory, in this introduction (and some places later)
we will discuss some comparisons between the curve case and the theory of the present article
established  for  normal surface singularities,
mostly to emphasize the major conceptual differences and additional difficulties in the later case.
(For the Abel map of curves one can consult \cite{ACGH} and the references therein.)

The present manuscript is the first one in a series of articles planed (and partly already written) by the authors.
It contains the foundation, the presentation of the basic constructions and of the basic
properties. They are also supported by several examples. The forthcoming manuscripts of the series
treat the theory applied for several important families of singularities, e.g. for
singularities with generic analytic type, or elliptic or splice quotient singularities. E.g., in the second article
\cite{NNII},
 based on the results of the present one, we treat properties of the  generic  analytic structure supported by  a fixed resolution  graph (topological type). More precisely,
 we are able to determine topologically several discrete analytic invariants of
 such singularities
like multivariable  Hilbert series associated with the divisorial filtration, or cohomology of cycles and line bundles
supported on the fixed   resolution (in particular,  the geometric genus as well).

We wish to emphasize from the start that we are not generalizing the Abel construction from the curve
case  to
the --- smooth or singular ---  (quasi)projective surfaces: our goal is to develop its analogue
valid in the context of a resolution of
a complex normal surface singularity germ.
This means that if $(X,o)$ is such a singularity with
a fixed good resolution  $\tX\to X$,  then for any effective cycle
$Z$ supported on the reduced
exceptional curve  $E$ and for any (possible) Chern class $l'\in H^2(\tX,\Z)$ we construct the space $\eca^{l'}(Z)$
of effective Cartier
divisors $D$ supported on $Z$, whose associated line bundles $\calO_Z(D)$ have first Chern class $l'$.
Furthermore, we consider the space $\pic^{l'}(Z)\subset H^1(\calO_Z^*)$ of isomorphism classes of holomorphic line bundles with Chern class $l'$ and the Abel map  $c^{l'}(Z):\eca^{l'}(Z)\to \pic^{l'}(Z)$, $D\mapsto \calO_Z(D)$.
In this way, our Abel map is associate with non--reduced projective curves supported by the exceptional
set of a good resolution of a normal surface singularity. In particular, the combinatorial background is the
combinatorics of the dual resolution graph $\Gamma$ (or the intersection from $(\,,\,)$
of the irreducible exceptional curves),
that is, equivalently,  the 3--dimensional link of the singularity. In fact, in order to run properly
the theory (e.g. to be able to define the `natural' line bundles, cf. \ref{ss:analinv}),
we will even assume that the
 link of the singularity is a rational homology sphere. This happens exactly when the resolution graph
 $\Gamma$ represents  a tree of rational curves. In this way, in all the discussions regarding the
  analytic  types and properties we move the difficulties  from the
 moduli space of each irreducible exceptional curve $E_v$ (which is trivial in this case)
 to the analytic properties of their
  infinitesimal tubular  neighbourhoods and their gluings (analytic plumbing).

Therefore, the Abel map $c^{l'}$ behaves rather differently than the (projective)
Abel map of reduced smooth curves,
it shares more  the properties of non--proper
affine maps rather than the projective ones.
This will also be clear from the next preliminary  presentation of its source and  target.

In fact, the space $\eca^{l'}(Z)$ is already constructed in the literature. Note that by
a theorem of Artin \cite[3.8]{Artin69}, there exists an affine algebraic variety $Y$ and a point $y\in Y$ such that $(Y,y)$ and $(X,o)$ have isomorphic formal completions. Then, according to Hironaka \cite{Hironaka65},
$(Y,y)$ and $(X,o)$ are analytically isomorphic. In particular, we can regard $Z$
as a projective  algebraic scheme, in which situation $\eca^{l'}(Z)$
 was constructed by Grothendieck \cite{Groth62}, see also the article of
Kleiman \cite{Kleiman2013} and
the book of Mumford for curves on algebraic surfaces \cite{MumfordCurves}.
In particular, $\eca^{l'}(Z)$ is a quasiprojective variety. Though the existence of the space
$\eca^{l'}(Z)$ in this way is established,
we will provide several key properties valid in our particular situation,
including the local charts.
E.g., we will characterize topologically when the space $\eca^{l'}(Z)$ is nonempty ($\eca^{l'}(Z)\not=\empty$
if and only if $-l'$ belongs to the Lipman cone, cf. (\ref{eq:empty})), and in these cases
we show that it is smooth of dimension $(l',Z)$, cf. Theorem \ref{th:smooth}.
Furthermore, there exists a natural projection to
$\eca^{l'}(E)$, whose fibers are affine spaces. They can be considered as certain jet spaces
in the local infinitesimal neighbourhoods of the of the local equations of the effective Cartier divisors.
This fiber structure  makes the space rather special, with non--proper/non--compact behaviors. In fact, by
fixing the Chern class even $\eca^{l'}(E)$ becomes non--projective too; e.g. for $l'=-E^*_v$
(the dual of $E_v$, representing `cuts' which intersects $E_v$ but not the other curves) we get
$\eca^{l'}(E)=E_v\setminus \cup_{u\not=v}E_u$.

Note also that the base space $\pic^{l'}(Z)$ is also noncompact, it is
an affine space, it has dimension $h^1(\calO_Z)$.
(Here the assumption that the link  is a rational homology sphere plays a role; otherwise
$\pic^{l'}(Z)=H^1(\calO_Z^*)/H^1(\tX,\Z)$ would have a complex torus component as well).
This affine structure will be exploited deeply in the body of the paper.
Finally we also mention that the Abel map itself is algebraic, and in fact its (rather non--trivial)
expression   in local charts
can be done explicitly via Laufer duality (integrating forms along divisors in $\tX$), for details see section 7.

Since the Abel map is not proper,  its image usually is not closed,
and it can be a rather complicated constructible
set (it can be singular as well, cf. Example \ref{ex:whSING}).
In this note we give several examples and also we characterize  the dimension of this image.
It  is not topological,
usually it depends in a subtle way on the analytic structure of the singularity. In order to
show the presence of possible anomalies we list
several examples based on the theory  of elliptic and splice quotient singularities
(certain familiarity with them might help essentially the reading).

We also show that all the fibers of $c^{l'}$ are smooth (irreducible, quasiprojective),
however, their dimensions might jump. The dimension of $c^{-1}(\calL)$ ($\calL\in \pic^{l'}(Z)$)
is $h^0(Z,\calL)-h^0(\calO_Z)=(l',Z)+h^1(Z,\calL)-h^1(\calO_Z)$. Any fiber appears as quotient by the
algebraic free proper action of $H^1(\calO_Z^*)$, which,  as algebraic variety, has dimension $h^1(\calO_Z)$.
(This also shows a major difference with the curve cases, where the space of effective divisors
associated with a bundle has the form $H^0(\calL)\setminus \{0\}$, and the action is the projectivization action of
$\C^*$. In particular, the fibers are projective spaces.)
The above relation makes the connection with another
 major problem/task of the theory, namely determination of
possible values of $h^1(Z,\calL)$.

This `$h^1$'--problem can be formulated even independently of the Abel map, and in fact, it was our
most important motivation.
Let us fix a topological type (say, the resolution graph $\Gamma$), and we consider an arbitrary analytic
type of singularity and its resolution supported by $\Gamma$. Then for fixed Chern class $l'$ and cycle $Z$
 we can also consider all the possible line bundles $\calL\in \pic^{l'}(Z)$. The challenge is
to determine all the possible values of $h^1(Z, \calL)$, and understand/organize  them is a
conceptual way. This can be split in two major steps: in the first case one varies all the analytic structures,
in the second case one fixes an analytic structure $(X,o)$
(and one of its resolutions $\tX$)
and one moves $\calL\in \pic^{l'}(Z)$. E.g., in this second case, one can ask for the stratification $\cup_kW_{l',k}$ of $\pic^{l'}(Z)\simeq H^1(\calO_Z)$ by $W_{l',k}=\{\calL\,:\,
h^1(\calL)=k\}$. (These are the analogues  of the Brill--Noether strata. For the Brill--Noether
theory see \cite{ACGH,Flamini}.) Or, one can search for  the possible values
$k$ when $W_{l',k}\not=\emptyset$. In the body of the article we will provides several bounds and partial results
(with sharp lower bounds provided by generic structures).
Though the older previous results in normal surface singularities focus mostly
on particular analytic structures (rational, elliptic, weighted homogeneous, splice quotient, etc),
and to special line bundles (e.g. of type $\calO_Z(l)$),
in the present note we aim to create a theory which helps to attack the general case, e.g. to
treat the case of generic analytic structure or the generic line bundles as well.

Part of the results  are  reduced to the case of Abel maps which are dominant. This case is completely characterized
and solved  in
section 4; we show  in Theorem \ref{th:dominant} that the fact that $c^{l'}(Z)$ is dominant depend only on
combinatorial properties of the pair $Z$ and $l'$, and furthermore,
in such a case, $h^1(Z,\calL_{gen})=0$ for
$\calL_{gen}$ generic in $\pic^{l'}(Z)$. For fixed and large $Z$ (in which case $\pic^{l'}(Z)=\pic^{l'}(\tX)$)
we introduce $\calS'_{dom}$ as the set of those Chern classes $l'$
for which $c^{-l'}$ is dominant, and we list several
properties of it. It is a semigroup of the topological Lipman semigroup/cone $\calS'$, and it has several properties
of the analytic semigroups. The study of dominant maps emphasizes again
 the importance of the study of generic
line bundles.
In section 5
we will list several cohomological properties for the generic line bundle $\calL_{gen}$ of
$\pic^{l'}$  (e.g. we determine its $h^1$ topologically, and we show that this value
 is a sharp lower bound for any $h^1(\calL)$). Similarly, the generic line bundle
 {\it of the image of the Abel map}
 $c^{l'}$ is also
 studied (its $h^1$ is the codimension of $\im (c^{l'})$ and it is also the sharp lower bound for any
 $h^1(\calL)$ with $\calL\in \im(c^{l'})$). Upper bounds for $h^1(Z,\calL)$ are also established.

The Abel map is compatible with additive structure of the divisors and multiplicative structure of the line bundles.
The point is that if we iterate a Chern class sufficiently many times (that is, we replace
$l'$ with $nl'$ where $n\gg 0$), then the image of $c^{nl'}$ becomes an affine space, whose associated vector spaces
 stabilizes, and which depends only on the `dual-base-support' of $l'$ (see Theorem \ref{th:mult}).
This collection of
 stabilized linear subspaces (as a linear subspace arrangement) and their
dimensions become the source of  important new analytic invariants, see section 6.
E.g., the dimensions serve as
 correction terms in our new analytic surgery formulae (see e.g. Theorem \ref{th:mult}).
If the analytic structure of $(X,0)$ is `nice' (e.g. splice quotient), then these correction
 invariants can be
connected with known analytic invariants computable from the Poincar\'e series of the divisorial filtrations), and in such cases classical formulae can be recovered or improved (see section 9). It is worth to emphasize  that
the classical surgery formulae (see e.g. \cite{Ok}, or \cite{BN}) are valid for the special
`natural' line bundles  and under special analytic conditions, and it was not  clear at all if any extension to the general case might exists and/or how to define the correction terms in such  general situations. In the present note
 this is solved via the above  stabilized dimensions of the images of Abel maps
(without any required restriction). Furthermore,
 under the special analytic conditions of the old surgery formulae,  they are identified with
the classical correction terms.

Starting from section 7 we develop the `duality picture' between divisors and
differential forms. This not only describes the Abel map and its tangent map, but it gives a computational tool
in concrete examples as well.
The invariants of stable case in language of differential forms are described in section 8.
The general non--stable case is analyzed in section 10.

When a concrete basis of
$H^0(\tX\setminus E, \Omega^2_{\tX})/H^0(\tX,\Omega^2_{\tX})$ (dual to $H^1(\calO_{\tX})$) can be
explicitly determined, the Abel map also becomes more transparent, and several of the above listed
problems have precise (sometimes even combinatorial) solutions. This is exemplified in the case of
superisolated (section 11) and weighted homogeneous (section   12) singularities.
Some additional properties in the Gorenstein situation are also listed.

\vspace{2mm}

In the sequel $\#A$ denotes the cardinality of the finite set $A$.

\section{Preliminaries}\label{s:prel}

In this section we review some basic facts
about topological and analytical invariants of surface singularities,
and we introduce the  needed notations as well.

\subsection{The resolution}\label{ss:notation}
Let $(X,o)$ be the germ of a complex analytic normal surface singularity,
 and let us fix  a good resolution  $\phi:\widetilde{X}\to X$ of $(X,o)$.
We denote the exceptional curve $\phi^{-1}(0)$ by $E$, and let $\cup_{v\in\calv}E_v$ be
its irreducible components. Set also $E_I:=\sum_{v\in I}E_v$ for any subset $I\subset \calv$.
The support of a cycle $l=\sum n_vE_v$ is defined as  $|l|=\cup_{n_v\not=0}E_v$.
For more details see \cite{Lauferbook,trieste,NCL,Nfive,LPhd}.
\subsection{Topological invariants}\label{ss:topol}
Let $\Gamma$ be the dual resolution graph
associated with $\phi$;  it  is a connected graph.
Then $M:=\partial \widetilde{X}$ can be identified with the link of $(X,o)$, it is also
an oriented  plumbed 3--manifold associated with $\Gamma$.
It is known that $(X,o)$ locally is homeomorphic with the real cone over $M$,
and $M$ contains the same information as $\Gamma$.
We will assume that  {\it $M$ is a rational homology sphere},
or, equivalently,  $\Gamma$ is a tree and all genus
decorations of $\Gamma$ are zero. We use the same
notation $\mathcal{V}$ for the set of vertices, and $\delta_v$ for the valency of a vertex $v$.

$L:=H_2(\widetilde{X},\mathbb{Z})$, endowed
with a negative definite intersection form  $(\,,\,)$, is a lattice. It is
freely generated by the classes of 2--spheres $\{E_v\}_{v\in\mathcal{V}}$.
 The dual lattice $L':=H^2(\widetilde{X},\mathbb{Z})$ is generated
by the (anti)dual classes $\{E^*_v\}_{v\in\mathcal{V}}$ defined
by $(E^{*}_{v},E_{w})=-\delta_{vw}$ (where $\delta_{vw}$ stays for the  Kronecker symbol).
The intersection form embeds $L$ into $L'$. Then $H_1(M,\mathbb{Z})\simeq L'/L$, and it is
abridged by $H$.
Usually one  identifies $L'$ with those rational cycles $l'\in L\otimes \Q$ for which
$(l',L)\in\Z$, or, $L'={\rm Hom}_\Z(L,\Z)$.

There is a natural (partial) ordering of $L'$ and $L$: we write $l_1'\geq l_2'$ if
$l_1'-l_2'=\sum _v r_vE_v$ with all $r_v\geq 0$. We set $L_{\geq 0}=\{l\in L\,:\, l\geq 0\}$ and
$L_{>0}=L_{\geq 0}\setminus \{0\}$.

Each class $h\in H=L'/L$ has a unique representative $r_h=\sum_vr_vE_v\in L'$ in the semi-open cube
(i.e. each $r_v\in \bQ\cap [0,1)$), such that its class  $[r_h]$ is $h$.

All the $E_v$--coordinates of any $E^*_u$ are strict positive.
We define the Lipman cone as $\calS':=\{l'\in L'\,:\, (l', E_v)\leq 0 \ \mbox{for all $v$}\}$.
As a monoid it is generated over $\bZ_{\geq 0}$ by $\{E^*_v\}_v$.

The  \emph{multivariable topological Poincar\'e series} is the
Taylor expansion $Z(\mathbf{t})=\sum_{l'} z(l')\mathbf{t}^{l'}$ at the  origin of the rational function
\begin{equation}\label{eq:ratfunc}
{\mathcal Z}(\mathbf{t})=\prod_{v\in \mathcal{V}} (1-\mathbf{t}^{E^*_v})^{\delta_v-2},
\end{equation}
where
$\mathbf{t}^{l'}:=\prod_{v\in \mathcal{V}}t_v^{l'_v}$  for any $l'=\sum _{v\in \mathcal{V}}l'_vE_v\in L'$.
By definition, $Z(\bt)$  is supported on $\calS'$.
It has a natural decomposition $Z(\mathbf{t})=\sum_{h\in H}Z_h(\mathbf{t})$, where $Z_h(\mathbf{t})=\sum_{[l']=h}z(l')\mathbf{t}^{l'}$.
(Though the exponents of $\bt^{l'}$ might be rational, that is,
$Z(\bt)\in \Z[[t_1^{1/d},\ldots , t_{|\calv|}^{1/d}]]$, where $d=\det(\Gamma)$,  the right hand side of
(\ref{eq:ratfunc}) still will be called `rational function', and $\sum_{l'}z(l')\bt^{l'}$ a `series'.)

\subsection{Analytic invariants}\label{ss:analinv}
In this manuscript we focus mainly on the structure of the Picard group  and the holomorphic line bundles
of $\widetilde{X}$.
 The group ${\rm Pic}(\widetilde{X}):=
 H^1(\widetilde{X}, \calO_{\widetilde{X}}^*)$
of  isomorphism classes of {\it holomorphic } line bundles on $\widetilde{X}$ appears in the exact sequence
\begin{equation}\label{eq:PIC}
0\to {\rm Pic}^0(\widetilde{X})\to {\rm Pic}(\widetilde{X})\stackrel{c_1}
{\longrightarrow} L'\to 0, \end{equation}
where  $c_1$ denotes the first Chern class. Here
$ {\rm Pic}^0(\widetilde{X})=H^1(\widetilde{X},\calO_{\widetilde{X}})\simeq
\C^{p_g}$, where $p_g$ is the {\it geometric genus} of
$(X,o)$. $(X,o)$ is called {\it rational} if $p_g(X,o)=0$.
 Artin in \cite{Artin62,Artin66} characterized rationality topologically
via the graphs; such graphs are called `rational'. By this criterion, $\Gamma$
is rational if and only if $\chi(l)\geq 1$ for any effective non--zero cycle $l\in L_{>0}$.
Here $\chi(l)=-(l,l-Z_K)/2$, where $Z_K\in L'$ is the (anti)canonical cycle
identified by adjunction formulae
$(-Z_K+E_v,E_v)+2=0$ for all $v$.

The epimorphism
$c_1$ admits a unique group homomorphism section $l'\mapsto s(l')\in {\rm Pic}(\widetilde{X})$,
 which extends the natural
section $l\mapsto \calO_{\widetilde{X}}(l)$ valid for integral cycles $l\in L$, and
such that $c_1(s(l'))=l'$  \cite{trieste,OkumaRat}.
We call $s(l')$ the  {\it natural line bundle} on $\widetilde{X}$ with Chern class $l'$.
By  its definition, $\calL$ is natural if and only if some power $\calL^{\otimes n}$
of it has the form $\calO_{\tX}(l)$ for some $l\in L$.

Natural line bundles appear in the presence of coverings as well. Indeed, let $\pi:(X_{ab},o)\to (X,o)$  be
the {\it universal abelian covering} of $(X,o)$ (associated with the homomorphism
$\pi_1(M)\to H_1(M)=H$) and let $\widetilde{\pi}:\widetilde{X_{ab}}\to \widetilde{X}$ be the
(normalized) pullback of $\pi$ by the resolution $\phi:\widetilde{X}\to X$. Then the Galois group $H$ acts on
$\widetilde{\pi}_*(\calO_{X_{ab}})$, whose eigensheaves are
$\widetilde{\pi}_*(\calO_{X_{ab}})=\oplus_{h\in H}s(-r_h)$ \cite{trieste}.
Hence, in this way, one recovers all the natural  line bundles with Chern classes in the open--closed cube.
Those with arbitrary Chern clasess satisfy $s(-l-r_h)=\calO_{\widetilde{X}}(-l)\otimes s(-r_h)$ for  certain $l\in L$.

In the sequel we write uniformly $\calO_{\widetilde{X}}(l')$ for $s(l')$.

Since $\calO_{\widetilde{X_{ab}}}$ might have  only cyclic quotient singularities,
  $p_g(X_{ab},o)=h^1(\calO_{\widetilde{X_{ab}}})$ and
$H^1(\calO_{\widetilde{X_{ab}}})=\oplus_h H^1(\widetilde{X},\calO_{\widetilde{X}}(-r_h))$.
The dimensions $p_g(X_{ab},o)_h:=h^1(\widetilde{X},\calO_{\widetilde{X}}(-r_h)) $ ($h\in H$)
are called the {\it equivariant geometric genera} of $(X,o)$. Clearly,
$\sum_h p_g(X_{ab},o)_h=p_g(X_{ab},o)$ and $p_g(X_{ab},o)_0=p_g(X,o)$.

\bekezdes
Similarly, if $Z\in L_{>0}$ is an effective non--zero integral cycle supported by $E$, and $\calO_Z^*$ denotes
the sheaf of units of $\calO_Z$, then ${\rm Pic}(Z)=H^1(Z,\calO_Z^*)$ is  the group of isomorphism classes
of invertible sheaves on $Z$. It appears in the exact sequence
  \begin{equation}\label{eq:PICZ}
0\to {\rm Pic}^0(Z)\to {\rm Pic}(Z)\stackrel{c_1}
{\longrightarrow} L'(|Z|)\to 0, \end{equation}
where ${\rm Pic}^0(Z)=H^1(Z,\calO_Z)$.
Here and in the sequel, $L(|Z|)$ denotes the sublattice of $L$ generated by
the  base element $E_v\subset |Z|$, and $L'(|Z|)$ is its dual lattice.

If $Z_2\geq Z_1$ then there are natural restriction maps (for simplicity we denote all of them by
the same symbol $r$), ${\rm Pic}(\widetilde{X})\to {\rm Pic}(Z_2)\to {\rm Pic}(Z_1)$. Similar restrictions are defined at  ${\rm Pic}^0$ level too.
These restrictions are homomorphisms of the exact sequences  (\ref{eq:PIC}) and (\ref{eq:PICZ}):

 \begin{equation}\label{eq:diagr2} \
\end{equation}

\vspace*{-2cm}

\begin{picture}(400,80)(-100,-10)
\put(0,40){\makebox(0,0)[l]{$
0\to {\rm Pic}^0(\widetilde{X})\to {\rm Pic}(\widetilde{X})\stackrel{c_1}
{\longrightarrow} \ \ L'\ \ \ \to 0
$}}
 \put(0,5){\makebox(0,0)[l]{$
 0\to {\rm Pic}^0(Z)\to {\rm Pic}(Z)\stackrel{c_1}
{\longrightarrow} L'(|Z|)\to 0
$}}

\put(35,22){\makebox(0,0){$\downarrow$}}\put(35,20){\makebox(0,0){$\downarrow$}}
\put(85,22){\makebox(0,0){$\downarrow$}}\put(85,20){\makebox(0,0){$\downarrow$}}
\put(145,22){\makebox(0,0){$R$}}
\put(93,22){\makebox(0,0){$r$}}
\put(135,22){\makebox(0,0){$\downarrow$}}\put(135,20){\makebox(0,0){$\downarrow$}}
\end{picture}

Furthermore, for any $l'\in L'$ we define a line bundle in $\pic(Z)$ by
$r(s(l'))={\mathcal O}_{\widetilde{X}}(l')|_{Z}$,
and we call them  {\it restricted natural line bundles } on $Z$.
They satisfies $c_1( r(s(l')))=R(l')$.

We also use the notations ${\rm Pic}^{l'}(\widetilde{X}):=c_1^{-1}(l')
\subset {\rm Pic}(\widetilde{X})$ and
${\rm Pic}^{R(l')}(Z):=c_1^{-1}(R(l'))\subset{\rm Pic}(Z)$
respectively. Multiplications by $\calO_{\widetilde{X}}(-l')$ and  by
$\calO_{\tX}(-l')|_Z$ provide natural (affine  space) isomorphisms
${\rm Pic}^{l'}(\widetilde{X})\to {\rm Pic}^0(\widetilde{X})$ and
${\rm Pic}^{R(l')}(Z)\to {\rm Pic}^0(Z)$.

Here an important warning is appropriate. If $\tX'$ is a small connected neighbourhood
of some exceptional curves $\cup_{v\in\calv'}E_v$, $\calv'\subset \calv$,
then similarly as for $\tX$, but now starting with the invariants of $\tX'$, one can define the natural line bundles $\calO_{\tX'}(l')$ for any $l'\in L'(\calv')$. However, for $l'\in L'$,
in general, if $\calv'\not=\calv$ then   $\calO_{\tX}(l')|_{\tX'}\not=\calO_{\tX'}(R(l'))$, though both line bundles have the same Chern class (here $R$ is the restriction).
That is, $\calO_{\tX}(l')|_{\tX'}$ in general is not the intrinsic  natural line bundle of $\tX'$.

Similarly, for any cycle $Z$ one can define the (intrinsic) natural line bundles of $Z$ by group
section of (\ref{eq:PICZ}) by similar properties as the natural line bundles of $\tX$ are defined.
If $|Z|=E$ then they agree with the restrictions  $\calO_{\tX}(l')|_Z$.
However, if $|Z|\not=E$ then it can happen that $\calO_{\tX}(l')|_Z$ is not natural on $Z$.
This explains the  use of the terminology `restricted natural
line bundle' for $\calO_{\tX}(l')|_Z$: they are always restriction from the $\tX$--level.
In order to simplify the notations we will also write $\calO_Z(l'):=\calO_{\tX}(l')|_Z$, $l'\in L'$.

For any line bundle $\calL\in\pic(\tX)$ we also write $\calL(l'):=\calL\otimes \calO_{\tX}(l')$.

\bekezdes One of our  main interest is to understand the stratification $\{ \calL\in {\rm Pic}(\widetilde{X})
\, :\, h^1(\calL)=k\}_{k\in \Z_{\geq 0}}$
of ${\rm Pic}(\widetilde{X})$. In the literature about $h^1(\calL)$ --- for arbitrary $\calL$ ---
very little is known. However, about the natural line bundles (of some special analytic structures $(X,o)$)
recently several results were proved, see e.g.
 \cite{CDGPs,CDGEq,NPS,NJEMS,NCL}. Since some of these facts
are used in  several examples and play key role in the general presentation
 we review  them in the next subsection.

\bekezdes\label{bek:AMS}
{\bf The analytic multivariable Poincar\'e series} is defined as follows \cite{NCL}, see also \cite{CDGPs,CDGEq}.
For every  $\calL\in \pic(\widetilde{X})$  (respectively, for $Z\geq E$ and $\calL\in \pic(Z)$)
one defines
\begin{equation*}\label{eq:pZL}
p_{\calL}:= \sum_{I\subset \calv} \, (-1)^{|I|+1}\,\dim\ \frac{H^0(\widetilde{X},\calL)}{H^0(\widetilde{X},\calL(-E_I))}
\ \mbox{and} \
p_{Z,\calL}:= \sum_{I\subset \calv} \, (-1)^{|I|+1}\,\dim\ \frac{H^0(Z,\calL)}{H^0(Z-E_I,\calL(-E_I))}.
\end{equation*}
For $Z\gg 0$ and $\calL\in \pic(\widetilde{X})$ one has $p_{\calL}=p_{Z,\calL|_Z}$.
If $(c_1(\calL),E_v)<0$ for some $v\in\calv$, then
$H^0 (\widetilde{X},\calL(-E_{I\cup v}))\to H^0 (\widetilde{X},\calL(-E_{I}))$ is an isomorphism
 for any $I\not\ni v$ (and similar isomorphism holds  for any $Z\geq E$), hence
 \begin{equation}\label{eq:cegyL}
 p_{\calL}= p_{Z,\calL}=0 \ \ \mbox{whenever} \ \  c_1(\calL)\not\in -\calS'.
 \end{equation}
At the level of $\widetilde{X}$ one defines a
multivariable series as $P_{\calL}(\bt):= \sum_{l'\in L'} \, p_{\calL(-l')} \bt^{l'}$.
It also has an $H$--decomposition $\sum_h P_{\calL,h}$,
 $P_{\calL,h}=\sum_{[l']=h} p_{\calL(-l')}\bt^{l'}$,
 according to the classes $[l']\in H$ of the exponents of $\bt^{l'}$.
By (\ref{eq:cegyL}) it is supported on $c_1(\calL)+\calS'$. We write $P(\bt):=P_{\calO_{\tX}}(\bt)=\sum_{l'}p_{\calO_{\tX}(-l')}\bt^{l'}$.

The first cohomology of the natural line bundles and the series $P(\bt)$ are linked by the following identity
proved in  \cite{NCL}:
 \begin{equation}\label{eq:HPol2}
   h^1 (\widetilde{X}, \calO(-r_h-l)) = -
\sum_{a \in L,\, a \ngeq 0}
    p_{\calO(-r_h-l-a)} +p_g(X_{ab},o)_h + \chi(l)-(l,r_h).
\end{equation}

\bekezdes
Recently there is an intense activity in the comparison of the analytic invariant $P(\bt)$
 and the topological $Z(\bt)$ (their coincidence imply e.g. the so-called Seiberg--Witten Invariant Conjecture
 \cite{NJEMS,NCL}).
For the equality of $P(\bt)$ and $Z(\bt)$ for certain families singularities
(rational, weighted homogeneous, splice quotient) see e.g.
 \cite{CDGPs,CDGEq,NPS,NCL} and the references therein.




We emphasize that in the previous   results in the literature the main goal mostly
was to characterize for special
(`nice') analytic structures the sheaf--theoretical invariants $h^1(\calL)$ topologically,  and
those methods were applicable only for natural line bundles $\calL$.
In the present note our goal is to treat $h^1(\calL)$ for any line bundle and for
any analytic structure.

\subsection{Notations.}\label{not:min}
In the body of the article we will present several examples. In them we will
use the following standard notations.
 We will write $Z_{min}\in L$ for the  {\it minimal} (or fundamental) cycle of Artin, which is
the minimal non--zero cycle of $\calS'\cap L$ \cite{Artin62,Artin66}. Yau's {\it maximal ideal cycle}
$Z_{max}\in L$ is the  divisorial part of the pullback of the maximal ideal $\m_{X,o}\subset \calO_{X,o}$, i.e.
 $\phi^*{\m_{X,o}}\cdot \calO_{\widetilde{X}}=\calO_{\widetilde{X}}(-Z_{max})\cdot \cali$,
where $\cali$ is an ideal sheaf with 0--dimensional support \cite{Yau1}. In general $Z_{min}\leq Z_{max}$.
$Z_{min}$ can be found by {\it Laufer's algorithm} \cite{Laufer72}. This algorithm also shows that
$h^0(\calO_{Z_{min}})=1$, hence  $h^1(\calO_{Z_{min}})=1-\chi(Z_{min})$ is topological.

\section{Effective Cartier divisors} \label{s:efcart}
\subsection{} \label{ss:efcart}
For any $Z\in L_{>0}$ let
$\eca(Z)$  be the space of (analytic) effective Cartier divisors on 
$Z$. Their supports are zero--dimensional in $E$.
Taking the class of a Cartier divisor provides  the {\it Abel map}
$c:\eca(Z)\to \pic(Z)$.
Let
$\eca^{l'}(Z)$ be the set of effective Cartier divisors with
Chern class $l'\in L'(|Z|)$, that is,
$\eca^{l'}(Z):=c^{-1}(\pic^{l'}(Z))$.
Sometimes we denote the restriction of $c$ by  $c^{l'}:\eca^{l'}(Z)
\to \pic^{l'}(Z)$, $l'\in L'(|Z|)$. It is also convenient to use the simplified notation
$\eca^{l'}(Z):=\eca^{R(l')}(Z)$ and $\pic^{l'}(Z):= \pic^{R(l')}(Z)$ for any $l'\in L'$.

For any $Z_2\geq Z_1>0$ (and $l'\in L'$) one has the commutative diagram
\begin{equation}\label{eq:diagr}
\begin{picture}(200,45)(0,0)
\put(50,37){\makebox(0,0)[l]{$
\eca^{l'}(Z_2)\,\longrightarrow \, \pic^{l'}(Z_2)$}}
\put(50,8){\makebox(0,0)[l]{$
\eca^{l'}(Z_1)\,\longrightarrow \, \pic^{l'}(Z_1)$}}
\put(70,22){\makebox(0,0){$\downarrow$}}
\put(135,22){\makebox(0,0){$\downarrow$}}
\end{picture}
\end{equation}

Regarding the existence of $\eca(Z)$ and the Abel map we note the following.
First, by a theorem of Artin \cite[3.8]{Artin69}, there exists an affine algebraic variety $Y$ and a point $y\in Y$ such that $(Y,y)$ and $(X,o)$ have isomorphic formal completions. Then, according to Hironaka \cite{Hironaka65},
$(Y,y)$ and $(X,o)$ are analytically isomorphic. In particular, we can regard $Z$
as a projective  algebraic scheme, in which case $\eca^{l'}(Z)$ together with the algebraic Abel map,
as part of the general theory,
was constructed by Grothendieck \cite{Groth62}, see e.g. the article of
Kleiman \cite{Kleiman2013} with several comments and citations and
the book of Mumford for curves on algebraic surfaces \cite{MumfordCurves}.
In particular,
$$c:\eca(Z)\to \pic(Z) \ \mbox{ is algebraic}.$$
(For concrete charts of $\eca^{l'}(Z)$ see e.g. the proof of theorem \ref{th:smooth} and
for the Abel map in concrete charts see section \ref{s:ADIFFFORMS}.)
Though these spaces are identified   by the general theory,
in the body of this  note we  verify directly
several properties of them
in order to illuminate the peculiarities of the present situation,  e.g.
we discuss the smoothness and the dimension
of $\eca^{l'}(Z)$ and the structure of the fibers of the
Abel map: the related numerical invariants will be crucial in the further
 discussions. Doing this we develop  several special
properties of the Abel map in the language of invariants of normal
surface singularities; these connections will be exploited deeply.

We write $\eca(\tX)$ for  the  {\it set}
 of effective Cartier divisors on $\tX$.

\bekezdes
Let us fix  $Z\in L$, $Z>0$.
As usual, we say that $\calL\in \pic^{l'}(Z)$ has no fixed components if
\begin{equation}\label{eq:H_0}
H^0(Z,\calL)_{\reg}:=H^0(Z,\calL)\setminus \bigcup_{E_v\subset |Z|} H^0(Z-E_v, \calL(-E_v))
\end{equation}
is non--empty. Note that $H^0(Z,\calL)$ is a module over the algebra
$H^0(\calO_Z)$, hence one has a natural action of $H^0(\calO_Z^*)$ on
$H^0(Z, \calL)_{\reg}$. For the next lemma see e.g. \cite[\S 3]{Kl}.

\begin{lemma}\label{lem:H_0} $\calL\in \pic^{l'}(Z)$  is in the image of
 $c^{l'}:\eca^{l'}(Z) \to \pic^{l'}(Z)$  if and only if
$H^0(Z,\calL)_{\reg}\not=\emptyset$. In this case, $c^{-1}(\calL)=H^0(Z,\calL)_{\reg}/H^0(\calO_Z^*)$.
\end{lemma}
In the next discussion we assume $Z\geq E$ basically  imposed by the
easement of the presentation; everything can be adopted  for any
$Z>0$, see e.g.  \ref{rem:Zsupport} or \ref{ss:Lauferseq}.

Note that
$H^0(Z,\calL)_{\reg}\not=\emptyset \
 \Rightarrow \ H^0(\calL|_{E_v})\not=0 \ \forall \, v
\Rightarrow \ (l',E_v)\geq 0 \ \forall \, v\ \Rightarrow\  l'\in -\calS'$.
Conversely, if $l'=-\sum_vm_vE^*_v\in -\calS'$ (for certain $m_v\in\bZ_{\geq 0}$), and
$l'\not=0$, then one can construct
for each $E_v$ cuts in $\widetilde{X}$
intersecting $E_v$ in a generic point and having with  it intersection multiplicity $m_v$. Since $l'\not=0$ their collection is nonempty, and it
provides elements in $\eca^{l'}(\tX)$ and $\eca^{l'}(Z)$ respectively (the second one by restriction).
However, this collection is empty whenever $l'=0$, hence this special case needs slightly more attention. By definition we declare that $\eca^0(Z)$ is a space consisting of a point (what we can call the
`empty divisor'), $\eca^0(Z)=\{\emptyset\}$, and $c^0: \eca^0(Z)\to \pic^0(Z)$ is defined as
$c^0(\emptyset)=\calO_Z$. Since for $l'=0$
any section from $H^0(Z,\calL)_{\reg}$ trivializes $\calL$, one has:
$$H^0(Z,\calL)_{\reg}\not=\emptyset \ \Leftrightarrow \ \calL=\calO_Z \ \Leftrightarrow \
\calL\in \im (c^0) \ \ \ \ \ \ (l'=0).$$
Therefore, the above discussions combined provide
\begin{equation}\label{eq:empty}
\eca^{l'}(Z)\not =\emptyset \ \ \Leftrightarrow \ \ l'\in -\calS'.
\end{equation}
The action of  $H^0(\mathcal{O}^{*}_{Z})$ can be analysed quite explicitly.
 Note that from the exact sequence
 \begin{equation}
 0 \to H^0(\mathcal{O}_{Z-E}(-E)) \to H^0(\mathcal{O}_{Z}) \stackrel{r_E}{\longrightarrow}
 H^0(\mathcal{O}_{E})=\bC\to 0
\end{equation}
one gets that
  $ H^0(\mathcal{O}^{*}_{Z}) = r_E^{-1}(\bC^*)=
  H^0(\mathcal{O}_{Z}) \setminus H^0(\mathcal{O}_{Z-E}(-E))$.
 In particular, $H^0(\calO^*_Z)$,  as algebraic variety, 
has the dimension of the vector space $H^0(\cO_Z)$,
$\bP H^0(\calO^*_Z)$ as algebraic variety is isomorphic with $H^0(\cO_{Z-E}(-E))$,
 and $H^0(Z,\calL)_{\reg}/H^0(\calO^*_Z)=
 \bP H^0(Z,\calL)_{\reg}/\bP H^0(\calO^*_Z)$.

 \begin{lemma}\label{lem:free} Assume that
 $H^0 (Z,\calL)_{\reg}\not=\emptyset$. Then

 (a) the action of $H^0(\calO_Z^*)$ on $H^0(Z,\calL)_{\reg}$ is algebraic, free and proper;

 (b) 
 $\bP H^0(Z,\calL)_{\reg}$ over $\bP H^0(Z,\calL)_{\reg}/\bP H^0(\calO_Z^*)$
 is a principal affine  bundle.

 \noindent
 Hence, the fiber $c^{-1}(\calL)$, $\calL\in \im (c^{l'})$,
 is an irreducible quasiprojective variety of  dimension
 \begin{equation}\label{eq:dimfiber}
 h^0(Z,\calL)-h^0(\calO_Z)=
 (l',Z)+h^1(Z,\calL)-h^1(\calO_Z).
 \end{equation}
 \end{lemma}
\begin{proof} For $s\in H^0(Z,\calL)_{\reg}$ the multiplication by $s$,
$\calO_Z\stackrel{\cdot s}{\longrightarrow} \calL$, is injective, hence induces injections $H^0(\calO_Z)\stackrel{\cdot s}{\longrightarrow} H^0(\calL)$ and
 $H^0(\calO_Z^*)\stackrel{\cdot s}{\longrightarrow} H^0(\calL)_{\reg}$.
 Hence the action is free.
Next we prove that the action of  $\bP H^0(\calO_Z^*)$ on $\bP H^0(Z,\calL)_{\reg}$ is proper.

Introduce hermitian metrics in both $H^0(\calO_Z)$ and $H^0(Z,\calL)$.
Write $H^0:=H^0(\calO_{Z-E}(-E))$ in $H^0(\calO_Z)$ and choose  $h^\perp$ with
 $H^0(\calO_Z)=H^0\oplus \C\langle h^\perp\rangle$.
Set also $B:=\cap_v H^0(Z-E_v,\calL(-E_v))\subset H^0(Z,\calL)$ and let $B^\perp$ be its unitary complement in $H^0(Z,\calL)$. Note that $H^0(Z,\calL)\setminus B$ is also stable
with respect to the action of $H^0(\calO^*_Z)=B\oplus \C\langle h^\perp\rangle\setminus B\oplus 0$.
Since $H^0(Z,\calL)_{\reg}$ is open in $H^0(Z,\calL)\setminus B$, it is enough to show that
 $H^0(\calO^*_Z)$ acts properly on $H^0(Z,\calL)\setminus B$. Fix $K$ compact in
 $H^0(Z,\calL)\setminus B$ and let $K'$ be its lift to the unit sphere of $H^0(Z,\calL)$.
 We need to show that  if $h=h^0+h^\perp \in H^0\oplus \C\langle h^\perp \rangle$  and $|h^0|\to \infty$, and $k\in K'$,
 then the components $(hk)_1+(hk)_2\in B\perp B^\perp $ of $hk$ satisfy
 $|(hk)_1|/|(hk)_2|\to \infty$.
 For this note the following facts.

 First, $H^0\cdot H^0(Z,\calL)\subset B$, hence $(h^0k)_2=0$.
 Next, since $K'$ is compact, $|(h^\perp k)_1|$ and $|(h^\perp k)_2|$ are bounded from above. Finally,
 since $h^0k\not=0$, for any $h^0$ in the unit sphere, the set $\{|h^0k|\}_k$ is bounded from  below by a positive number. Hence, whenever $|h^0|\to \infty$ one also has
 $$|(hk)_1|/|(hk)_2|=| (h^\perp k)_1+|h^0|\cdot (\frac{h^0}{|h^0|}\cdot k)|/
  |(h^\perp k)_2|\to \infty\ .$$
 (a) implies (b) (since $\bP H^0(\calO_Z^*)\simeq H^0$ is an affine space) and
the equality in (\ref{eq:dimfiber}) follows from Riemann--Roch formula.
\end{proof}

 \begin{example}\label{ex:rateca}
 Assume that $(X,o)$ is rational, and $l'\in-\calS'$. Then $\pic^{l'}(Z)=0$, hence if
 $c_1(\calL)=l'$ then $\calL=\calO(l')$. Furthermore,
 $\calL$ is basepoint free \cite[Th. 12.1]{Lipman}. Thus  $\eca^{l'}(Z)= H^0(Z,\calL)_{\reg}/H^0(\calO^*_Z)$
and since the action of $H^0(\cO_Z^*)$ is  free (cf. \ref{lem:free}),
$\eca^{l'}(Z)$  is smooth. Since $h^1(Z,\calL)=h^1(\calO_Z)=0$ (cf. \cite{Lipman,Nfive}),
 its dimension is $(l',Z)$ (use (\ref{eq:dimfiber})). Furthermore, its topological Euler characteristic is
 $\chi_{top}(\eca^{l'}(Z))=\chi_{top}(\bP H^0(Z,\calL)_{\reg})$, which is the coefficient $z(-l')$ of
 the multivariable series  $Z({\bf t})$ by \cite{CDGEq,NPS,NCL}.
 \end{example}
These facts  generalize as follows.
 \begin{theorem}\label{th:smooth} If $l'\in-\calS'$ then the following facts hold.

  (1)  $\eca^{l'}(Z)$ is a smooth complex (irreducible) variety of dimension $(l',Z)$.

  (2) The topological Euler characteristic of $\eca^{l'}(Z)$ is $z(-l')$.
 In fact, the natural restriction  $r:\eca^{l'}(Z)\to \eca^{l'}(E)$ is a
  locally trivial  fiber bundle with fiber isomorphic to an affine space. Hence,
 the homotopy type of $\eca^{l'}(Z)$ is independent of the choice of $Z$ and
 it depends only on the topology of $(X,o)$.

 (3) $r:\eca^{l'}(Z_2)\to \eca^{l'}(Z_1)$ is surjective for any $Z_2\geq Z_1$.
 \end{theorem}
 \begin{proof}
 As we already said in \ref{ss:efcart},
 $\eca^{l'}(Z)$ is an algebraic variety, cf. \cite{Groth62,Kleiman2013}.
 We need to construct in the neighbourhood of each Cartier divisor  a smooth chart.

 First assume that $Z=E$. Then $\eca^{l'}(E)$ is independent of the self-intersections
 $E^2_v$, hence (keeping the analytic type of $E$, but) modifying the self--intersections into
 very negative integers, we can assume that the singularity is rational. In this modified case,
 $\eca^{l'}(E)=\bP( H^0(E,\calO(l'))_{\reg})$, see Example \ref{ex:rateca}. Note that $H^0(E,\calO(l'))_{\reg}$ is also independent of the
 self--intersection numbers, hence, in any case,   $\eca^{l'}(E)=\bP ( H^0(E,\calO(l'))_{\reg})$.
 In particular, $\eca^{l'}(E)$ is smooth, irreducible and with the required dimension and Euler characteristic, cf.
 Example \ref{ex:rateca}.

 Let us provide some local charts of $\eca^{l'}(E)$.
 Fix $D\in \eca^{l'}(E)$ with support $\{p_i\}_i\subset E$.

 If $p_i\in E_v$ is a smooth point of $E$, then there exists a local
 neighbourhood $U_i$ of $p_i$ in $\widetilde{X}$ with local coordinates
 $(x,y)$ such that $\{x=0\}=E\cap U_i$ and $D$ in $U_i$ is represented by the
 local Cartier equation $\{y^m\}$ for some $m\in\bZ_{>0}$. Then a local neighbourhood $\calU_i(E)$
 of the divisor $\{y^m\}$ in $\eca^{-mE_v^*}(E)$ is given by local
 Cartier divisors $\{y^m+f(y)\}$, where $f\in \calO(E\cap U_i)$
 is a small perturbation of the zero function,
 modulo the multiplicative action of $\calO^*(E\cap U_i)$.
 Multiplying $y^m$ by $1+a_ky^k$ we get that
 perturbation of type $y^m+\sum_{k\geq 0} a_{k}y^{k+m}$  constitute
 the orbit of $y^m$ (or, differently said, $\sum_{k\geq 0} a_{k}y^{k+m}$
 is the tangent space of the orbit). Therefore,
 the smooth transversal slice to this orbit
 $(a_i)_{0\leq i<m}\mapsto \{y^m+\sum_ia_iy^i\}$  $(|a_i|\ll 1)$
 provides a smooth chart   $\calU_i (E)$ of dimension  $m=(-mE_v^*,E)$.
Here, $ -mE^*_v$ is the local contribution in the Chern class $l'$.

 Similarly, if $p_i=E_u\cap E_v$, then there exists a neighbourhood $U_i$
 of $p_i$ in $\widetilde{X}$ with local coordinates $(x,y)$ such that
 $\{x=0\}=U_i\cap E_v$ and  $\{y=0\}=U_i\cap E_u$, and $D$ in $U_i$ is
 represented by $\{x^n+y^m\}$, $n,m\in\Z_{>0}$. Then,  a local neighbourhood
 $\calU_i(E)$ of $x^n+y^m$ in $\eca^{-mE_v^*-nE_u^*}(E)$ is given by
  $\{x^n+y^m+a_0+\sum_{i\geq 1}a_ix^i+\sum_{i\geq 1}b_iy^i\}$ modulo the action of
  $\calO^*(E\cap U_i)$. The orbit of this action 
  at $ x^n+y^m$ is
   $\{x^n+y^m+\sum_{i>n }a_ix^i+\sum_{i>m }b_iy^i+\lambda(x^n+y^m)\}$, it is smooth.
   A possible smooth slice of it is
 $\{x^n+y^m+a_0+\sum_{i=1}^{n}a_ix^i+\sum_{i=1}^{m}b_iy^i\}/\{a_n+b_m=0\}$, which is
  of dimension $(-mE_v^*-nE_u^*,E)$ (the local contribution into $(l',E)$).

 Products of type $\calU(D)=\prod_i \calU_i(E)$ constitute a local neighbourhood of $D$ in
 $\eca^{l'}(E)$.

 Consider now an arbitrary $Z\geq E$ and the restriction $r:\eca^{l'}(Z)\to \eca^{l'}(E)$.
 We show that $\eca^{l'}(Z)$ can be covered by open sets of type $r^{-1}(\prod_i\calU_i(E))=
 \prod_ir^{-1}_i(\calU_i(E))$,
 where $r_i$ is either the restriction
 $\eca^{-mE^*_v}(Z)\to \eca^{-mE^*_v}(E)$ or $\eca^{-mE^*_v-nE^*_u}(Z)
 \to \eca^{-mE^*_v-nE^*_u}(E)$,
 and each $r^{-1}_i(\calU_i(E))$ is a product of $\calU_i(E)$ and an affine space.

 Indeed, assume first that
 $p_i$ is a smooth point of $E$ as above, $p_i\in E_v$, and let $N\geq 1$
 be the multiplicity of $Z$ along $E_v$. Then in $U_i$ the local equation of $Z$ is
 $x^N$ and let us fix a Cartier divisors in  $r^{-1}(\calU_i(E))$
  whose restriction is $y^m$,  represented by $f:=y^m+xg(x,y)$   for some
 $g\in \calO(U_i)/(x^{N-1})$,  modulo $\calO^*(U_i)/(x^N)$.
 Multiplication $f(1+a_iy^ix^{N-1})\equiv f+a_iy^{m+i}x^{N-1}$ shows that
$f+ y^{m}x^{N-1}\calO(U_i)$  (mod $(x^N)$) is in the orbit.
Using this fact, and  multiplication by
 $1+a_iy^ix^{N-2}$ one shows that $f+ y^{m}x^{N-2}\calO(U_i)$ (mod $(x^N)$) is
 also in the orbit. By induction, we get that the orbit is
 $f+ y^{m}\calO(U_i)$ (mod $(x^N)$), and it is smooth.
 A transversal smooth cut can be parametrized by the
 chart $\{y^m+\sum_{i<N,\, j<m} a_{ij}x^iy^j\}$, which has dimension $(-mE_v^*,Z)=mN$.
 For $i>0$ the variables $a_{ij}$ can be chosen as affine coordinates.

 More conceptually, in this case, multiplication of $f$ by $1+h$ gives
 $f+fh $ (mod $(x^N)$), hence the orbit is identified with $f+\mbox{ideal}(f,x^N)$, which
 has a smooth section whose dimension is the codimension of  $\mbox{ideal}(f,x^N)$, that is, the
  intersection multiplicity $(f,x^N)_{p_i}=mN$.

 Similar  chart can be found in the case of $p_i=E_u\cap E_v$ as well.
 Let us use the previous notations, let us fix a divisor $f=x^n+y^m +xyg(x,y)$ whose restriction to
 $E$ is $x^n+y^m$, and assume that in $Z$ the  multiplicities of $\{x=0\}$ and $\{y=0\}$ are $N$ and $M$. Then the orbit is identified with $f+\mbox{ideal}(f,x^Ny^M)$, which
 has a smooth transversal
 cut whose dimension is the  intersection multiplicity $(f,x^Ny^M)_{p_i}=mN+nM$.
 The $mN+nM$ coordinates of the cut cannot be  chosen canonically. We invite the reader to check
 that these coordinated can be chosen in such a way that first we choose the $m+n $
 (local) coordinated of the reduces part (as above in the case $Z=E$) then we can complete them with
 $m(N-1)+n(M-1)$   affine coordinates.

 Taking product  we obtain   charts of type
 $\prod_i\calU_i(Z):=r^{-1}(\prod_i\calU_i(E))=(\prod_i\calU_i(E))\times \bC^{(l',Z-E)}$.

(3) follows from the description of the above charts.
 \end{proof}

 \subsection{The tangent map of $c$. The smoothness of $c^{-1}(\calL)$.}\label{ss:GlSecA} \
Assume that $\calL\in \pic^{l'}(Z)$ has no fixed components. Fix any $D\in c^{-1}(\calL)\subset \eca^{l'}(Z)$,
and let $s\in H^0(Z,\calL)$ be the section whose divisor is $D$.
 Then multiplication by $s$ gives an exact sequence of sheaves
\begin{equation}\label{eq:ML}
0\to \calO_Z\stackrel{\cdot s}{\longrightarrow} \calL\to \calO_D\to 0.\end{equation}
Division by $s$ identifies $\calL$ by $\calO_Z(D)$, hence the above exact sequence
can be identified with the exacts sequence
$0\to \calO_Z\to \calO_Z(D)\to \calO_D(D)\to 0$
(this  is a generalization of the so-called {\it Mittag--Lefler sequence}, defined for effective divisors on curves).

Since $\calO_D$ is finitely
supported $H^0(\calO_D)=\calO_D$. Its dimension is  $(l',Z)$.

\begin{proposition}\label{lem:Mumford} The coboundary
homomorphism $\delta^1_D :H^0(\calO_D)\to H^1(\calO_Z)$ of the
cohomological long exacts sequence of  (\ref{eq:ML}) can be identified with the tangent map
$$T_D(c^{l'}):T_D(ECa^{l'}(Z))\to T_\calL(\pic^{l'}(Z))$$ of $c^{l'}$ at $D$.
Moreover, the Zariski tangent space  $T_D(c^{-1}(\calL))$ of $c^{-1}(\calL)$
at $D$
is identified with its kernel, hence (by the cohomological long exact sequence) by $H^0(Z,\calL)/H^0(\calO_Z)$.
This shows that  $\dim T_D(c^{-1}(\calL))=\dim c^{-1}(\calL)$ at any $D\in c^{-1}(\calL)$ (cf. (\ref{eq:dimfiber})),
hence $c^{-1}(\calL)$ is smoothly embedded into $\eca^{l'}(Z)$, and
 $c^{-1}(\calL)$, as a subscheme of $\eca^{l'}(Z)$,  can be identified with $H^0(Z,\calL)_{\reg}/H^0(\calO^*_Z)$.

This fact reformulated
shows that $\delta^1_D$ induced on  $N_D(c^{-1}(\calL)):=
T_D(ECa^{l'}(Z))/T_D(c^{-1}(\calL))$, the normal space
of $c^{-1}(\calL)\subset ECa^{l'}$ at $D$,  is injective.
\end{proposition}
\begin{proof} See \cite[p. 164]{MumfordCurves}, or \cite[Remark 5.18]{Kl},
or \cite[\S 5]{Kleiman2013}.
\end{proof}
\begin{corollary}\label{cor:smooth}
If $\dim(\eca^{l'}(Z))=1$ and $c^{l'}$ is not constant then $\im(c^{l'})$
is smooth.
\end{corollary}

\subsection{The special fibers of $c^{l'}$
}\label{ss:REGVAL}
 Though all the fibers of $c^{l'}$ are smooth, still we wish to distinguish  certain fibers  of $c^{l'}$ with pathological behaviour.
 There are several types we can consider.

 \begin{definition}\label{def:regular}\
 (a) $D\in \eca^{l'}(Z)$ is called a critical divisor (point)
 if  ${\rm rang}(T_Dc) < {\rm rang}(T_{D_{gen}}c)$, where $D_{gen}\in \eca^{l'}(Z) $ is a generic  divisor. If  $(c^{l'})^{-1}(\calL)$ contains a
 critical divisor (point)  then
 $\calL$ is a called a critical bundle (value).

(b) We say that  $\calL\in \im(c^{l'})$ is
$T$--typical (`tangent--map--typical')
if the linear subspace
 $\im (T_D(c^{l'}))\subset T_{\calL}\pic^{l'}(Z)$ is independent of the choice of $D\in c^{-1}(\calL)$. Otherwise $\calL$ is $T$--atypical.
 \end{definition}
The prototype of a map with a $T$--atypical value is the blowing up
$c:B\to \C^2$  at the origin $0\in \C^2$: then 0 is a
 $T$--atypical  value.
 For such an example realized by a concrete $c^{l'}$ see   \ref{ex:notequidim}.

 \begin{lemma}\label{lem:REGVAL}
 For fixed $l'$ and $\calL\in \im (c^{l'})$ consider the following properties:

 (i) $\calL$ is a $T$--atypical value of $c^{l'}$,

(ii) $\calL$ is a singular point of the closure
$\overline{\im(c^{l'})}$ of the image of $c^{l'}$
(where $\overline{\im(c^{l'})}$ is taken with the reduced structure),

 (iii) $\dim ((c^{l'})^{-1}(\calL))$ is strict larger the the dimension of the generic fiber of $c^{l'}$,

 (iv)  $\calL$ is critical bundle,

 (v) \underline{any} $D\in (c^{l'})^{-1}(\calL)$ is a critical divisor.

\vspace{1mm}

\noindent
 Then (iii) $\Leftrightarrow$ (iv) $\Leftrightarrow$ (v),
 (i) $\Rightarrow$ (iii) and
 (ii) $\Rightarrow$ (iii).
 \end{lemma}
 \begin{proof}
 The equivalences  {\it (iii)} $\Leftrightarrow$ {\it (iv) }
 $\Leftrightarrow$ {\it (v) }
 follow from Proposition \ref{lem:Mumford}. For {\it (i)} $\Rightarrow$ {\it (iii)}
 first notice that
 $c^{-1}(\calL)$ is smooth and irreducible, hence it is enough to verify the statement
 locally at a generic point of $c^{-1}(\calL)$. On the other hand,
 if  {\it (iii)} is not true, that is,
 if (locally)
 ${\rm rang}(T_Dc) = {\rm rang}(T_{D_{gen}}c)$, then $c$ in that neighbourhood
 is a fibration, hence (locally) the normal bundle of $c^{-1}(\calL)$ is a
 pullback of a vector space $V$, hence (using also  Proposition \ref{lem:Mumford})
  $\im (T_D(c))$ is constant $V$.

 {\it (ii)} $\Rightarrow$ {\it (iii)}. Assume that
  {\it (iii)} is not true, hence, as in the previous case,
    ${\rm rang}(T_Dc) = {\rm rang}(T_{D_{gen}}c)$ for any
     $D\in c^{-1}(\calL)$, and  $c$ in that neighbourhood
 is a fibration. $\im(c)$ is the image of the quotient
 space obtained from the total space by collapsing each
  fibers into a point. But for any  $D\in c^{-1}(\calL)$
  the space $N_D(c^{-1}(\calL))$ is mapped by $T_Dc$ injectively
  onto $\im (T_Dc)$, and this image is independent of
  the choice of $D$  (by the proof {\it (i)} $\Rightarrow$ {\it (iii)}).
  This shows that, in fact, $\im (c)$ is immersed at $\calL$. Since the fiber  $c^{-1}(\calL)$   is  connected, $\im(c)$ is in fact embedded. Hence,
  $\im (c)$ is smooth at $\calL$.
 \end{proof}

\begin{remark} (a) In principle, by general properties of algebraic maps,
neither {\it (iii)}  nor {\it (ii)}   imply {\it (i)} by general nonsense.
 Indeed,
 set $Y:=\{(x,y,a)\in\bC^3\,:\, y=ax^2\}$ and consider the projection
$p:Y\to \bC^2$ induced by $(x,y,a)\mapsto (x,y)$. Then the fiber $p^{-1}(0,0)$
satisfy {\it (iii)}  but not {\it (i)}.

Take also  $Y:=\C^*\times \C$ and consider the map $Y\to  \C^3$ given by
$(x,y)\mapsto (yx,y^2x,y^3x)$. Then the closure of the
image is singular at the origin. This value satisfies
 {\it (ii)} and {\it (iii)} but not {\it (i)}.

The implication   {\it (iii)} $\Rightarrow$ {\it (ii)} also fails, in general.
E.g., if $c^{l'}$ is dominant, and $\calL\in \im(c^{l'})$, then $\calL$ is not a distinguished  point of the {\it closure}  of the image even if it is critical,
see e.g. \ref{ex:notequidim}.
However, examples suggest   the following conjectural property:
if $c^{l'}$ is dominant and some $\calL\in \pic^{l'}(Z)$ satisfies
{\it (iii)} then $\calL$ is not an interior point of $\im (c)$.

(b)
We wish to emphasize again that $c^{l'}$ is not proper. In particular,
 above a small (relative) neighbourhood
  in $\im(c^{l'})$ of   a regular value, the global map  $c^{l'}$ is not necessarily $C^{\infty}$ locally trivial fibration
 (see e.g. Example \ref{ex:nonfibration}).
\end{remark}

\subsection{Examples} Next we exemplify  some typical anomalies  of the map $c$.

\begin{example}\label{ex:dimim} Fix a topological type of singularities (e.g. a resolution graph)
and consider different analytic structures realizing it. Then
not only the dimension of the target of $c:\eca^{l'}(Z)\to \pic^{l'}(Z)$ (that is, $h^1(\calO_Z)$)
but also the {\bf dimension of the image of $c^{l'}$ might depend on the analytic structure of $(X,o)$}.
Indeed, let us fix the following graph
(picture from the left):

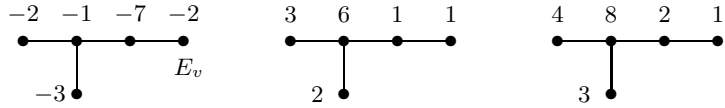
\begin{picture}(200,50)(0,0)
\put(70,40){\makebox(0,0){\small{$-2$}}}
\put(90,40){\makebox(0,0){\small{$-1$}}}
\put(70,30){\circle*{4}}
\put(110,40){\makebox(0,0){\small{$-7$}}}\put(130,40){\makebox(0,0){\small{$-2$}}}
\put(132,20){\makebox(0,0){\small{$E_v$}}}
\put(80,10){\makebox(0,0){\small{$-3$}}}
\put(90,30){\circle*{4}}
\put(110,30){\circle*{4}}
\put(130,30){\circle*{4}}
\put(90,10){\circle*{4}}
\put(70,30){\line(1,0){60}}\put(90,10){\line(0,1){20}}

\put(170,40){\makebox(0,0){\small{$3$}}}
\put(190,40){\makebox(0,0){\small{$6$}}}
\put(170,30){\circle*{4}}
\put(210,40){\makebox(0,0){\small{$1$}}}
\put(230,40){\makebox(0,0){\small{$1$}}}
\put(180,10){\makebox(0,0){\small{$2$}}}
\put(190,30){\circle*{4}}
\put(210,30){\circle*{4}}
\put(230,30){\circle*{4}}
\put(190,10){\circle*{4}}
\put(170,30){\line(1,0){60}}
\put(190,10){\line(0,1){20}}

\put(270,40){\makebox(0,0){\small{$4$}}}
\put(290,40){\makebox(0,0){\small{$8$}}}
\put(270,30){\circle*{4}}
\put(310,40){\makebox(0,0){\small{$2$}}}
\put(330,40){\makebox(0,0){\small{$1$}}}
\put(280,10){\makebox(0,0){\small{$3$}}}
\put(290,30){\circle*{4}}
\put(310,30){\circle*{4}}
\put(330,30){\circle*{4}}
\put(290,10){\circle*{4}}
\put(270,30){\line(1,0){60}}\put(290,10){\line(0,1){20}}
\end{picture}

Then $(X,o)$ is a numerically Gorenstein elliptic singularity with $1\leq p_g\leq 2$;
for details regarding elliptic singularities
 see \cite{weakly,Nfive}. Set $-l':=Z_{min}$ (the minimal  cycle, which equals $E_v^*$,
the cycle shown in the middle diagram), and $Z=Z_K $ (the last diagram),
hence $(Z,l')=1$.
Then $\eca^{l'}(Z)=\bC$, and $\pic^{l'}(Z)=\bC^{p_g}$.
Write $\calL=\calO_{Z_K}(-Z_{min})$.

If $p_g=2$ (hence $(X,o)$ is Gorenstein) then $\calL$ has no fixed
components \cite[5.4]{weakly}, and $h^1(Z,\calL)=1$ \cite[2.20(d)]{weakly}. Hence $\calL\in \im(c)$
and $\dim c^{-1}(\calL)=0$ (use (\ref{eq:dimfiber})). Therefore, $\dim\im(c)=1$.

On the other hand, if $p_g=1$, then
 $Z_{max}>Z_{min}$, see e.g. \cite[2.20(f)]{weakly}.
 Hence $\calL$  has fixed components and  $\calL\not\in \im(c)$.
Since the fibers of $c$ are connected (cf. \ref{lem:free}), $c:\bC\to \bC$ (with  $\calL\not\in \im(c)$)
cannot be quasi--finite, hence  $c$ is constant and  $\dim\im(c)=0$.
(This last statement  can  be deduced from  Theorem \ref{th:dominant} too,
or from \ref{ex:dimimzero}  {\it (i)} $\Leftrightarrow$ {\it (v)}, where we characterize completely the cases
$\dim (\im (c^{l'}))=0$.)
\end{example}
\begin{example}\label{ex:notclosed} {\bf The image of $c$ usually is not closed.}
We construct such an example in two steps. First,
assume that $(X,o)$ is a singularity with topological type given by the  graph  $\Gamma_1$ from the left

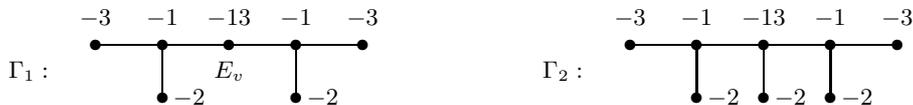
\begin{picture}(300,45)(80,0)
\put(125,25){\circle*{4}}
\put(150,25){\circle*{4}}
\put(175,25){\circle*{4}}
\put(200,25){\circle*{4}}
\put(225,25){\circle*{4}}
\put(150,5){\circle*{4}}
\put(200,5){\circle*{4}}
\put(125,25){\line(1,0){100}}
\put(150,25){\line(0,-1){20}}
\put(200,25){\line(0,-1){20}}
\put(125,35){\makebox(0,0){\small{$-3$}}}
\put(150,35){\makebox(0,0){\small{$-1$}}}
\put(175,35){\makebox(0,0){\small{$-13$}}}
\put(200,35){\makebox(0,0){\small{$-1$}}}
\put(225,35){\makebox(0,0){\small{$-3$}}}
\put(160,5){\makebox(0,0){\small{$-2$}}}
\put(210,5){\makebox(0,0){\small{$-2$}}}
\put(175,15){\makebox(0,0){\small{$E_v$}}}
\put(100,15){\makebox(0,0){\small{$\Gamma_1:$}}}

\put(300,15){\makebox(0,0){\small{$\Gamma_2:$}}}
\put(325,25){\circle*{4}}
\put(350,25){\circle*{4}}
\put(375,25){\circle*{4}}
\put(400,25){\circle*{4}}
\put(425,25){\circle*{4}}
\put(350,5){\circle*{4}}
\put(400,5){\circle*{4}}\put(375,5){\circle*{4}}
\put(325,25){\line(1,0){100}}
\put(350,25){\line(0,-1){20}}\put(375,25){\line(0,-1){20}}
\put(400,25){\line(0,-1){20}}
\put(325,35){\makebox(0,0){\small{$-3$}}}
\put(350,35){\makebox(0,0){\small{$-1$}}}
\put(375,35){\makebox(0,0){\small{$-13$}}}
\put(400,35){\makebox(0,0){\small{$-1$}}}
\put(425,35){\makebox(0,0){\small{$-3$}}}
\put(360,5){\makebox(0,0){\small{$-2$}}}\put(385,5){\makebox(0,0){\small{$-2$}}}
\put(410,5){\makebox(0,0){\small{$-2$}}}
\end{picture}

\noindent
Furthermore, assume that the minimal cycle $Z_{min}$ equals the maximal ideal cycle $Z_{max}$.
In particular, $\calO(-Z_{min})$ has no fixed components. For a detailed study of this singularity
(and any analytic type with the above graph) see \cite{NO17}.
Set $-l'=Z=Z_{min}=E_v^*$, and $\calL:=\calO_{Z}(-Z)$. Since $(E_v^*,E^*_v)=-1$ and the $Z(\bt)$--coefficient $z(E^*_v)=0$,
one has
$\eca^{l'}(Z)=\bC^*$. In fact, the corresponding effective divisors correspond to the points of
$E_v^{reg}:=E_v\setminus {\rm Sing}(E)$. Using (\ref{eq:dimfiber}) and  \cite[\S 4]{NO17}
(which shows $h^1(\calL)=1$)
one obtains that $\dim c^{-1}(\calL)=0$.
Furthermore, $\pic^{l'}(Z)=\bC^2$ (cf. \ref{not:min}), hence we get an injection $c:\bC^*\hookrightarrow \bC^2$.
For any $q\in E_v^{reg}=\bC^*$ we write $\calL_q:=c^{l'}(q)\in \pic^{l'}(Z)$.

In fact, $\im (c^{l'})$ can be determined explicitly. Let $\Gamma_l$ and $\Gamma_r$ be the subgraphs
consisting of the left/right cusp together with $v$. They determine minimally elliptic singularities with
$p_g=1$, and the corresponding restrictions provide the two coordinates in $\pic^{l'}(Z)$.
Applying
 \cite[6.11.4]{hartshorne} for these two coordinates  we get that $\im (c^{l'})$ in some affine coordinates
$(z_1,z_2)$ has the form $z_1z_2=1$.

Furthermore, this situation can be used to
analyze another singularity $(X',o)$, whose $\im (c')$ equals $\im (c)\setminus \{\mbox{1 point}\}$.
Fix an arbitrary point $p\in E^{reg}_v$, and glue to the resolution of $(X,o)$
(associated with $\Gamma_1$)
another irreducible $(-2)$--exceptional curve $E'_{p}$
transversally to $E_v$ at $p$.
In this way we create the resolution of a new singularity $(X',o)$  with exceptional curve
$E'=\{E_v'\}_v\cup\{E'_p\}$ (with natural notations).
The new graph is on the right hand side above.

In the new situation we take $-l'=E_v'^*$ and $Z':= Z_{min}'=E_p'^*$.
Then  $\eca^{l'}(Z')$ can be identified with $(E_v')^{reg}=E_v^{reg}\setminus \{p\}=\bC^*\setminus \{p\}$,
and $c':\bC^*\setminus \{p\}\to \bC^2$ with the restriction of $c$ to $\bC^*\setminus \{p\}$.
(More precisely, for $q\in \bC^*\setminus \{p\}$ one has $c'(q)|_Z=\calL_q\otimes \calO_Z(p)$.)
Since $c$ is injective, the image of $c'$ cannot be closed. Via similar construction
we can eliminate from the image of $c$ any point.
\end{example}
\begin{example}\label{ex:notequidim} The map $c$ usually is not a locally trivial fibration
over its image, in fact,  {\bf the fibers of $c$ usually are not even equidimensional}.

Consider the graph $\Gamma_1$ from Example \ref{ex:notclosed}. It can be realized also by a
complete intersection (splice quotient) singularity with $p_g=3$, cf. \cite{NWCasson,NO17}.
Set $-l'=2Z_{min}=2E_v^*$ and $Z=Z_{min}$. Then $ \eca^{l'}(Z)$ is the double symmetric product of
$E_v^{reg}$, namely $\bC^*\times \bC^*/\bZ_2\simeq \bC^*\times \bC$.
On the other hand, $\pic^{l'}(Z)=\bC^2$. (For numerical cohomological  invariants see again \cite{NO17}.)
It turns out that $c$ is dominant (use e.g.  Theorem \ref{th:dominant}(3)), hence $c$ is birational,
with all fibers connected.
Since $Z_{max}=2Z_{min}$, $\calL=\calO_Z(-2Z_{min})$ has no fixed components, hence $\calL\in\im(c)$.
Furthermore, $h^1(\calL)=1$ (see e.g. \cite[(5.4)]{NO17}), hence  $\dim c^{-1}(\calL)=1$ by (\ref{eq:dimfiber})
(since  $h^1(\calO_Z)=2$ and $(l',Z)=2$). This can be seen in the following way as well.
By Riemann--Roch $h^0(\calL)=2$ and $H^0(\calO_Z)^*=\bC^*$, hence by \ref{lem:free} $c^{-1}(\calL)$ is
1--dimensional.
In particular,  the fibers of $c$ are not equidimensional.
(Furthermore,  one can show that
$\im(c)$ is homeomorphic to $(\bC^*)^2\cup \{(0,0)\}$, where $(0,0)$ corresponds to $\calL$.)
\end{example}

\begin{example}\label{ex:nonfibration}
{\bf Even if all the fibers have the same dimension} (and by Theorem \ref{lem:Mumford}  they are smooth)
{\bf the topology of some fibers might jump}. Take for example the graph

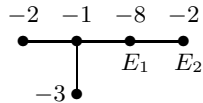
\begin{picture}(200,50)(0,0)
\put(70,40){\makebox(0,0){\small{$-2$}}}
\put(90,40){\makebox(0,0){\small{$-1$}}}
\put(70,30){\circle*{4}}
\put(110,40){\makebox(0,0){\small{$-8$}}}\put(130,40){\makebox(0,0){\small{$-2$}}}
\put(112,22){\makebox(0,0){\small{$E_1$}}}
\put(132,22){\makebox(0,0){\small{$E_2$}}}
\put(80,10){\makebox(0,0){\small{$-3$}}}
\put(90,30){\circle*{4}}
\put(110,30){\circle*{4}}
\put(130,30){\circle*{4}}
\put(90,10){\circle*{4}}
\put(70,30){\line(1,0){60}}\put(90,10){\line(0,1){20}}
\end{picture}

It supports a non--numerically Gorenstein elliptic singularity. Recall that
if $C$ denotes the elliptic cycle
(here it is supported on the union  of all irreducible exceptional curves except $E_2$), and
 $(C,Z_{min})<0$, then  the  length of the elliptic sequence is one,
cf. \cite{Yau5,Yau1}.
Hence, for any analytic realization,  $p_g=1$. Take $-l'=Z=Z_{min}=E_1^*+E_2^*$.
A computation shows that $\eca^{l'}(Z)=\bC^2\setminus \{0\}$.
Then $c:\bC^2\setminus \{0\}\to \bC$ can be identified with the restriction to $\bC^2\setminus \{0\}$
of the linear projection $\bC^2\to \bC$. Hence the generic fiber is $\bC$ while there is a special fiber
$\simeq \bC^*$. By this correspondence $\pic^{l'}(Z)=\bC$ is identified by $E_1\setminus E_{node}$.
The generic fibers correspond to the divisors $\{p,q\}$, where $p\in E_1^{reg}\simeq \bC^*$,
and $q\in E_2^{reg}\simeq \bC$; they are sent by $c$ to $p\in E_1^{reg}\subset  E_1\setminus E_{node}
\simeq\pic^{l'}(Z)$. Since $q$ can be any
 point on $E_2^{reg}$, the fibers are $\bC$.
On the other hand, any divisor given by a smooth cut at $E_1\cap E_2$, transversal to both
$E_1$ and $E_2$, (parametrized by the slope $\bC^*$)
is sent by $c$ to  $E_1\cap E_2$, whose fiber is exactly this parameter space $\bC^*$.
\end{example}
\begin{example}\label{ex:singularimage}
For an example when the image of $c$ is singular see
Example \ref{ex:whSING}.
\end{example}
\subsection{The topology of the fibers of $c$ and the coefficients of the Poincar\'e series}\label{ss:PequalsZ} \

Let us analyse again the fibers of $c:\eca^{l'}(Z)\to \pic^{l'}(Z)$, $Z\geq  E$. Assume that
$\calL\in \im (c)$. Then $\{H_v:=H^0(Z-E_v,\calL(-E_v))\}_{v\in\calv}$
is a proper linear subspace arrangement in $H^0:=H^0(Z,\calL)$. For any subset
$\emptyset \not=I\subset \calv$ write $H_I:=\cap_{v\in I}H_v$, and introduce also
$H_\emptyset :=H^0$. Note that the topological Euler characteristic satisfies
$\chi_{top}(\bP H_I)=\dim H_I$, hence by the inclusion--exclusion principle one obtains
\begin{equation}\label{eq:Arrang}
\chi_{{\rm top}} ( \bP (H^0\setminus\cup _v H_v))=\sum_{I\subset \calv} \, (-1)^{|I|}
\dim H_I=
\sum_I(-1)^{|I|+1}\mathrm{codim}(H_I\subset H^0).
\end{equation}
In particular  the analytic invariant $p_{Z,\calL}$ (cf. \ref{bek:AMS}) equals
the topological Euler characteristic of the corresponding linear subspace arrangement complement,
$p_{Z,\calL}=\chi_{top}(\bP ( H^0(Z,\calL)_{\reg}))$.
Using Lemmas \ref{lem:H_0} and \ref{lem:free} this reads as
$$p_{Z,\calL}=\chi_{top}( c^{-1}(\calL)).$$
This fact links the coefficients
of the topological series $Z(\bt)$ and the numerical analytical
invariants $p_{Z, \calL}$: the Euler characteristic of the total space $\eca^{l'}$
is $z(-l')$, while the Euler characteristic of each fiber $c^{-1}(\calL)$ ($\calL\in \im (c)$) is
$p_{Z,\calL}$.
\begin{example}\label{eq:RATZ=P}
Assume that $(X,o)$ is rational. Then $\pic^{l'}(Z)$ is a point: if $c_1(\calL)=l'$ then
 $\calL=\calO(l')$. Hence $\eca^{l'}$ is the unique fiber
 $c^{-1}(\calO(l'))$. Therefore, $z(-l')=p_{Z,\calO(l')}$ ($l'\in -\calS'$), or
$Z(\bt)=P_{Z,\calO}(\bt)$. This generalizes the similar  identity proved in
\cite{CDGPs,CDGEq,NPS,NCL} valid for $Z\gg 0$ (or, for $\widetilde{X}$).
\end{example}
This identity $Z(\bt)=P_{\calO_{\widetilde{X}}}(\bt)$ is valid for a
 more general family of   singularities, namely for splice quotient singularities \cite{NCL,NPS}.
(This family was introduced  by Neumann and Wahl in  \cite{NWsqI,NWsq}).
This identity
reinterpreted in our present language says  that for any $-l'\in\calS'$ and $Z\gg 0$
the Euler characteristic
of the total space $\eca^{l'}(Z)$ and the Euler characteristic of the
 very special fiber $c^{-1}(\calO(l'))$ (over the unique natural line bundle) coincide.

\begin{conjecture}
For a splice quotient singularity and  $-l'\in\calS'$ the fiber  
$c^{-1}(\calO(l'))$ is a topological deformation retract of $\eca^{l'}(Z)$.
\end{conjecture}
A detailed study of the Abel map in the case of splice quotient singularities
will appear in one of the parts of  the present series of articles.

In the present work we wish to focus
(instead/besides  of the `$P_{\calO}=Z$ identity')
on the more complex package of invariants provided
by (all the fibers of) $c$. In particular, we analyse other, less specific fibers as well,
e.g. the generic fibers over $\im (c)$.

\section{Characterization of $c^{l'}$ dominant}\label{s:dominant}

 \subsection{} In order to determine properties
of line bundles $\calL\in \pic(Z)$  with  given Chern class we need first to understand the situations
when $c^{l'}$ is dominant.

\begin{theorem}\label{th:dominant}
 Fix $l'\in -\calS'$,
$Z\geq E$ as above, and consider $c^{l'}:\eca^{l'}(Z) \to \pic^{l'}(Z)$.

(1) $c^{l'}$ is dominant if and only if $H^0(Z,\calL)_{\reg}
\not=\emptyset $ for generic $\calL\in\pic^{l'}(Z)$.

(2) If $c^{l'}$ is dominant then $h^1(Z,\calL)=0$ for  generic $\calL\in\pic^{l'}(Z)$.

(3) $c^{l'}$ is dominant  if and only if $\chi(-l')<\chi(-l'+l)$ for all $0<l\leq Z$, $l\in L$.
In particular, the fact that $c^{l'}$ is dominant is independent of the analytic structure supported by $\Gamma$ and it can be characterized topologically (and explicitly).
\end{theorem}
\begin{proof}
For {\it (1)} use Lemma \ref{lem:H_0}. For {\it (2)}
note that for $c$  dominant
 the dimension of $\eca^{l'}(Z)$ is the sum of the dimensions of the generic fiber and of the base
 (which equals $h^1(\calO_Z)$). Hence, by (\ref{eq:dimfiber}) and \ref{th:smooth}{\it (1)},
  $h^0(Z,\calL)=\dim c^{-1}(\calL)+h^0(Z)=(l',Z)-h^1(Z)+h^0(Z)=(l',Z)+\chi(Z)=\chi(Z,\calL)$.

  {\it (3)} First note that for any cycle $l\in L$, $0<l\leq Z$, and any $\calL\in \pic^{l'}(Z)$,
  one has
  \begin{equation}\label{eq:chieq}
 \chi(-l')\geq \chi(-l'+l) \ \Leftrightarrow \
 \chi(Z,\calL)\leq \chi(Z-l,\calL(-l)),
  \end{equation}
  where, by convention, $\chi(Z-l,\calL(-l))$ is zero whenever $l=Z$.

 Assume that $c$ is dominant and the equivalent conditions from
 (\ref{eq:chieq}) are satisfied for some $l$, where $0<l\leq Z$.
  Take a generic $\calL\in \pic^{l'}(Z)$. Hence $H^0(Z,\calL)_{\reg}\not=
 \emptyset$ (cf. part {\it (1)}) and $\chi(Z,\calL)=h^0(Z,\calL)$ by part
 {\it (2)}. Hence
 $h^0(Z,\calL)=\chi(Z,\calL)\leq \chi(Z-l,\calL(-l))\leq h^0(Z-l,\calL(-l))$. Therefore,
 by the cohomological exact
  sequence of $0\to \calL(-l)|_{Z-l}\to \calL$, we necessarily have equality
 $H^0(Z-l, \calL(-l))=H^0(Z,\calL)$. Then for any $E_v$ in the support of $l$  we also have
  equality $H^0(Z-E_v, \calL(-E_v))=H^0(Z,\calL)$, hence $H^0(Z,\calL)_{\reg}=\emptyset$,
 which leads to a contradiction.

 Assume that $\chi(-l')<\chi(-l'+l)$ for any $0<l\leq Z$. This, for $l=Z$, implies $\chi(Z,\calL)>0$,
 hence necessarily $h^0(Z,\calL)>0$ for any $\calL\in \pic^{l'}(Z)$.
 If for a generic $\calL$ one has  $H^0(Z,\calL)_{\reg}=\emptyset$, then
  there exists $E_v$ such that
 $H^0(Z,\calL)=H^0(Z-E_v,\calL(-E_v))$.
  If $H^0(Z-E_v,\calL(-E_v))_{\reg}=\emptyset$ again,
 then we continue the procedure. In this way we obtain
 a cycle $0<l\leq Z$ such that  $H^0(Z-l, \calL(-l))=H^0(Z,\calL)$
 and $H^0(Z-l, \calL(-l))_{\reg}\not=\emptyset$. Note that for $\calL$ generic $\calL(-l)|_{Z-l}
 \in\pic^{l'-l}(Z-l) $
 is generic as well. Hence $c^{l'-l}$ is dominant and
 by {\it (1)--(2)} $h^1(Z-l,\calL(-l))=0$.
 Therefore,
 $\chi(Z,\calL)\leq h^0(Z,\calL)=h^0(Z-l,\calL(-l))=\chi(Z-l,\calL(-l))$, which by
 (\ref{eq:chieq})
 reads as $\chi(-l')\geq \chi(-l'+l)$, a contradiction.
\end{proof}

\begin{example}\label{ex:RATZ} The statement of Theorem \ref{th:dominant}{\it (3)}
is non--empty even for $l'=0$. In this case, since $\eca^0$ is a point,
$c^0$ is dominant if and only if $\pic^0(Z)$ is a point, that is,
$h^1(\calO_Z)=0$. Hence part {\it (3)} reads as the following
topological characterization of the vanishing of $h^1(\calO_Z)$:
For any normal surface singularity and any cycle $Z>0$,
$h^1(\calO_Z)=0$ if and only if $\chi(l)>0$ for any $0<l\leq Z$.
(This is a generalization of the rationality criterion of Artin \cite{Artin62,Artin66}, which corresponds
to $Z\gg 0$.)
\end{example}
\begin{remark}\label{rem:Zsupport}
Above, we assumed that $Z\geq E$. This is not really necessary: if the support $|Z|$ of $Z$
is smaller then one can restrict all the objects to $|Z|$, and the above statements (and also
the next  Theorem \ref{th:hegy})
remain valid. (Along the restriction, $\widetilde{X}$ will be replaced by a small convenient neighbourhood of
$\cup_{E_v\subset|Z|}E_v$, and $L$ by $\bZ\langle E_v\rangle_{E_v\subset |Z|}$.)
\end{remark}
\bekezdes\label{bek:DOMChern}
 {\bf The semigroup of dominant Chern classes ($Z\gg 0$).} \
Theorem \ref{th:dominant}{\it (3)} motivates the introduction of the following combinatorial
set
$$\calS'_{dom}:=\{-l'\ |\ \chi(-l')<\chi(-l'+l)\ \mbox{for all $l\in L_{>0}$}\}.$$
By definition, $-l'\in \calS'_{dom}$ if and only if
 $c^{l'}$ is dominant for $Z\gg 0$.

Sometimes it is more convenient to use the next equivalent form (note the sign modification):
\begin{equation}\label{eq:DOMIN}
\calS'_{dom}=\{l'\ |\ \chi(l)>(l',l)\ \mbox{for all $l\in L_{>0}$}\}.\end{equation}

\begin{lemma}\label{lem:DOMIN}
$\calS'_{dom}$ has the following properties:

(i) $\calS'_{dom}\subset \calS'$.

(ii) $0\in\calS'_{dom}$ iff $L$ is rational.
More generally, for $I\subset \calv$ and $n_v>0$ for all $v\in I$,
if $\sum_{v\in I}n_vE^*_v\in \calS'_{dom}$ then the components
of $ \cup_{v\not \in I} E_v$ are rational. Hence, in general,
  $\calS'\setminus \calS'_{dom}$ is infinite.

(iii)   $\calS'\cap (Z_K/2+\calS'_{\Q})\subset \calS'_{dom}$, where $\calS'_{\Q}:=\{ l' \in L\otimes \Q\,:\,
(l',E_v)\leq 0 \ \mbox{for all} \ v\}$.

(iv) $\calS'_{dom}$ is a semigroup (not necessarily with
identity element).

(v) $\calS'_{dom}$ is an $\calS'$--module, that is, if $l_1'\in \calS'_{dom}$,
$l_2'\in \calS'$ then $l_1'+l_2'\in \calS'_{dom}$.

(vi) $\calS'_{dom}$ is $\min$--stable, like $\calS'$, that is,
if \,$l_1', l_2'\in \calS'_{dom}$ then $m:=\min\{l_1',l_2'\}\in \calS'_{dom}$.

\end{lemma}
\begin{proof} For $(i)$ use (\ref{eq:empty}) or (\ref{eq:DOMIN}). $(ii)$  follow from Artin's criterion.
$(iii)$ is clear.
For $(iv)-(v)$ use (\ref{eq:DOMIN}):
if $\chi(l)>(l'_1,l)$ and $0>(l'_2,l)$ (cf. $(i)$), then $\chi(l)> (l'_1+l'_2,l)$.
 Next we prove $(vi)$.

We wish to show that $\chi(l)>(m,l)$ for any $l>0$. Set $x_i=l_i'-m$ ($i=1,2$). Assume first that
$l\geq x_1$, and write $l=x_1+z$. Then from the assumptions $\chi(x_1)\geq (m+x_2,x_1)$  (equality only
if $x_1=0$) and
$\chi(z)\geq (m+x_1,z)$ (equality only if $z=0$). These added provide $\chi(l)> (m,l)+(x_1,x_2)\geq (m,l)$.

Next assume that $l\not\geq x_1$, and choose $u_1>0$ minimal,   supported by the support of
$x_1$,  such that $l+u_1\geq x_1$.  Then the hypothesis applied for $l_1'=m+x_1$ gives
$\chi(l+u_1-x_1)\geq (m+x_1,l+u_1-x_1)$ (equality only if $l+u_1-x_1=0$)
and applied for $l_2'=m+x_2$ gives $\chi(x_1-u_1)\geq (m+x_2,x_1-u_1)$ (equality only if $x_1-u_1=0$).
These added gives
$\chi(l)-(m,l)> (u_1,l+u_1-x_1)+(x_2,x_1-u_1)\geq 0$.
\end{proof}
\begin{corollary}\label{cor:DOMIN}
(i) \ For any $-l'\in L'$ there exists a unique minimal $l_{dom}\in L_{\geq 0}$ with
$-l'+l_{dom}\in \calS'_{dom}$.

(ii) \ $l_{dom}$ can be found by the following
algorithm (see the analogy with \cite{Laufer72}). We construct a computation sequence $\{z_i\}_{i=0}^t$,
(where $z_{i+1}=z_i+E_{v(i)}$ for some $v(i)\in \calv$) as follows. Fix a generic line bundle
$\calL\in \pic^{l'}(\tX)$.  Start with  $z_0=0$. Assume that
$z_i$ is already constructed and consider $\calL(-z_i)$. If $H^0(\calL(-z_i))$
has no fixed components then stop and $z_i$ is the last term $z_t$.
If $H^0(\calL(-z_i))$ has a fixed component, choose one of them, say
$E_{v(i)}$, and  write $z_{i+}:=z_i+E_{v(i)}$ and repeat the algorithm.
Then this procedure stops after finitely many steps and $z_t=l_{dom}$.
\end{corollary}
\begin{proof} $(i)$ Set ${\mathcal D}:=(-l'+L_{\geq 0})\cap \calS'_{dom}$. Then
 ${\mathcal D}\not=\emptyset$ by \ref{lem:DOMIN}$(iii)$ and it has a unique minimal element by \ref{lem:DOMIN}$(vi)$.

$(ii)$ We show inductively that $z_i\leq l_{dom}$ and the construction stops exactly when
$z_i=l_{dom}$. Note that $z_0=0\leq l_{dom}$. If $z_i=l_{dom}$ then $-l'+z_i\in \calS'_{dom}$,
hence by Theorem \ref{th:dominant}(1) $H^0(\calL(-z_i))$ has no fixed components, hence we have to stop.

If, by induction $z_i<l_{dom}$, we have to show that the algorithm does not stop and
$z_{i+1}\leq l_{dom}$ as well. Indeed, if $-l'+z_i<-l'+l_{dom}$ then $-l'+z_i\not\in
\calS'_{dom}$ by the minimality of $l_{dom}$, hence by Theorem \ref{th:dominant} $H^0(\calL(-z_i))$
has fixed components. Hence the procedure continues. Note also that the generic section of
$H^0(\calL(-l_{dom}))$ has no fixed components, hence the fixed components of $H^0(\calL(-z_i))$
should be supported on $l_{dom}-z_i$. Hence $z_i+E_{v(i)}\leq l_{dom}$.
\end{proof}

\begin{remark}
Though $\calS'_{dom}$  is defined above combinatorially/topologically,
it shares (see e.g. {\it (iv)} and {\it (vi)}) several properties of an {\it analytic
semigroup } associated with an analytic structure supported on $\Gamma$. This `coincidence' will be
clarified completely in the forthcoming part \cite{NNII}, where we prove that
the analytic semigroup associated with the {\it generic analytic structure } is exactly
$\calS'_{dom}\cup \{0\}$.
\end{remark}

 \section{Cohomology of line bundles and $\dim \im (c^{l'})$}\label{s:generic}

\subsection{Line bundles with $c_1(\calL)\not\in-\calS'$.} \label{ss:Lauferseq}

Recall that by (\ref{eq:empty}) $\eca^{l'}(Z)\not=\emptyset$ iff $l'\in -\calS'$.
Hence any result based on  the Abel map uses  $l'\in -\calS'$. E.g.,
in this section
we establish a sharp lower bound for $h^1(Z, \calL)$ whenever $c_1(\calL)=l'\in -\calS'$. Before we provide that statement we wish
to emphasise that this extends automatically to the case of all bundles $\calL$, even if $c_1(\calL)\not\in -\calS'$.

 Indeed, it is known that for any $x\in L'$ there exist $s(x)=x+l\in L'$ with the following properties:
(a) $s(x)\in \calS'$, (b) $l\in L_{\geq 0}$, (c) $s(x)$ is minimal with properties (a)-(b).
Furthermore, the cycle $l$ can be determined explicitly using a generalized Laufer sequence \cite[Prop. 4.3.3]{trieste}.
One constructs a computation sequence $\{z_i\}_{i=0}^t$, $z_0=0$, $z_{i+1}=z_i+E_{v(i)}$ for some $v(i)\in\calv$
inductively as follows. If $x+z_i\in\calS'$ then one stops, and automatically $i=t$ and $z_i=l$.
If there exists $E_v$ with $(x+z_i,E_v)>0$ then choose $E_{v(i)}$ as  such an $E_v$, and one defines $z_{i+1}=z_i+E_{v(i)}$.
Along the computation sequence $i\mapsto \chi(x+z_i)$ is decreasing. Furthermore,
if $Z>l$, then the sequence applied for $x= -l'=-c_1(\calL)$, we get that
 $h^0(Z-z_i,\calL(-z_i))$ is constant, and
 \begin{equation}\label{eq:H1Gen}
 h^1(Z,\calL)=h^1(Z-l,\calL(-l))-\chi(\calL|_l) \ \ \ \mbox{and} \ \ \
 c_1(\calL(-l))\in-\calS'.
 \end{equation}
Here, clearly, $\chi(\calL|_l)=(l',l)+\chi(l)=\chi(-l'+l)-\chi(-l')$. If $l\not\leq Z$, then one constructs
a computation sequence inductively as follows: if $-l'+z_i\in \calS'(|Z-z_i|)$ (the Lipman cone associated with the
support $|Z-z_i|$) then one stops, otherwise there exists $E_{v(i)}$ (identified as above)
supported on $Z-z_i$, which provides
$z_{i+1}=z_i+E_{v(i)}$. In particular, for any $\calL\in \pic(Z)$, there exists $l\in L_{\geq 0}$ such that
$-c_1(\calL(-l))\in \calS'(|Z-l|)$, and (\ref{eq:H1Gen}) holds.

Summarized, the computation of any $h^1(Z,\calL)$,  up to the topology of the graph,
can be  reduced to the case $-c_1(\calL)\in \calS'$ (maybe supported on a smaller set).

\subsection{Semicontinuity.} We emphasise another specific fact as well: since
 $c^{l'}$ is not proper, the semicontinuity
of the dimension of the
fiber (with respect to the points of the target) does not follow automatically from the general theory.
Nevertheless, we have the following result.

\begin{lemma}\label{lem:semicont}
$h^0(Z,\calL)$ and $h^1(Z,\calL)$  are semicontinuous with respect to $\calL\in \pic^{l'}(Z)$.
In particular, via (\ref{eq:dimfiber}), $\dim c^{-1}(\calL)$ is also semicontinuous
with respect to $\calL\in \pic^{l'}(Z)$.
\end{lemma}
\begin{proof}
Consider a  covering by small balls  $\{U_\alpha\}_\alpha$ of $\tX$. Since $\calL|_{U_\alpha}$ is trivial for any
$\alpha$ and $\calL$, $H^0(Z,\calL)=\ker(\delta_\calL:\oplus_\alpha H^0(\calO_Z|_{U_\alpha})\to
\oplus_{\alpha\not=\beta} H^0(\calO_Z|_{U_\alpha\cap U_{\beta}}))$, where the
$\calL$--dependence is codified in $\delta_{\calL}$.
But the corank
of the linear map (hence, consequently $h^0(Z,\calL )$ too)
is semicontinuous. The  semicontinuity of
$h^1(Z,\calL )$ follows by Riemann--Roch.
\end{proof}

\subsection{} We prove the following  sharp semicontinuity inequality.

\begin{theorem}\label{th:hegy} (1)
Fix an arbitrary $l'\in L'$. Then for any $\calL\in \pic^{l'}(Z)$ one has
\begin{equation}\label{eq:genericL}
\begin{array}{ll}h^1(Z,\calL)\geq \chi(-l')-\min_{0\leq l\leq Z,\, l\in L} \chi(-l'+l), \ \ \mbox{or, equivalently}\\
h^0(Z,\calL)\geq \max_{0\leq l\leq Z,\, l\in L}\chi(Z-l,\calL(-l))=\max_{0\leq l\leq Z,\, l\in L}
\{\, \chi(Z-l)+(Z-l,l'-l)\,\}.\end{array}\end{equation}
Furthermore, if $\calL$ is generic in $\pic^{l'}(Z)$ then in both inequalities we have equality.

In particular,   $h^*(Z,\calL)$ is topological and explicitly computable from $L$,
whenever $\calL$ is generic.

(2) Assume that $l'\in -\calS'$ and $c^{l'}$ is not dominant.   Then the inequalities
in (\ref{eq:genericL}) are strict for any $\calL\in \im ( c^{l'})$.
\end{theorem}

\begin{proof} {\it (1)}
The two inequalities (and the corresponding equalities) are equivalent by Riemann--Roch. We will prove
the statement for $h^0$. For any $l$ and $\calL$ (by a cohomological exact sequence) one has
\begin{equation}\label{eq:ineqH}
h^0(Z,\calL)\geq h^0(Z-l,\calL(-l))\geq \chi(Z-l,\calL(-l)),
\end{equation} hence the inequality follows.
We need to show the opposite inequality for $\calL$ generic.
Clearly, if $h^0(Z,\calL)=0$, then the opposite inequality follows (take e.g. $l=Z$).
Hence, assume $h^0(Z,\calL)\not=0$. Then, as in the proof of Theorem \ref{th:dominant},
there exists $0\leq l<Z$ such that $h^0(Z,\calL)=h^0(Z-l,\calL(-l))$ and
$H^0(Z-l,\calL(-l))_{\reg}\not=\emptyset$. In this case $l'-l\in -\calS'$ by (\ref{eq:empty})
and (by Theorem \ref{th:dominant})
$h^1(Z-l,\calL(-l))=0$ as well.
Hence $h^0(Z,\calL)=\chi(Z-l,\calL(-l))\leq \max_{0\leq l\leq Z}\chi(Z-l,\calL(-l))$.

{\it (2)} Assume that $h^0(Z,\calL)= \max_{0\leq l\leq Z}
\chi(Z-l,\calL(-l))$. If the $\max$ at the right hand side can be  realized by
a certain $l_0>0$ then using (\ref{eq:ineqH}) for $l_0$  we get that $h^0(Z,\calL)= h^0(Z-l_0,\calL(-l_0))$,
hence $\calL$ has fixed components, that is,  $\calL\not\in \im (c^{l'})$. On the other hand, if the
$\max$ is realized only by $l=0$, then $c^{l'}$ is dominant by Theorem \ref{th:dominant}{\it (3)}.
\end{proof}

Since  $H^1(\widetilde{X},\calL)=\lim_{\leftarrow,Z}H^1(Z,\calL)$,
cf. \cite[Th. 11.1]{hartshorne},
we obtain the following.
\begin{corollary}\label{cor:generic}   For   $l'\in L'$  and
any $\calL\in\pic^{l'}(\widetilde{X})$
one has $h^1(\widetilde{X},\calL)\geq \chi(-l')-\min _{l\in L_{\geq 0}}
\,\chi(-l'+l)$. Equality holds whenever $\calL$ is generic in $\pic^{l'}(\tX)$.
Furthermore, if $l'\in-\calS'$ and  $c^{l'}$ is not dominant, then  $h^1(\widetilde{X},\calL)>  \chi(-l')-\min _{l\in L_{\geq 0}}
\,\chi(-l'+l)$ whenever $\calL\in \im(c^{l'})$.
\end{corollary}

\begin{example}
Assume that $l'=0$ and $h^1(\calO_Z)\not=0$. Then $c^0$ is not dominant, hence
$h^1(Z,\calL)\geq -\min _{0\leq l\leq Z} \chi(l)$
for any $\calL$, and $h^1(\calO_Z)\geq 1-\min _{0\leq l\leq Z} \chi(l)$.

Moreover, for generic $\calL\in \pic^0(Z)$ one has $h^1(Z,\calL)= -\min_{0\leq l\leq Z} \chi(l)$.
 This for $Z\gg 0$ and $\Gamma$  elliptic reads
as $h^1(\widetilde{X}, \calL)=0$;
this fact for minimally elliptic $\Gamma$  was proved by Laufer in
\cite{Laufer77}, and for arbitrary elliptic  case in  \cite{weakly}.
\end{example}

\begin{example}\label{ex:lvan} Consider the situation of Corollary \ref{cor:generic}.
For certain topological types one can find for any $l'$ explicitly a cycle $l_{min}\in L_{\geq 0}$
which realizes $\min_{l\in L_{\geq 0}}\chi( -l'+l)=\chi(-l'+l_{min})$.
Indeed, consider the construction $x\mapsto x+l=s(x)$ described in \ref{ss:Lauferseq}.
Since  $\chi$ is decreasing along the sequence,
$(*)$ $\chi(s(x))\leq \chi(x)$. Next, assume e.g. that the lattice has the property that
$\chi(l)\geq 0$ for all $l\in L_{\geq 0}$ (hence the graph is either rational or elliptic).
Then for any $s\in\calS'$ one has $(**)$ $\chi(s)\leq \chi(s+l)$ for all $l\in L_{\geq 0}$.

We claim that  for rational and elliptic singularities
  $\min_{l\in L_{\geq 0}}\chi( -l'+l)=\chi(s(-l'))$.

Indeed, by $(*)$ one has $\chi(-l'+l_{min})\geq \chi(s(-l'+l_{min}))$, and by the universal property of
the operator $s$ one also has $s(-l'+l_{min})\geq s(-l')$, hence by $(**)$ $\chi(s(-l'+l_{min}))
\geq \chi(s(-l'))$.

In particular, for rational and elliptic germs $h^1(\tX,\calL)=\chi(-l')-\chi(s(-l'))$ whenever
$\calL$ is generic.

See also Corollary \ref{cor:cohcyc},  where we prove for any $(X,o)$
the existence of a unique minimal cycle with the property of $l_{min}$.
\end{example}

\subsection{}In parallel to $\calS'_{dom}$ (see \ref{bek:DOMChern}),
Corollary \ref{cor:generic} indicates another subset of $L'$:
\begin{equation}\label{eq:VANDEF}
Van':=\{-l'\ |\ \chi(-l')\leq \chi(-l'+l)\ \mbox{for all $l\in L_{\geq 0}$}\}.\end{equation}
This indexes those  cycles $-l'$ for which
$h^1(\widetilde{X},\calL)=0$ for generic  $\calL\in \pic^{l'}(\widetilde{X})$.

For arbitrary line bundles $\calL\in\pic^{l'}(\widetilde{X})$ the existent vanishing theorems formulate
sufficient (but usually not necessary) criterions. E.g., $h^1(\widetilde{X},\calL)=0$
 for any $(X,o)$ whenever
$-l'\in Z_K+\calS'$ (this is the so-called Grauert-Riemenschneider vanishing)
\cite{GrRie,Laufer72,Ram}, or, for rational $(X,o)$ whenever $-l'\in \calS'$ (Lipman's Criterion)
\cite{Lipman}. Even so, Corollary \ref{cor:generic} provides a {\it necessary and sufficient} vanishing condition
for {\it generic line bundles},
which, surprisingly, is independent of the analytic structure of $(X,o)$.
$Van'$ lists precisely the corresponding Chern classes.

 For rational singularities (since $h^1(\widetilde{X},\calL)$ depends only on $c_1(\calL)$,
 cf. \cite[4.3.3]{trieste}), $h^1(\widetilde{X},\calL)=0$ for {\it any} line bundle with fixed $c_1(\calL)$ exactly
 when $-c_1(\calL)\in Van'$. This is not valid for more general singularities:
 $-l'\in Van'$ does not guarantee the
 vanishing $h^1(\widetilde{X}, \calL)=0$
  for non--generic (hence for arbitrary) bundles. E.g., in the elliptic case, $0\in Van'$, however $h^1(\widetilde{X},\calO_{\widetilde{X}})
 =p_g>0$.

Though  most of the statements of the next lemma will not be needed in this first part of the series
of articles,  for completeness and further references we list
some properties of $Van'$ (which can be compared e.g. with those from Lemma \ref{lem:DOMIN}).
Note that a  semigroup module structure of type {\it (iv)} usually is not
studied/observed in vanishing theorems.

\begin{lemma}\label{rem:min}
$Van'$ satisfies the following properties:

(i) $Van'\subset \{l'\,|\, (l', E_v)\leq 1 \ \mbox{for all $v$}\}$;
in general $Van'\not\subset \calS'$ (e.g. for rational singularities
each $E_v\in Van'$), furthermore $\calS'_{dom}\subset Van'$,

(ii) $0\in Van'$ iff $L$ is rational or elliptic,

(iii) $Van'$ is not necessarily a semigroup  ($2E_v\not \in Van'$ if $|\calv|>1$, cf. $(i)$),

(iv) $Van'$ is closed to the  $\calS'$--action,

(v) $Van'$ is $\min$--stable,

(vi) $Van'\setminus \calS'$ might have infinitely many elements
(e.g. if $E_v\in Van'$ then
$E_v+\calS'\subset Van'$ too),

(vii)  
$Van'$ is not necessarily in the first quadrant, however
  $Van'\cap L$ is in the first quadrant for a minimal resolution
  (hence for $\calL$ generic and with  $c_1(\calL)\in L$, the vanishing
 $h^1(\widetilde{X},\calL)$ implies $c_1(\calL)\leq 0$).

\end{lemma}

\begin{proof} For $(i)$ take $l=E_v$ in (\ref{eq:VANDEF}), and check $h^1(\calO(-E_v))=0$
for rational germs. For $(iv)-(v)$ repeat the arguments from the proof of  \ref{lem:DOMIN}.
For $(vii)$ note that if the graph consists of a $(-1)$ (resp. $(-2)$) vertex then $-E$
(resp. $-E/2$) is in $ Van'$. On the other hand,
if $-l'=x_1-x_2$,
where $x_1,x_2\in L_{\geq 0}$ have no common $E_v$ in their supports, then $\chi(-l')\leq \chi(-l'+x_2)$
 implies $\chi(-x_2)\leq 0$.
 But, in a  minimal graph if
 $\chi(-x)\leq 0$ and $x\geq 0$ then  $x= 0$. Indeed, take $E_v\subset |x|$ such that $(E_v,x)<0$. Then
 $\chi(-x+E_v)\leq \chi(-x)\leq 0$. If we continue the procedure, in the last step we get
 $\chi(-E_w)\leq 0$ for some $w$, a fact which can happen only if $E_w$ is a $(-1)$--curve.
 \end{proof}

\begin{remark}\label{rem:stratif}
In Theorem \ref{th:hegy} (see also Corollary \ref{cor:generic} too) the set of `generic' line bundles
$\calL\in \pic^{l'}(Z)$ which satisfy (\ref{eq:genericL}) with equality  is not explicit.
 There exists an  open Zariski set
for which (\ref{eq:genericL}) holds with equality,
but this  usually is not the complement of $\im (c^{l'})$.
In other words, the complement of $\im (c^{l'})$ might have a non--trivial
stratification according to the values of
$h^1(Z,\calL)$, and the Zariski open strata corresponds to the `generic' bundles of Theorem \ref{th:hegy}.

Indeed, take the graph $\Gamma_1$ from Example \ref{ex:notclosed}, and consider the splice quotient analytic structure on it (for details see e.g. \cite{NO17}). In particular, $p_g=3$.
Set $Z\gg 0$ (e.g. $Z=Z_K$), and $\calL:=\calO_{Z}(-Z_{min})$. Since
$h^1(\calO_{Z_{min}})=2$ and
$h^1(\widetilde{X},\calO(-Z_{min}))=1 $, one also has $h^1(Z,\calL)=1$. Note also that the maximal ideal
cycle $Z_{max}$ is $2Z_{min}$, hence $\calL\not\in \im (c^{-Z_{min}})$. On the other hand,
$\min\chi=\chi(Z_{min})=-1$, hence $h^1(Z,\calL_{gen})=0$
for generic bundles $\calL_{gen}\in \pic^{-Z_{min}}(Z)$.
Hence, the complement of $\im (c^{-Z_{min}})$ has a non--trivial
$h^1$--stratification.
\end{remark}

\subsection{The cohomology cycle of line bundles}\label{ss:cohcycle}
If $(X,o)$ is a singularity with $p_g>0$, then its cohomology cycle (associated
with a fixed resolution $\phi$)  is the unique minimal
cycle $Z_{coh}\in L_{>0}$ such that $p_g=h^1(Z_{coh},\calO_{\widetilde{X}})$.
We extend this definition as follows.

\begin{proposition}\label{prop:cohcyc} (a) Fix a line bundle $\calL\in \pic(\widetilde{X})$
with $h^1(\widetilde{X},\calL)>0$.  The set
$L_{\calL}:=\{l\in L_{>0}\, : \, h^1(l,\calL)=h^1(\widetilde{X},\calL)\}$
has a unique minimal element, denoted by $Z_{coh}(\calL)$, called
the cohomological cycle of $\calL$ (and of $\phi$). It has the property that
$  h^1(l,\calL)<h^1(\widetilde{X},\calL)$
for any $l\not\geq Z_{coh}(\calL)$ ($l>0$).

(b) Fix $Z>0$ and $\calL\in \pic(Z)$
with $h^1(Z,\calL)>0$.  The set
$L_{Z,\calL}:=\{l\in L, \ 0<l\leq Z\, : \, h^1(l,\calL)=h^1(Z,\calL)\}$
has a unique minimal element, denoted by $Z_{coh}(Z,\calL)$, called
the cohomological cycle of $(Z,\calL)$. It has the property that
$  h^1(l,\calL)<h^1(Z,\calL)$
for any $l\not\geq Z_{coh}(Z,\calL)$ ($0<l\leq Z$).
\end{proposition}
\begin{proof} The proof of  \cite[4.8]{MR},  valid for $\calO_{\widetilde{X}}$,
can be adopted to this situation as well.
\end{proof}
If $h^1(\widetilde{X},\calL)=0$, then by convention $Z_{coh}(\calL)=0$.

\begin{corollary}\label{cor:cohcyc}
(a) For any $l'\in L'$ consider the set
$$L_{l'}:=\{l_{min}\in L_{\geq 0}\ |\ \chi(-l'+l_{min})=
 \min_{l\in L_{\geq 0}}\chi( -l'+l)\}.$$
 Then $L_{l'}$ has a unique minimal element $Z_{coh}(l')$, which coincides with the
 cohomological cycle of any generic $\calL\in\pic^{l'}(\widetilde{X})$.

 (b) For any $Z>0$ and $l'\in L'$ consider the set
$$L_{Z,l'}:=\{l_{min}\in L, \ 0\leq l_{min}\leq Z, \ |\ \chi(-l'+l_{min})=
 \min_{0\leq l\leq Z,\, l\in L}\chi( -l'+l)\}.$$
 Then $L_{Z,l'}$ has a unique minimal element $Z_{coh}(Z,l')$, which coincides with the
 cohomological cycle of any generic $\calL\in\pic^{l'}(Z)$.
\end{corollary}
\begin{proof}
Combine Theorem \ref{th:hegy} and Proposition \ref{prop:cohcyc}.
\end{proof}
\begin{corollary}\label{cor:cohcyc2}
\begin{enumerate}
\item Elements of type $-l'+Z_{coh}(l')$ ($l'\in L'$) belong to $Van'$.
\item If $-l'\leq -l''$ then $-l'+Z_{coh}(l')\leq -l''+Z_{coh}(l'')$ as well.
Furthermore, if $-l'\leq -l''\leq -l'+Z_{coh}(l')$ then  $-l'+Z_{coh}(l')= -l''+Z_{coh}(l'')$.
\end{enumerate}
\end{corollary}
\begin{example}\label{ex:cogcyc}
Assume that $L$ is numerically Gorenstein (that is, $Z_K\in L$). Then
by \cite[Lemma 6]{KN} (and $\chi(l)=\chi(Z_K-l)$) one gets $Z_{coh}(l'=0)\leq Z_K/2$.
\end{example}

\subsection{The dimension of  $\im (c)$}\label{ss:genfiber} For an arbitrary  element $\calL$
 of the image
$\im (c:\eca^{l'}(Z)\to \pic^{l'}(Z))$ one has $\dim \im(c)+\dim c^{-1}(\calL)\geq \dim \eca^{l'}(Z)=(l',Z)$,
with equality whenever  $\calL$ is a  generic element {\it of the image} $\im(c)$.
This combined with Lemma \ref{lem:free}(b) gives the following.
\begin{proposition}\label{lem:dimIm}
For any $\calL\in \im(c^{l'})\subset \pic^{l'}(Z)$ one has
\begin{equation}\label{eq:imC}
h^1(Z,\calL) \geq h^1(\calO_Z)- \dim (\im (c^{l'}))= {\rm codim}(\im  (c^{l'})).\end{equation}
In (\ref{eq:imC}) equality holds whenever $\calL$ is
generic in the image of $c$ (that is, generic with the property  $H^0(Z,\calL)_{\reg}\not=\emptyset$).
This fact and  Theorem \ref{th:hegy} applied for the generic element of $\im (c)$ imply
\begin{equation}\label{eq:ineqIM} {\rm codim}(\im (c^{l'}))\geq \chi(-l')-\min _{0\leq l\leq Z}\,\chi(-l'+l).
\end{equation}
Furthermore, if $c^{l'}$ is not  dominant then the inequality in (\ref{eq:ineqIM}) is strict.
\end{proposition}
In general, the codimension of $\im(c)$ cannot be characterized topologically. Indeed, take e.g.
$l'=0$, then $\im(c)$ is a point with codimension $h^1(\calO_Z)$. Moreover, by Example \ref{ex:dimim},
the dimension of $\im(c)$ is not topological either.

\subsection{Upper bounds for $h^1(Z,\calL)$.}
Theorem \ref{th:hegy} and Corollary \ref{cor:generic} provide sharp lower bounds for
$h^1(Z,\calL)$ and $h^1(\tX,\calL)$. A possible  upper bound is given by the next proposition.

\begin{proposition}\label{prop:UPPERBOUND}
Fix $Z>0$ and an arbitrary  $\calL\in \pic(Z)$ with   $l'=c_1(\calL)\in -\calS'$.

(a) If $h^0(Z,\calL)=0$ then  $h^1(Z,\calL) \leq -\chi(Z)<  h^1(\calO_Z)$.

(b) If $H^0(Z,\calL)_{\reg}\not=\emptyset$ then $h^1(Z,\calL) \leq   h^1(\calO_Z)$.

(c) In general, if  $h^0(Z,\calL)\not =0$ then
\begin{equation}\label{eq:HegyL}
h^1(Z,\calL)\leq \max_{0\leq l\leq Z}\{\, h^1(\calO_{Z-l})+\chi(-l')-\chi(-l'+l)\,\}\leq
 h^1(\calO_Z)+\chi(-l')-\min_{0\leq l\leq Z}\,\chi(-l'+l).
 \end{equation}

In particular, by (\ref{eq:genericL}) and (\ref{eq:HegyL}), $h^1(Z,\calL)$ takes values in an interval of
length (at most) $h^1(\calO_Z)$.
\end{proposition}
Note that
$ h^1(\calO_Z)\leq \max_{0\leq l\leq Z}\{\, h^1(\calO_{Z-l})+\chi(-l')-\chi(-l'+l)\,\}$ (take $l=0$).
Hence {\it (b)} gives a better bound than {\it (c)} whenever
$H^0(Z,\calL)_{\reg}\not=\emptyset$. (Examples with $h^1(Z,\calL)\not\leq h^1(\calO_Z)$ exist even for $l'=0$,
see e.g. Example 8.2.4 in  part II \cite{NNII}, when we will
treat the generic analytic structures).

Furthermore, {\it (c)} for $l'=0$ reads as
$h^1(Z,\calL)\leq \max_{0\leq l\leq Z}\{\, h^1(\calO_{Z-l})-\chi(l)\,\}$, which for   $Z=Z_K\in L$ transforms
into $h^1(Z_K,\calL)\leq \max_{0\leq l\leq Z_K}\{\, h^1(\calO_{l})-\chi(l)\,\}$ (use $\chi(Z_K-l)=\chi(l)$).
\begin{proof} {\it (a)}
$h^1(Z,\calL)=-\chi(Z,\calL)=-\chi(Z)-(Z,l')\leq -\chi(Z)=-h^0(\calO_Z)+h^1(\calO_Z)$.

{\it (b)}
 Multiplication by a generic  $s\in H^0(Z,\calL)$ gives  an exact sequence of
sheaves $0\to \calO_Z\to \calL\to \calF\to 0$, where $\calF$ is Stein. Hence
$H^1(\calO_Z)\to H^1(Z,\calL)$ is onto and $h^1(Z,\calL)\leq h^1(\calO_Z)$.

{\it (c)}  If $l$ is the fixed divisor of $\calL$ supported on $E$, then from the exact sequence
$0\to \calL(-l)|_{Z-l}\to \calL\to \calL|_l\to 0$ we get $h^1(\calL)=h^1(Z-l,\calL(-l))-\chi(\calL|_l)$, and
$\calL(-l)|_{Z-l}$ has no fixed components.
Hence   $h^1(Z-l,\calL(-l))\leq h^1(\calO_{Z-l})$ by {\it (b)}.
\end{proof}

\begin{remark}\label{rem:OWY} The inequality  $h^1(Z,\calL)\leq h^1(\calO_Z)$, valid for the case
when $\calL$ has no fixed components,  has the following geometric
interpretation, cf. (\ref{eq:dimfiber}):
 $h^1(\calO_Z)-h^1(Z,\calL)={\rm codim}(c^{-1}(\calL)\subset \eca^{l'})\geq 0$.
The inequality for  $\calL=\calO(-l)$, $l\in L_{>0}$, was already proved in
\cite[Th. 3.1]{OWY14}.
\end{remark}

\subsection{The $h^1$--stratification of $\pic^{l'}(Z)$.}
Fix $Z>0$,  $l'\in -\calS'$ and $k\in \Z$ with
$$\chi(-l')-\min_{0\leq l\leq Z}\,\chi(-l'+l)\leq\  k\ \leq
 h^1(\calO_Z)+\chi(-l')-\min_{0\leq l\leq Z}\,\chi(-l'+l).$$

 \begin{definition}\label{def:BrillNoether}
 For any $l'$ and $k$ as above we set
 \begin{equation}\label{eq:BN}
 W_{l',k}:=\{\calL\in \pic^{l'}(Z)\,:\, h^1(Z,\calL)=k\}.\end{equation}
 \end{definition}
 From the semicontinuity lemma \ref{lem:semicont} we automatically have for the closure
 $\overline { W_{l',k}}$
  \begin{equation}\label{eq:BN2}
 \overline {W_{l',k}}\subset \{\calL\in \pic^{l'}(Z)\,:\, h^1(Z,\calL)\geq k\}.\end{equation}

 These sets constitute  the analogs of the Brill--Noether strata defined for projective curves
 by the Brill--Noether theory, see
 \cite{ACGH,Flamini} and the references therein.

Lemmas \ref{lem:dimIm} and \ref{lem:REGVAL} have the following consequences.
\begin{corollary}\label{cor:ImW} Fix $l'\in -\calS'$. Then
$\im (c^{l'})\subset \overline{ W_{l',{\rm codim}\,\im(c^{l'})}}$.
Furthermore, the set of critical bundles of $c^{l'}$  are included in
$\overline{ W_{l',{\rm codim}\,\im(c^{l'})+1}}$.
\end{corollary}

\begin{example}\label{ex:BN}
If the fibers of $c^{l'}$ over $\im(c^{l'})$ are not equidimensional, then $\im(c^{l'})$ consists of more
strata of type $W_{l',k}$ (see e.g. Example \ref{ex:notequidim}). But, even if the fibers over
$\im(c^{l'})$ are equidimensional, hence $\im(c^{l'})$ consists of only one stratum, it can happen that
$c^{l'}$ is not a (topological) locally trivial fibration over $\im(c^{l'})$, see e.g. Example \ref{ex:nonfibration}.
In particular, $c^{l'}$ over a strata $W_{l',k}$ usually is not a (topological) locally trivial fibration.
\end{example}


\section{`Multiple' structures. The `stable' $\im (c^{l'})$.}\label{s:multiple}

\subsection{Monoid structure of divisors}\label{ss:addDiv}
In this section we will exploit the additional natural additive structure
$s^{l'_1,l'_2}(Z):\eca^{l'_1}(Z)\times \eca^{l'_2}(Z)\to \eca^{l'_1+l'_2}(Z)$ ($l'_1,l'_2\in-\calS'$)
provided by the sum of the divisors. (Sometimes we will abridge  $s^{l'_1,l'_2}(Z)$ as $s$.)

\begin{lemma}\label{lem:domin}
$s^{l'_1,l'_2}(Z)$ is dominant and quasi--finite.
\end{lemma}
\begin{proof} An effective  divisor decomposes in finitely many ways, hence  the
quasi--finiteness follows. Since the dimensions of the source and the target
are equal, cf. Theorem \ref{th:smooth},
$s$ is dominant.
\end{proof}
In general, $s$ is not surjective. E.g.,
in Example \ref{ex:nonfibration}, the elements of $c^{-1}(E_1\cap E_2)=
\bC^*$ are not in the image of $s^{E_1^*,E_2^*}(Z)$.

There is a parallel multiplication $\pic^{l'_1}(Z)\times \pic^{l'_2}(Z)\to \pic^{l'_1+l'_2}(Z)$,
$(\calL_1,\calL_2)\mapsto \calL_1\otimes\calL_2$. Clearly, $c^{l_1'+l'_2}\circ s^{l'_1,l'_2}= c^{l'_1}
\otimes c^{l'_2}$ in $\pic^{l'_1+l'_2}$.
In the next discussions we  replace  $c^{l'}$ by the composition
$$\widetilde{c}^{l'}:\eca^{l'}(Z)\stackrel{c^{l'}}{\longrightarrow} \pic^{l'}(Z)\stackrel{\calO_Z(-l')}
{\longrightarrow} \pic^0(Z),$$
where the second map is the multiplication by the natural line bundle $\calO_Z(-l')$.
Since $\calO_Z(l'_1+l'_2)=\calO_Z(l'_1)\otimes \calO_Z(l'_2)$ we also have
 $\widetilde{c}^{l_1'+l'_2}\circ s^{l'_1,l'_2}= \widetilde{c}^{l'_1}
\otimes \widetilde{c}^{l'_2}$ in $\pic^{0}$.
After identification of $\pic^0$ with (the additive) $H^1(\calO_Z)$, this reads as
 $\widetilde{c}^{l_1'+l'_2}\circ s^{l'_1,l'_2}= \widetilde{c}^{l'_1}
+ \widetilde{c}^{l'_2}$ in $H^1(\calO_Z)$.
The advantage of this new map is that it collects all the images of the effective Cartier divisors in a single
vector space $H^1(\calO_Z)$. Lemma \ref{lem:domin} and
the construction imply
\begin{equation}\label{eq:addclose}
\im (\widetilde{c}^{l_1'})+ \im (\widetilde{c}^{l_2'})\subset
\im (\widetilde{c}^{l_1'+l_2'})\subset \overline{
\im (\widetilde{c}^{l_1'})+ \im (\widetilde{c}^{l_2'})}.
\end{equation}
\begin{definition}\label{def:VZ}
For any  $l'\in-\calS'$  let $A_Z(l')$ (if there is no confusion, $A(l')$)
be the smallest dimensional affine subspace of
$H^1(\calO_Z)$ which contains $\im (\widetilde{c}^{l'})$. Let $V_Z(l')$ be the parallel vector subspace
of $H^1(\calO_Z)$,
the translation of  $A_Z(l')$ to the origin. 
\end{definition}
\begin{remark}\label{rem:Vineq}
From this definition follows that $\dim V_Z(l')$ is greater than or equal to the dimension of
the Zariski tangent space at any $\calL\in \overline { \im (c^{l'}(Z))}$; in particular,
$\dim V_Z(l'))\geq \dim \im (c^{l'}(Z))$. Hence, by (\ref{eq:imC}) one also has
$\dim V_Z(l')\geq h^1(\calO_Z)-h^1(Z,\calL)$ for any $\calL\in \im (c^{l'}(Z))$.
\end{remark}
\begin{example} \label{ex:supportI}
In general, $\im(\widetilde{c}^{l'})\varsubsetneq A_Z(l')$; take e.g. the first case of Example \ref{ex:notclosed}, when
 $\dim \im (c^{l'})=1$ and $A_Z(l')=\C^2$.
 (The fact that $A_Z(l')=\C^2$ can be deduced in the following way as well.
 $c^{nl'}$ is dominant for
$n\gg 1$, hence  $A_Z(nl')=\bC^2$. But  $V_Z(l')=V_Z(nl')$, see e.g. the next Lemma.)
\end{example}
Using (\ref{eq:addclose}) one obtains the following properties of the spaces $\{A_Z(l')\}_{l'}$ of $H^1(\calO_Z)$:
\begin{lemma}\label{lem:propAZ} \

(a) \ $A_Z(l_1'+l_2')= A_Z(l_1')+A_Z(l_2'):= \{a_1+a_2\,:\, a_i\in A_Z(l'_i\}$;
in particular,
$V_Z(l_1')\subset V_Z(l_2')$ whenever $l_1'\leq l_2'$ and
  $V_Z(nl')=V_Z(l')$ for  any $n\geq 1$.

(b) \  For any $-l'=\sum_va_vE^*_v\in \calS'$ let the $E^*$--support of $l'$ be
$I(l'):=\{v\,:\, a_v\not=0\}$. Then  $V_Z(l')$ depends only on $I(l')$.

E.g., if $I(l')=\calv$, then $c^{nl'}$ is dominant for any $n\gg 1$ (use Theorem \ref{th:dominant}(3).)
Hence,  $V_Z(l')=V_Z(nl')=H^1(\calO_Z)$.
\end{lemma}
\begin{proof}
{\it (b)}
$V_Z(l')\subset V_Z(l'+nE^*_v)\subset V_Z(l')+V_Z(nE^*_v)\subset V_Z(l')+V_Z(E^*_v)\subset V_Z(l')$
for $v\in I(l')$.
\end{proof}
\begin{definition}\label{def:VI} (a)
\ref{lem:propAZ}{\it (b)}  motivates to use the notation $V_Z(I)$ for $V_Z(l')$ whenever $I=I(l')$.

Hence Lemma \ref{lem:propAZ}(a) reads as $V_Z(I_1\cup I_2)=V_Z(I_1)+V_Z(I_2)$.

(b) If $Z_2\geq Z_1$, then the restriction (cf. \ref{ss:efcart}) satisfies $r(V_{Z_2}(l'))=V_{Z_1}(l')$,
hence $\dim V_{Z_2}(l')\geq \dim V_{Z_1}(l')$ and the pair
$V_Z(l')\subset H^1(\calO_Z)$ stabilizes as $Z$ increases.
Set $ \big(V_{\widetilde{X}}(l')\subset H^1(\calO_{\widetilde{X}})
\big)$ for  $\lim_{\leftarrow} \big(V_Z(l')\subset H^1(\calO_Z)\big)$
 and
$\big(V_{\widetilde{X}}(I)\subset H^1(\calO_{\widetilde{X}})\big):=\lim_{\leftarrow} \big(V_Z(I)\subset H^1(\calO_Z)\big)$.
\end{definition}

\begin{remark}
$\widetilde{c}^{l'}:\eca^{l'}(Z)\to \pic^0(Z)=H^1(\calO_Z)$ has a very strong
 hidden  rigidity property as well. Assume e.g. that $Z\geq E$ and $\eca^{l'}(Z)$  is
 1--dimensional. Then $\eca^{l'}(Z)$ can be identified with some
 $E_v^{reg}:=E_v\setminus \cup_{w\not=v}E_w$. Therefore,
the symmetric product $\eca^{l'}(Z)^{\times n}/\mathfrak{S}_n$ (where
$\mathfrak{S}_n$ is the permutation group of $n$ letters) embeds as a Zariski open set into $\eca^{nl'}(Z)$.
Hence, by Lemma \ref{lem:free}, the generic fibers of the
restriction of $\widetilde{c}^{nl'}$ ($ \eca^{l'}(Z)^{\times n}/\mathfrak{S}_n \to H^1(\calO_Z)$,
$[D_1, \cdots, D_n]\mapsto \sum_i \widetilde{c}^{l'}(D_i)$) must be irreducible.
This fact imposes serious restrictions for a map to be equal to some  $\widetilde{c}^{l'}$.

E.g., $\bC\to \bC^2$, $t\mapsto (t,t^4)$ is not birational equivalent with
 a certain  $\widetilde{c}^{l'}$.
Indeed, its `double', $\bC^{\times 2}/\mathfrak{S}_2\to \bC^2$,
$(t,s)\mapsto (t+s, t^4+s^4)$, rewritten in terms of elementary symmetric functions
reads as $\bC^2\to \bC^2$, $(\sigma_1,\sigma_2)\mapsto (\sigma_1,
\sigma_1^4-4\sigma_2\sigma_1^2+2\sigma_2^2)$, which has non--irreducible
generic fibers.
\end{remark}

By the next theorem, $V_Z(l')=H^1(\calO_Z)$ if and only if $c^{nl'}$ is dominant for $n\gg 1$; and in
\ref{ex:dimimzero} we will characterize those cases when $V_Z(l')=0$. But besides these two limit situations
the construction provides a rather complex {\it linear subspace
arrangement} $\{V_Z(l')\}_{l'}$, which, in general, contains  deep analytic information about $(X,o)$.

\begin{theorem}\label{th:mult}
Fix $l'\in-\calS'$ and $Z>0$ as above.  Then for $n\gg 1$ the following facts hold.

(a) The image of  $\widetilde{c}^{nl'}$ is the 
affine subspace $A_Z(nl')$ of $H^1(\calO_Z)$ (a translated of $A_Z(l')$).

(b) All the (non--empty) fibers of $\widetilde{c}^{nl'}$  have the same dimension.

In particular, for any $\calL\in \pic^{nl'}(Z)$ without fixed components (and $n\gg 1$) one has
\begin{equation}\label{eq:HD}
h^1(Z,\calL)=h^1(\calO_Z)-\dim  V_Z(l')= {\rm codim} \big(V_Z(l')\subset H^1(\calO_Z)\big).
\end{equation}

(c) Let $I\subset \calv$ be the $E^*$--support of $l'$.
Decompose $Z$ as $Z|_I+Z|_{\calv\setminus I}$ according to the supports $I$ and $\calv\setminus I$.
Then for all  $\calL\in  \pic^{nl'}(Z)$ without fixed components (and $n\gg 1$) $h^1(Z,\calL)$
depends only on the $E^*$--support $I$ of\, $l'$:
\begin{equation}\label{eq:support}
h^1(Z,\calL)=h^1(\calO_{Z|_{\calv\setminus I}}).
\end{equation}
 Hence, by (\ref{eq:HD}),
\begin{equation}\label{eq:HD2}
\dim V_Z(I)=h^1(\calO_Z)-h^1(\calO_{Z|_{\calv\setminus I}}).
\end{equation}
In particular, if $(\widetilde{X}/E_{\calv\setminus I},o_{\calv\setminus I})$ denotes the multi--germ (the disjoint union of
singularities) obtained by
contracting the connected components of $E_{\calv\setminus I}$ in $\widetilde{X}$, then
for $Z\gg 0$ we obtain
\begin{equation}\label{eq:additivity}
\dim V_Z(I)=p_g(X,o)-p_g(\widetilde{X}/E_{\calv\setminus I},o_{\calv\setminus I}).
\end{equation}
Therefore,
 $V_Z(I)=H^1(\calO_Z)=\bC^{p_g(X,o)}$, if and only if
$\Gamma\setminus I$ is a disjoint union of rational graphs.

(d) With the notations of (c), $V_Z(I)={\rm ker} ( H^1(\calO_Z)\to H^1(\calO_{Z|_{\calv\setminus I}}))$.

(e) Any $\calL\in \pic^{nl'}(Z)$ without fixed components is generated by global sections.
\end{theorem}
\begin{remark}
(a) In (\ref{eq:HD}) $h^1(Z,\calL)>-\chi(Z,\calL)$ (since $h^0(Z,\calL)>0$),
which gives a topological lower bound for ${\rm codim} \big(V_Z(l')\subset H^1(\calO_Z)\big)$.

(b) (\ref{eq:additivity}) generalizes the `$p_g$--additivity formula' of Okuma   \cite{Ok},
which was proved for splice quotient singularities, for details  see \ref{ss:APP1}.
Note that the present formula is valid for any singularity.

(c) Part {\it (a)}
 of Theorem \ref{th:mult} is equivalent (by a similar argument as the proof of Lemma
\ref{lem:propAZ}{\it (b)})
by the following statement: (a$'$) If $-l'=\sum_{v\in I}a_vE^*_v$ with $a_v\gg 0$
(but  no other relations between them), then
the image of $\widetilde{c}^{l'}$ is an affine subspace, a translated of $V_Z(I)$.

(d) Parts {\it (b)--(c)} of Theorem \ref{th:mult}  imply that $\im (c^{nl'})$ (for $n\gg 1$) is closed and
consists of only one $h^1$--strata:
 $\im (c^{nl'})=W_{nl', h^1(\calO_{Z})-\dim V_Z(I)}$.
\end{remark}

\begin{proof}[Proof of Theorem \ref{th:mult}]
{\it (a)} Write $A(l')$ as $a+V(l')$ for some $a\in A(l')$. Then by (\ref{eq:addclose})
$\im(\widetilde{c}^{nl'})\subset  na +V(l')$.
We have to show that for $n\gg 0$ we have equality
$\im(\widetilde{c}^{nl'})= na +V(l')$.

We choose smooth points $x_1,\ldots, x_k$ in $\im (\widetilde{c}^{l'})$ such that the tangent
spaces $T_{x_i} \im (\widetilde{c}^{l'})$, translated to the origin, generate $V(l')$. Then taking Zariski
neighborhoods $U_i$ of $x_i$ in $\im (\widetilde{c}^{l'})$, we notice that
$\sum_i(-x_i+U_i)$ contains a Zariski open set of $V(l')$. But $\sum_i(-x_i+U_i)\subset \sum_i
(-x_i+\im (\widetilde{c}^{l'}))\subset -\sum_ix_i+\im (\widetilde{c}^{kl'})\subset V(l')$, hence
$-\sum_ix_i+\im (\widetilde{c}^{kl'})$ contains a Zariski open subset of $V(l')$. On the other hand,
if $U$ is a Zariski open set of a vector space $V$, then $U+U=V$. This shows that
$\im (\widetilde{c}^{2kl'})$ is an affine space associated with $V(l')$.

{\it (b)} If we replace $l'$ by some multiple if it, by part {\it (a)} we can assume that $
\widetilde{c}^{l'}:\eca^{l'}(Z)\to H^1(\calO_Z) $ has image $A(l')$. Consider the following diagram
(for some $m\in\bZ_{>0}$ which will be determined later):

\begin{picture}(200,60)(-50,-5)
\put(50,40){\makebox(0,0)[l]{$
(\,\eca^{l'}(Z)\,)^m\ \stackrel{s}{\longrightarrow} \ \eca^{ml'}(Z)$}}
\put(50,5){\makebox(0,0)[l]{$
\ \ \ \  (\,A(l')\,)^m\ \, \ \stackrel{\Sigma}{\longrightarrow} \ \  A(ml')$}}
\put(80,22){\makebox(0,0){$\downarrow$}}
\put(155,22){\makebox(0,0){$\downarrow$}}
\put(65,22){\makebox(0,0){$\oplus\, \widetilde{c}^{l'}$}}
\put(170,22){\makebox(0,0){$\widetilde{c}^{ml'}$}}
\end{picture}

\noindent Fix any $x\in A(ml')$. Since $\oplus\, \widetilde{c}^{l'}$ and $\Sigma$ are surjective,
the fiber $(\widetilde{c}^{ml'})^{-1}(x)$ intersects $\im (s)$ at some point $p$.
Since the source and target spaces of $s$ are
smooth of the same dimension, by Open Mapping Theorem
(see e.g. \cite[p. 107]{GR})
there exists an (analytic) open neighbourhood $U$ of $p$
(hence intersecting the fiber) contained in $\im (s)$. Hence, using also the quasi--finiteness of  $s$,
 $\dim(\widetilde{c}^{ml'})^{-1}(x)=\dim (\widetilde{c}^{ml'}\circ s)^{-1}(x)=
 \dim ( \Sigma\circ \oplus \,\widetilde{c}^{l'})^{-1}(x)$.
Thus, if ${\bf x}=(x_1,\ldots, x_m)$ are the coordinates in $(\,A(l')\,)^m$, then we have to analyse
the set $(\oplus \,\widetilde{c}^{l'})^{-1}\{{\bf x}:\sum_i x_i=x\}$ for any fixed $x$.

In $A(l')$ there is a Zariski open subset $U$,
 with the following two properties:

(i) for any $y\in U$, the fiber $(\widetilde{c}^{l'})^{-1}(y)$ has the minimal possible dimension,
namely $\dim \eca^{l'}(Z)-\dim A(l')=(l',Z)-d(l')$;

(ii) if  $F:=A(l')\setminus U$ is its complement, then $\dim (\widetilde{c}^{l'})^{-1}(F)<\dim \eca^{l'}(Z)
=(l',Z)$.

We stratify $H_x:=\{{\bf x}:\sum_ix_i=x\}$ with the sets $\calF_k:=\{{\bf x}\in H_x\,:\,
\# \{i: x_i\in F\}=k\}$, where $0\leq k\leq m$. Set also $E\calF_k:=(\oplus\, \widetilde{c}^{l'})^{-1}(\calF_k)$.

Then $\calF_0$ is a non--empty open set of $H_x$ of dimension $(m-1)d(l')$, hence
$\dim E\calF_0=(m-1)d(l')+m((l',Z)-d(l'))=(ml',Z)-d(l')$.
Next we estimate the dimensions of the other strata as well.

First, we consider the case $1\leq k<m$. Then $\calF$ is covered by several components  according to the
position of $I=\{i_1,\ldots, i_k\}$ indexing those $x_i$ which belong to $F$. Fix suxh a component
$\calF_{k,I}$, and write $(\oplus\, \widetilde{c}^{l'})^{-1}(\calF_{k,I})=E\calF_{k,I}$.
We consider the projection $pr_I:\calF_{k,I}\to \sqcap_I F$, ${\bf x}\mapsto (x_{i_1},\ldots, x_{i_k})$,
and the lifted one $Epr_I:E\calF_{k,I}\to \sqcap_I(\widetilde{c}^{l'})^{-1}(F)$. Note that
$Epr_I$ is an injection and its target has dimension $\leq k((l',Z)-1)$.
Furthermore, the fibers of $Epr_I$ have dimension $(m-k-1)d(l')+(m-k)((l',Z)-d(l'))=(m-k)(l',Z)-d(l')$.
Hence, $\dim E\calF_{k,I}\leq (m-k)(l',Z)-d(l')+ k((l',Z)-1)=(ml',Z)-d(l')-k$.

The case $k=m$ is slightly different. Using the injection
$\calF_{m}\to \sqcap_m(\widetilde{c}^{l'})^{-1}(F)$ we get `only' $\dim E\calF_m\leq m((l',Z)-1)$.
Therefore, if $m\geq d(l')$ then we get $\dim E\calF_m\leq \dim E\calF_0$. Hence, finally,
$\dim (\widetilde{c}^{ml'})^{-1}(x)=\dim E\calF_0=\dim \eca^{ml'}(Z)-\dim A(ml')$.

For (\ref{eq:HD}) use part {\it (b)} and Lemma \ref{lem:dimIm}.

{\it (c)}  For any $n\gg 1 $ and $\calL\in \im(c^{nl'})$  (\ref{eq:HD}) gives
$h^1(Z,\calL)=h^1(\calO_Z)-d_Z(l')$.
By Grauert--Riemenschneider vanishing theorem $h^1(Z|_{I}, \calL(-Z|_{\calv\setminus I}))=0$, hence
$h^1(Z,\calL)=h^1(Z|_{\calv\setminus I},\calL)$. If $\calL$ is associated with certain effective
divisor $D\in\eca^{nl'}(Z)$ (as the image of $c^{nl'}$), then $\calL|_{Z|_{\calv\setminus I}}$ is
associated with the restriction of this divisor to $Z|_{\calv\setminus I}$. But this restriction has an
empty support, hence $\calL|_{Z|_{\calv\setminus I}}$ is the trivial bundle over $Z|_{\calv\setminus I}$.

{\it (d)} Since the restriction of any element of $\eca^{nl'}(Z)$ to $Z|_{\calv\setminus I}$ is the empty divisor, the image of the composition $\eca^{nl'}(Z) \to
\eca^0(Z|_{\calv\setminus I})\to \pic^0(Z|_{\calv\setminus I})$ is the trivial bundle (that is, the zero element of $\pic^0(Z_{\calv\setminus I})$). Therefore, $\im (c^{nl'})\subset
{\rm ker} ( H^1(\calO_Z)\to H^1(\calO_{Z|_{\calv\setminus I}}))$. Since they have the same dimension
(cf. \ref{eq:HD2}) they must agree.

{\it (e)} Let $n$ be so large that $\im (\widetilde{c}^{nl'})=A_Z(nl')$ is an affine subspace. We claim that
any $\calL\in \im (\widetilde{c}^{2nl'})=A_Z(2nl')$ is generated by global sections.
Indeed, fix such a bundle  and one of its sections
$s\in H^0(Z,\calL)$ whose divisor is an
element of $\eca^{2nl'}(Z)$, whose support with reduced structure is   ${\bf p}:=\{p_1,\ldots, p_k\}\subset E$.
Let $\eca^{nl'}_p(Z)$ be the subspace of $\eca^{nl'}(Z)$ consisting of divisors supported in the complement of
${\bf p}$. This is a Zariski open set of
$\eca^{nl'}(Z)$, hence $c(\eca^{nl'}_p(Z))$ contains  a Zariski open set $U$ in $A_Z(nl')$.
Then  $U+U=A_Z(2nl')$, hence $\calL$ admits  a section whose  divisor
has support  off ${\bf p}$.
\end{proof}

\subsection{Cohomological reinterpretations of $ V_Z(l')$.}\label{ss:GlSec}
Fix $\calL\in \im (c^{nl'})$ ($n\gg 1$), $D\in (c^{nl'})^{-1}(\calL)$, and
$s\in H^0(Z, \calL)$ without fixed components. Then, as in the situation of \ref{ss:GlSecA} one has the cohomological long exact sequence
$H^0(Z,\calL)\stackrel{R_{\calL}}{\longrightarrow} \calO_D
\stackrel{\delta}{\longrightarrow} H^1(\calO_Z)\to H^1(Z,\calL)\to 0$
from (\ref{eq:ML}).
Then by Theorem \ref{th:mult}, $\im(c^{nl'})=A(nl')$. Therefore,
$\im (T_Dc^{nl'})\subset T_{\calL}A(nl')$.
But, by Lemma \ref{lem:free},
$\dim \im T_Dc^{nl'}=\dim \eca^{nl'}(Z)-\dim\im (c^{nl'})^{-1}(\calL)=
h^1(\calO_Z)-h^1(Z,\calL)=\dim T_{\calL}A(nl')=\dim V_Z(l')$.
Hence,  $\im (T_D\widetilde{c}^{nl'})=V_Z(l')$. As $\im (T_D\widetilde{c}^{nl'})=\im \delta$ (cf.
Prop. \ref{lem:Mumford}) for $V_Z(l')$ we get two other cohomological reinterpretations.
Either  it is
the Artin algebra  $\calO_D/\im(R_{\calL})$, as a vector space,
 identified as the image of $O_D$ into $H^1(\calO_Z)$,
or it is also the kernel of
$H^1(\times s):H^1(\calO_Z)\to H^1(Z,\calL)$.

In other words, for $n\gg 1$, the image of $\calO_D\to H^1(\calO_Z)$ is independent of the choice of $D$, while the kernel of
$H^1(\times s):H^1(\calO_Z)\to H^1(Z,\calL)$
is independent of the choice of $s$. Furthermore,   they  are equal,  and in fact
 this subspace of $H^1(\calO_Z)$ depends only on the
$E^*$--support $I$ of $l'$, and it equals $V_Z(I)$.

There is a parallel analogous discussion for $\widetilde{X}$ (instead of $Z$) as well (in that case
the reduced structure of $D$ is Stein, hence $h^1(\calO_D)=0$ again).

\subsection{Example. Characterization of the cases $\dim \im (c)=0$}\label{ex:dimimzero}
Fix $l'\in-\calS'$ with $E^*$--support $I\subset \calv$ and $Z>0$ as above.
Using (\ref{eq:dimfiber}) and (\ref{eq:HD2}) one proves that the following facts are equivalent
(for an additional equivalent property  see also Example \ref{ex:CONTDIMzero}):

(i) \ $ \im (c^{l'})$ is a point (or, $V_Z(l')=0$);

(ii) \ there exists $\calL\in \pic^{l'}(Z)$  without fixed components such that
$h^1(Z,\calL)=h^1(Z)$;

(iii) \  any  $\calL\in \pic^{l'}(Z)$  without fixed components satisfies
$h^1(Z,\calL)=h^1(Z)$;

(iv) \ all line bundles  $\calL\in \pic^{l'}(Z)$  without fixed components are isomorphic to each other;

(v) \ $h^1(\calO_Z)=h^1(\calO_{Z|_{\calv\setminus I}})$.

\vspace{2mm}

Let us define $\calS'_{pt}$ as $\{-l'\in\calS'\,:\,   \im (c^{l'}) \ \mbox{is a point}\}\subset \calS'$,
this is the set of Chern classes satisfying the above equivalent conditions.
Using (\ref{eq:addclose})  we obtain  that $\calS'_{pt}$ is a semigroup.

Part (v) via  Proposition \ref{prop:cohcyc} reads as follows:
\begin{equation}\label{eq:Spt}
\calS'_{pt}=\bZ_{\geq 0}\langle \,E^*_v \ |\ E_v\not\subset |Z_{coh}(Z,\calO_Z)|\, \rangle.
\end{equation}
Note that (in contrast with $\calS'_{dom}$) $\calS'_{pt}$ is not topological.
Indeed, take e.g. the graph from Example \ref{ex:dimim}, $-l':=Z_{min}=E^*_v$ (where $v$ is the $(-2)$--vertex
adjacent with the $(-7)$ vertex),  and set $Z=Z_K$.
Then, if $p_g(X,o)=2$ (that is, $(X,o)$ is Gorenstein) then $Z_{coh}(Z,\calO_Z)=Z$, and $\calS'_{pt}=\{0\}$.
If $p_g(X,o)=1$, then $Z_{coh}(Z,\calO_Z)$ is the minimally elliptic cycle, and $\calS'_{pt}=\bZ\langle E^*_v\rangle$.

In \cite{OWY14,OWY15a,OWY15b} a cycle $l\in \calS'\cap L$ is called  $p_g$--cycle if $\calO_{\widetilde{X}}(-l)\in\pic(\widetilde{X})$
has no fixed components, and $h^1(\widetilde{X},\calO_{\widetilde{X}}(-l))=p_g$. Note that this in our language means that
$-l \in \calS'_{pt}$ for $Z\gg 0$.
Our results generalizes several statements of [loc.cit.] for arbitrary bundles $\calL $ without fixed components
(replacing $\calO_{\widetilde{X}}(-l)$)
and arbitrary $\dim \im (c^{l'})$.

This particular case and several similar classical results valid for bundles of type $\calO(l')$
motivate to investigate the position of the natural line  bundles with respect to $\im (c^{l'})$ (i.e., whether
$\calO(l')$ has fixed components or no). This is the subject of section \ref{s:Hilb}.

\section{The Abel map  via differential forms}\label{s:ADIFFFORMS}

\subsection{Review of Laufer Duality \cite{Laufer72},
\cite[p. 1281]{Laufer77}}\label{ss:LD} \
Following Laufer, we identify the dual space $H^1(\tX,\cO_{\tX})^*$ with the space of global holomorphic
2-forms on $\tX\setminus E$ up to the subspace of those forms which can be extended
holomorphically over $\tX$.

For this, use first  Serre duality
$H^1(\tX,\cO_{\tX})^*\simeq
H^1_c(\tX,\Omega^2_{\tX})$. Then,   in the exact sequence
$$0\to H^0_c(\tX,\Omega^2_{\tX}) \to
 H^0(\tX,\Omega^2_{\tX}) \to
H^0(\tX\setminus E,\Omega^2_{\tX})\to
H^1_c(\tX,\Omega^2_{\tX})\to
H^1(\tX,\Omega^2_{\tX})$$
$H^0_c(\tX,\Omega^2_{\tX})=H^2(\tX,\cO_{\tX})^*=0$ by dimension argument, while $H^1(\tX,\Omega^2_{\tX})=0$ by the Grauert--Riemenschneider
vanishing. Hence,
\begin{equation}\label{eq:LD}
H^1(\tX,\cO_{\tX})^*\simeq  H^1_c(\tX,\Omega^2_{\tX})\simeq
H^0(\tX\setminus E,\Omega^2_{\tX})/ H^0(\tX,\Omega^2_{\tX}).\end{equation}

The second  isomorphism can be realized as follows.
 Fix a small tubular neighbourhood $N\subset \tX$
of $E$ such that its closure is compact in $\tX$. Take any
$\omega\in H^0(\tX\setminus E,\Omega^2_{\tX})$, and extend the restriction $\omega|_{\tX\setminus N}$ to a
$C^\infty(2,0)$--form $\tilde{\omega}$ on $\tX$. Then $\bar{\partial}\tilde{\omega}$ is a compactly supported
$C^\infty(2,1)$--form, $\bar{\partial}\bar{\partial}\tilde{\omega}=0$,
 hence $\bar{\partial}\tilde{\omega}$ determines a class in $H^1_c(\tX,\Omega^2)$.
If $\tilde{\omega} $ is a holomorphic extension then $\bar{\partial}\tilde{\omega}=0$.
Next, let $\lambda$ be a $C^\infty(0,1)$ form in $\tX$.
Then the  duality $ H^1(\tX,\cO_{\tX})\otimes H^1_c(\tX,\Omega^2)\to \bC$ is the perfect pairing
$$\langle [\lambda],[\bar{\partial}\tilde{\omega}]\rangle=\int_{\tX}\
\lambda\wedge \bar{\partial}\tilde{\omega}.$$
Assume that the class $[\lambda]\in H^1(\tX,\cO_{\tX})$ is realized by a $\check{C}ech$ {\it cocyle}
$\lambda_{ij}\in \cO(U_i\cap U_j)$, where $\{U_i\}_i$ is an open cover of $E$,
$U_i\cap U_j\cap U_k=\emptyset$, and each connected component of the
intersections $U_i\cap U_j$ is either a coordinate bidisc $B=\{|u|<2\epsilon,\ |v|<2\epsilon\}$
 with coordinates $(u,v)$, such that
$E\cap B\subset  \{ uv=0\}$,  or a punctured
coordinate bidisc $B=\{\epsilon/2<|v|<2\epsilon,\ |u|<2\epsilon\}$
with coordinates $(u,v)$, such that
$E\cap B=  \{ u=0\}$.
 Then $\lambda $ is obtained as follows:
one finds $C^\infty$ functions $\lambda_i$ on $U_i$ such that $\lambda_i-\lambda_j=\lambda_{ij}$
on $U_i\cap U_j$, and one sets  $\lambda$ as $\bar{\partial} \lambda_i$ on $U_i$.
Then, by Stokes theorem
\begin{equation}\label{eq:Stokes}
\langle [\lambda],[\bar{\partial}\tilde{\omega}]\rangle=
\sum_{B}\ \int_{|u|=\epsilon, \ |v|=\epsilon} \lambda_{ij}\omega.
\end{equation}
By Stokes theorem,
if $\omega$ has no pole along $E$ in $B$, then the $B$--contribution in the above sum is zero.

\bekezdes \label{bek:Z}
Above $H^0(\tX\setminus E,\Omega^2_{\tX})$ can be replaced by
$H^0(\tX,\Omega^2_{\tX}(Z))$ for a large cycle $Z$
(e.g. for $Z\geq \lfloor Z_K \rfloor$). Indeed, for any cycle $Z>0$ from the exacts sequence of sheaves
$0\to\Omega^2_{\tX}\to \Omega^2_{\tX}(Z)\to \calO_{Z}(Z+K_{\tX})\to 0$ and from the vanishing
$h^1(\Omega^2_{\tX})=0$ and Serre duality one has
\begin{equation}\label{eq:duality}
H^0(\Omega^2_{\tX}(Z))/H^0(\Omega^2_{\tX})=H^0(\calO_Z(Z+K))\simeq H^1(\calO_Z)^*.
\end{equation}
Since $H^1(\calO_Z)\simeq H^1(\calO_{\tX})$ for $Z\geq \lfloor Z_K \rfloor$, the natural inclusion
\begin{equation}\label{eq:inclusion}
H^0(\Omega^2_{\tX}(Z))/H^0( \Omega^2_{\tX})\hookrightarrow
H^0(\tX\setminus E, \Omega^2_{\tX})/H^0(\Omega^2_{\tX})
\end{equation}
is an isomorphism.
\bekezdes \label{bek:transport}
The above duality,  via the isomorphism
$\exp: H^1(\tX,\cO_{\tX})\to c^{-1}(0)\subset H^1(\tX,\cO_{\tX}^*)=\pic(\tX)$, can be transported as follows.
Consider the following situation.
We fix a smooth point $p$ on $E$, a local bidisc $B\ni p$ with local coordinates $(u,v)$ such that $B\cap E=\{u=0\}$.
We assume that a certain form   $\omega\in H^0(\tX,\Omega^2_{\tX}(Z))$ has  local equation
$\omega=\sum_{ i\in\Z,j\geq 0}a_{i,j} u^iv^jdu\wedge dv$ in $B$.

In the same time,
we fix a divisor $\widetilde{D}$ on $\tX$, whose local equation in $B$ is $v^n$, $n\geq 1$.
 Let $\widetilde{D}_t$ be another divisor, which
is the same as $\widetilde{D}$ in the complement of $B$ and in $B$ its local equation is $(v+tu^{o-1})^n$,
where $o\geq 1$ and $t\in \bC$ (with $|t|\ll 1$ whenever $o=1$).

Next we will provide three type of   formulae.

The first one is the composition of several maps.
Note that the pairing $\langle \cdot , [\bar{\partial}\tilde{\omega}]\rangle$
(abridged as $\langle \cdot, \omega \rangle$) produces a map
$H^1(\tX, \calO_{\tX})\to \C$. Then we identify
$H^1(\tX, \calO_{\tX})$ with $\pic^0(\tX)$
by the exponential map. Then we consider the composition
$t\mapsto \widetilde{D}_t-\widetilde{D}
\mapsto \calO_{\tX}(\widetilde{D}_t-\widetilde{D})
\mapsto \exp^{-1} \calO_{\tX}(\widetilde{D}_t-\widetilde{D})
\mapsto \langle\exp^{-1} \cO_{\tX}(\widetilde{D}_t-\widetilde{D}),\omega\rangle$.
The first formula makes this composition explicit.
This restricted to any cycle $Z\gg0 $ can be reinterpreted as $\omega$--coordinate
of the Abel map
restricted to the path $t\mapsto D_t:=\widetilde{D}_t|_Z$
(and shifted by the image of $D:=\widetilde{D}|_Z$).

 The second formula  determines the tangent application of the above  composition (in this way it determines the $\omega$--coordinate of the tangent application of the Abel map restricted to $D_t$).

In the third formula we replace the path $D_t$ by a complete neighborhood of $D$ in $\eca(Z)$.

Note that if we consider --- instead of a single form $\omega$ ---
 a complete set of representatives of a basis of
$H^0(\tX,\Omega^2_{\tX}(Z))/ H^0(\tX,\Omega^2_{\tX})$,
then we get by the above three constructions  the restriction of the Abel map to the path $D_t$,
 the tangent map of this restriction, and  in the third case the `complete' Abel map
defined in some neighbourhood of $D$.

\subsection{The Abel map restricted to $D_t$}\label{ss:7.2}
The first two cases
 start with the explicit computation  of  $\langle\exp^{-1} \cO_{\tX}(\widetilde{D}_t-\widetilde{D}),\omega\rangle$, as follows.
$\widetilde{D}_t-\widetilde{D}$  is the divisor $\widetilde{D}'={\rm div}( (v+tu^{o-1})/v)^n$,
supported in $B=\{|u|,\, |v|<\epsilon\}$.
We can fix $\epsilon$ such that the support of $\widetilde{D}'$ is in $\{|v|<\epsilon/2\}$, and set
$B^*:=\{\epsilon/2<|v|<\epsilon,\, |u|<\epsilon\}$. Using the trivialization of $\cO(\widetilde{D}')$ in
$\tX\setminus \{|v|\leq \epsilon/2\}$ and the realization   $\cO(\widetilde{D}')$ on $B$, we get that $\cO(\widetilde{D}')$ can be
represented by the cocycle $g=( (v+tu^{o-1})/v)^n\in \cO^*(B^*)$. Therefore,
 $\log( (v+tu^{o-1})/v)^n=n\log(1+tu^{o-1}/v)$ is a cocycle in
$B^*$ representing its lifting into $H^1(\tX,\cO_{\tX})$. This  paired with $\omega$ gives:
\begin{equation}\label{eq:Tomega1}
\langle\langle \widetilde{D}_t,\omega\rangle\rangle:= \langle\exp^{-1} \cO_{\tX}(\widetilde{D}_t-\widetilde{D}),\omega\rangle=
n\int_{|u|=\epsilon, \ |v|=\epsilon} \log(1+t\frac{u^{o-1}}{v})\cdot
 \sum_{ i\in\Z,j\geq 0}a_{i,j} u^iv^jdu\wedge dv.\end{equation}
If $\omega_1,\ldots, \omega_{p_g}$ are representatives of  a basis for
$H^0(\tX,\Omega^2_{\tX}(Z))/ H^0(\tX,\Omega^2_{\tX})$, and $Z\gg 0$, then
\begin{equation}\label{eq:Tomega2}
\widetilde{D}_t\mapsto (\langle\langle \widetilde{D}_t,
\omega_1\rangle\rangle, \ldots,
\langle\langle \widetilde{D}_t,\omega_{p_g}\rangle\rangle)
\end{equation}
is the restriction of the Abel map to $\widetilde{D}_t$ (associated with $Z$,
and shifter by the image of $\widetilde{D}$).

At the level of tangent application on has the formula
for $(T_{\widetilde{c}(D)}\omega)\circ T_D\widetilde{c})(\frac{d}{dt} D_t|_{t=0})$:
\begin{equation}\label{eq:Tomega}
\frac{d}{dt}\Big|_{t=0}\, \Big[n\int_{|u|=\epsilon, \ |v|=\epsilon} \log(1+t\frac{u^{o-1}}{v})\cdot \sum_{ i\in\Z,j\geq 0}a_{i,j} u^iv^jdu\wedge dv\, \Big]=
\lambda\cdot a_{-o,0}\ \  (\lambda\in \bC^*).\end{equation}
If $\omega$ has no pole along the divisor $\{u=0\}$ then $\langle \exp^{-1}\cO_{\tX}(\widetilde{D}_t-\widetilde{D}),\omega\rangle=0$
for any path $\widetilde{D}_t$.

\begin{definition}\label{not:residue}
Consider the above situation in the bidisc $B$: $B\cap E=\{u=0\}$,
$\widetilde{D}$ has  local equation $v$ (i.e. $n=1$), and
$\omega=\sum_{ i\in\Z,j\geq 0}a_{i,j} u^iv^jdu\wedge dv$. Then we introduce the {\it Leray residue }
of $\omega/du$ along $\{v=0\}$ as the 1--form (with possible poles at $\widetilde{D}\cap E$) defined by
$(\omega/dv)|_{v=0}=\sum_i a_{i,0}u^i du$. We denote it by ${\rm Res}_D(\omega)$.
\end{definition}
Note that
the right hand side of (\ref{eq:Tomega}) tests exactly the pole part  of the Leray residue  ${\rm Res}_D(\omega)$.

\subsection{The Abel map}\label{ss:abelforms}
Assume as above that in the ball $B$ the divisor $\widetilde{D}$ is given by
$v=0$ (i.e. $n=1$), and its `perturbation' $\widetilde{D}(c)$ is given by
$v=c_0+c_1u+c_2u^2+\cdots$ with $|c_0|\ll \epsilon$.
Furthermore, assume that the form
$\omega$ in $B$  has the form $(f(v)/u^{\ell+1})du\wedge dv$, where $f\in\calO(B)$ and  $\ell\geq 0$. (Note that
 the Laurent expansion in variable $u$
of any differential form is a sum of such terms.)

Our aim is the computation of $\langle\langle \widetilde{D}(c),\omega\rangle\rangle$.

If $\{p_i\}_{\geq 1}$ (resp. $\{h_i\}_{i\geq 1}$ ) denote the power sum (resp. complete) symmetric polynomials (functions) then (cf. \cite[p. 23]{MacD})
\begin{equation} \label{eq:SYM1}
p_1u+p_2u^2/2+p_3u^3/3+\cdots =\log( 1+h_1u+h_2u^2+\cdots).
\end{equation}
Furthermore, by \cite[p. 28]{MacD}, for $n\geq 1$,
\begin{equation}\label{eq:SYM2}
(-1)^{n+1}p_n=\begin{pmatrix} h_1 & 1 & 0 & \ldots & 0 \\
2h_2 & h_1 & 1 & \ldots & 0\\
\vdots & \vdots  & \vdots &  & \vdots \\
nh_n & h_{n-1} & h_{n-2} & \ldots & h_1\end{pmatrix}
\end{equation}
We rewrite (\ref{eq:SYM1}) as $\log(A)+p_1u+p_2u^2/2+\cdots =
\log(A+h_1Au +h_2Au^2+\cdots)$ and we make the substitution
$A=(v-c_0)/v$, $h_1A=-c_1/v$, $h_2A=-c_2/v$, etc., and we obtain
\begin{equation}\label{eq:SYM3}
\log\big(1-\frac{c_0+c_1u+c_2u^2+\cdots}{v}\big)=
\log\big(1-\frac{c_0}{v}\big)+\delta _1(c)u+\delta_2(c)u^2+\cdots ,
\end{equation}
where for $n\geq 1$
\begin{equation}\label{eq:SYM4}
\delta_n(c)=
 \sum_{i=1}^n \frac{\delta_{n, i}(c)}{(v-c_0)^i}=
\frac{-1}{n}
\begin{pmatrix} \frac{c_1}{v-c_0} & -1 & 0 & \ldots & 0 \\
\frac{2c_2}{v-c_0} & \frac{c_1}{v-c_0} & -1 & \ldots & 0\\
\vdots & \vdots  & \vdots &  & \vdots \\
\frac{nc_n}{v-c_0} & \frac{c_{n-1}}{v-c_0}
 & \frac{c_{n-2}}{v-c_0} & \ldots & \frac{c_1}{v-c_0}
 \end{pmatrix}.
\end{equation}
Note that $\delta_{n,i}$ are certain universal polynomials in variables $c_1,
\ldots, c_n$.
Then $\langle\langle \widetilde{D}(c),\omega\rangle\rangle$ equals
\begin{equation}\label{eq:SYM5}
\int_{|u|=\epsilon, \ |v|=\epsilon} \log\Big(1-\frac{c_0+c_1u+\cdots }{v}\Big)\cdot
\frac{f(v)}{u^{\ell+1}}\,du\wedge dv=
\sum_{i=1}^{\ell}\
\frac{\delta_{\ell, i}(c)}{(i-1)!} \cdot
\frac{d ^{i-1} f}{d v ^{i-1} }(c_0).
\end{equation}
\subsection{Reduction to an arbitrary $Z>0$.}\label{ss:red7} Consider the above perfect pairing
$H^1(\tX,\calO_{\tX})\otimes H^0(\tX\setminus E, \Omega^2_{\tX})/H^0(\Omega^2_{\tX})\to \C$
given via integration
 of class representatives.  In $H^1(\tX,\calO_{\tX})$ let $A$ be the image of the $H^1(\tX,\calO_{\tX}(-Z))$, hence $H^1(\tX,\calO_{\tX})/A=H^1(\calO_Z)$. On the other hand,
 in $H^0(\tX\setminus E, \Omega^2_{\tX})/H^0(\Omega^2_{\tX})$ consider the subspace $B:=H^0(\Omega^2_{\tX}(Z))/H^0(\Omega^2_{\tX})$ of dimension $h^1(\calO_Z)$
(cf. (\ref{eq:duality}). Since $\langle A,B\rangle=0$, the pairing factorizes to a perfect
pairing $H^1(\calO_Z)\otimes H^0(\Omega^2_{\tX}(Z))/H^0(\Omega^2_{\tX})\to \C$.
It can be described by the very same integral form of the corresponding class representatives.

Moreover, if $\widetilde{D}_t$ is an 1--parameter family of divisors as in \ref{bek:transport},
representing an element in $H^1(\calO_Z)$ (via the surjection $H^1(\calO_{\tX})\to H^1(\calO_Z)$),
and $\omega$ is a representative of a class $[\omega]\in H^0(\Omega^2_{\tX}(Z))/H^0(\Omega^2_{\tX})$, then the expression of the
pairing
$H^1(\calO_Z)\otimes H^0(\Omega^2_{\tX}(Z))/H^0(\Omega^2_{\tX})\to \C$,
 $\langle\exp^{-1} \cO_{Z}(\widetilde{D}_t-\widetilde{D}),[\omega]\rangle$,
can be represented by the very same formula (\ref{eq:Tomega1}) (as in the case $Z\gg 0$). Furthermore, all other formulae of
subsections \ref{ss:7.2} and \ref{ss:abelforms} also have their extended versions.
E.g., (\ref{eq:Tomega}) gives
$T_{\widetilde{c}(D)}(\omega)\circ T_D\widetilde{c}^{l'}(Z))(\frac{d}{dt}D_t|_{t=0})$,
and (\ref{eq:SYM5}) is the $[\omega]$--coordinate of the Abel map $\eca^{l'}(Z)\to H^1(\calO_Z)$.

\section{The `stable' arrangement $\{V_{\widetilde{X}}(I)\}_{I\subset \calv}$ and differential forms}\label{s:DIFFFORMS}

\subsection{The arrangement $\{\Omega_{\tX}(I)\}_I$ of forms and its duality with $\{V_{\tX}(I)\}_I$}\label{ss:DU}
\begin{definition}\label{def:omegaI}
Let $\Omega_{\tX}(I)$ (or, $\Omega(I)$)  be the subspace of $H^0(\tX\setminus E,\Omega^2_{\tX})/ H^0(\tX,\Omega^2_{\tX})$
generated by differential forms $\omega\in H^0(\tX\setminus E,\Omega^2_{\tX})$,
 which have no poles along $E_I\setminus \cup_{v\not\in I}E_v$. 
\end{definition}
As in Theorem \ref{th:mult}{\it (c)}, let
$(\widetilde{X}/E_{\calv\setminus I},o_{\calv\setminus I})$ denote the multi--germ  obtained by
contracting the connected components of $E_{\calv\setminus I}$ in $\widetilde{X}$. Let $\tX(\calv\setminus I)$ be a small
neighbourhood of $E_{\calv\setminus I}$ in $\tX$, which is the inverse image by $\phi$ of a small Stein
neighbourhood of $(\widetilde{X}/E_{\calv\setminus I},o_{\calv\setminus I})$.

\begin{proposition}\label{prop:I} (a)
$\dim \Omega(I)=p_g(\widetilde{X}/E_{\calv\setminus I},o_{\calv\setminus I})$.

(b) Set $\overline{\Omega}(\emptyset):= H^0(\tX(\calv\setminus I)\setminus E_{\calv\setminus I}, \Omega^2_{\tX(\calv\setminus I)})/
H^0(\tX(\calv\setminus I), \Omega^2_{\tX(\calv\setminus I)})$. Then
linear map $\rho:\Omega(I)\to \overline{\Omega}(\emptyset)$, induced by restriction, is an isomorphism.

(c) Fix $I\subset \calv$ as above and set $J\subset \calv$ with $J\cap I=\emptyset$.
Let  $\overline{\Omega}(J)$ be the subspace of $\overline{\Omega}(\emptyset)$ generated by forms from
$H^0(\tX(\calv\setminus I)\setminus E_{\calv\setminus I},\Omega^2_{\tX(\calv\setminus I)})$ without pole along $E_J$.
Then the restriction of $\rho$ to $\Omega(J)\cap \Omega(I)$ induces an isomorphism
$\Omega(J)\cap \Omega(I)\to \overline{\Omega}(J)$.

In particular, for any $I$,
the subspace arrangement $\{\overline {\Omega}(J)\}_{J\cap I=\emptyset}$ of the
multigerm $(\widetilde{X}/E_{\calv\setminus I},o_{\calv\setminus I})$ and resolution $\tX(\calv\setminus I)$
can be recovered from the arrangement $\{\Omega (M)\}_M$ via $\{\Omega(I)\cap \Omega(J)\}_{J\cap I=\emptyset}$.
\end{proposition}
\begin{proof} {\it (a)}
Fix  $Z=\sum_{v\in \calv\setminus I} n_vE_v$  with all  $n_v\gg 0$.  By (\ref{eq:duality})
$\dim \Omega(I)=\dim\,H^0(\Omega_{\tX}^2(Z))/H^0(\Omega_{\tX}^2)=
h^1(\calO_Z)$,
which equals $p_g(\widetilde{X}/E_{\calv\setminus I},o_{\calv\setminus I})$
by formal function theorem.

{\it (b)} If $[\omega]\in \ker(\rho)$, then $\omega$ has no pole along $E_I$ (since $[\omega]\in\Omega(I)$), and has no pole
along $E_{\calv\setminus I}$ either (since $\rho[\omega]=0$). Hence $[\omega]=0$, and $\rho$ is injective.
Since by {\it (a)} the dimension of the source and the target is the same,
$\rho$ is an isomorphism.

{\it (c)} By {\it (b)}, for any  $\bar{\omega}\in \overline{\Omega}(J)$ there exists $\omega\in \Omega(I)$ with
$\rho(\omega)=\bar{\omega}$. Note that $\omega$ is necessarily in $\Omega(I\cap J)$, hence
$\Omega(J)\cap \Omega(I)\to \overline{\Omega}(J)$ is onto.
\end{proof}

The next result shows that the linear subspace arrangement   $\{V_{\tX}(I)\}_I$ of
$H^1(\tX,\calO_{\tX})$ (cf. \ref{def:VI})
is dual to  the linear subspace arrangement $\{\Omega_{\tX}(I)\}_I$ of $\Omega_{\tX}(\emptyset)=
H^0(\tX\setminus E,\Omega^2_{\tX})/ H^0(\tX,\Omega^2_{\tX})$.
\begin{theorem}\label{th:DUALVO}
Via duality (\ref{eq:LD}) one has  $V_{\tX}(I)^*=\Omega_{\tX}(I)$.
\end{theorem}
\begin{proof} We fix a cycle $Z\gg 0$ for which $V_Z(I)=V_{\tX}(I)$.
Choose $l'=-\sum_{v\in I}a_vE^*_v$ such that each $a_v$ is so large that $\im ( c^{l'})$ is an affine space,
cf. Theorem \ref{th:mult}.
Then, any element $\calL$ of $V_Z(I)$ has the form $\calO_{Z}(D_1-D_2)$, with both $D_1, \, D_2\in
\eca^{l'}(Z)$. Lift $\{D_i\}_{i=1,2}$  to  effective divisors $\{D_1'\}_{i=1,2}$ in $\tX$. Since they do not intersect
$E_{\calv\setminus I}$, the class $[\lambda]$ of $\calO_{\tX}(D_1'-D_2')$ in $\pic^0(\tX)$ can be
represented by a ${\rm \check{C}ech}$  cocycles $\{\lambda_{ij}\}$, which in a neighbourhood of
$E_{\calv\setminus I}$ are all zero. Therefore, if $\omega $ is a form which has no pole along $E_I$,
 $\langle [\lambda], [\omega]\rangle=0$ by (\ref{eq:Stokes}). That is, $\langle V_{\tX}(I), \Omega(I)\rangle=0$,
or $V_{\tX}(I)\subset \Omega(I)^*$. Since by (\ref{eq:HD2}) and Proposition \ref{prop:I}{\it (a)} one has
$\dim V_{\tX}(I)=p_g-\dim \Omega(I)$, we get $V_{\tX}(I)=\Omega(I)^*$.
\end{proof}

\begin{example}\label{ex:CONTDIMzero} (Continuation of Example \ref{ex:dimimzero})
Fix $l'\in-\calS'$ with $E^*$--support $I\subset \calv$ as in \ref{ex:dimimzero}, and
choose $Z\gg0$. Then
$$ \mbox{$\im (c^{l'})$ is a point} \ \Leftrightarrow \  V_{\tX}(I)=0 \ \Leftrightarrow \
\Omega_{\tX}(I)=\Omega_{\tX}(\emptyset).$$

\end{example}
\subsection{Convexity property of $\Omega(\{v\})$'s}\label{ss:CONV} \
Clearly, the subspace arrangement has the properties $\Omega(\emptyset)\simeq \bC^{p_g}$, and
$\Omega(I\cup J)=\Omega(I)\cap \Omega(J)$.
In this subsection we establish an
interesting additional structure property of the arrangement.
It is the analytical analogue of topological convexity property
\cite[Prop. 4.4.1]{LNN}.

 For simplicity write $\Omega_v:=\Omega(\{v\})$ for $v\in\calv$, and define
$$\Pi(I):=\left\{\begin{array}{ll}\emptyset & \mbox{if $I=\emptyset$} \\
\sum_{v\in I}\Omega_v & \mbox{if $I\not=\emptyset$.}\end{array}\right.
$$

\begin{proposition}\label{th:STRUCT}
For any $I\subset \calv$ let $\Gamma_I$ be the smallest connected subtree of $\Gamma$ whose
set of vertices $\overline{I}$ contains $I$. Then $\Pi(J)=\Pi(I)$ for any $I\subset J\subset \overline{I}$.
\end{proposition}
\begin{proof}  By induction, it is enough to consider the case $J=I\cup \{u\}$, such that
$u$ is on the geodesic path connecting $v,w$ with $v,w\in I$. Moreover, it is enough to show that
$\Omega_u\subset \Omega_v+\Omega_w$. Write the connected components of $\Gamma\setminus u$ as $\cup_{k=0}^s
\Gamma_k$, and set $I_k:=\calv(\Gamma_k)$. Assume that $w\in I_0$.

Choose an arbitrary $\omega\in \Omega_u$ and consider its restriction $\omega|_{\tX(I_0)}$ in
$\overline{\Omega}(\emptyset):= H^0(\Omega^2(\tX(I_0)\setminus E_{I_0}))/ H^0(\Omega^2(\tX(I_0)))$.
By Proposition \ref{prop:I}(b) $\Omega(\calv\setminus I_0)\to \overline{\Omega}(\emptyset)$ is bijective, hence
there exists $\omega_v\in \Omega(\calv\setminus I_0)$ such that $\omega_v|_{\tX(I_0)}=\omega|_{\tX(I_0)}$.
But $\Omega_v\supset \Omega(\calv\setminus I_0)$, hence $\omega_v\in \Omega_v$. On the other hand,
$(\omega-\omega_v)|_{\tX(I_0)}=0$, hence $\omega_w:=\omega-\omega_v\in \Omega_w$. Thus $\omega=\omega_v+\omega_w\in
\Omega_v+\Omega_w$.
\end{proof}
\begin{example}\label{ex:445}
Consider the weighted homogeneous isolated hypersurface singularity $(X,o)=\{x^4+y^4+z^5=0\}\subset
(\C^3,0)$. One verifies that $p_g=4$ (use either \cite{pinkham}, or section \ref{s:WH} from here).
We consider the minimal good resolution, whose  graphs is

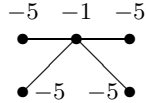
\begin{picture}(200,50)(-50,0)
\put(70,40){\makebox(0,0){\small{$-5$}}}
\put(90,40){\makebox(0,0){\small{$-1$}}}
\put(70,30){\circle*{4}}
\put(90,30){\circle*{4}}
\put(110,30){\circle*{4}}
\put(100,10){\makebox(0,0){\small{$-5$}}}
\put(110,40){\makebox(0,0){\small{$-5$}}}
\put(80,10){\makebox(0,0){\small{$-5$}}}
\put(90,30){\line(-1,-1){20}}
\put(110,10){\circle*{4}}
\put(70,10){\circle*{4}}
\put(70,30){\line(1,0){40}}\put(90,30){\line(1,-1){20}}
\end{picture}

If $\omega$ is the Gorenstein form, then $\omega$, $z\omega$, $x\omega$ and $y\omega$ generate
$H^0(\tX\setminus E,\Omega^2_{\tX})/H^0(\Omega^2_{\tX})$. The pole orders along the central
curve $E_0$ are 7, 3, 2, 2. Let $v_i$ ($1\leq i\leq 4$) be the end--vertices. Then
 for fixed $i$,
 $\calv\setminus \{v_i\}$ represents a minimally elliptic singularity. Hence
$\Omega_{v_i}\simeq\C$  by (\ref{eq:HD2}) and  Theorem \ref{th:DUALVO}.
If $\xi_i$ are the roots of $\xi^4+1=0$, then $(x+\xi_iy)\omega$ generates $\Omega_{v_i}$, hence
$\sum_{i=1}^4 \Omega_{v_i}\simeq\C^2=\langle x\omega,y\omega\rangle$.

In particular, the linear subspace arrangement $\{\Omega_{v}\}_v$ in $\C^{p_g}=\C^4$ is not generic at all. Furthermore, $\Omega_{v_0}=0$ hence \ref{th:STRUCT}
can also be exemplified on this concrete example.
\end{example}

\subsection{Reduction to an arbitrary $Z>0$.}\label{ss:red8}
The duality from Theorem \ref{th:DUALVO}, valid for $\tX$ (or, for any $Z\gg0$) can be generalized for any $Z\geq E$ as follows. For the definition of $V_Z(I)$ see Definitions \ref{def:VZ} and \ref{def:VI}.
In parallel, define $\Omega_Z(I)$ as the subspace $H^0(\Omega^2_{\tX}(Z|_{\calv\setminus I}))/
H^0(\Omega^2_{\tX})$ in $H^0(\Omega^2_{\tX}(Z))/
H^0(\Omega^2_{\tX})$. By (\ref{eq:duality})  $\dim H^0(\Omega^2_{\tX}(Z))/
H^0(\Omega^2_{\tX})=h^1(\calO_Z)$, while  $\dim \Omega_Z(I)=h^1(\calO_{Z|_{\calv\setminus I}})$.
But,  by pairing (similarly as in the proof of Theorem \ref{th:DUALVO}) $V_Z(I)\subset \Omega_Z(I)^*$.
Furthermore, by (\ref{eq:HD2}),  $\dim V_Z(I)=\dim \Omega_Z(I)^*$. Hence
\begin{equation}\label{eq:HDIZ}
V_Z(I)= \Omega_Z(I)^*.
\end{equation}

\section{The `stable' dimensions $\{\dim (V_{Z}(I))\}_{I}$ and natural line bundles}\label{s:Hilb}

\subsection{}\label{ss:position} Recall that the {\em saturation} in $\calS'$
of a submonoid $\calm\subset \calS'$ is the submonoid $\overline{\calm}:=\{l'\in\calS'\,:\, \exists \ n\geq 1 \ \mbox{with}
\ nl'\in\calm\}$.

Let us fix some cycle $Z\geq E$.
Regarding the mutual position of the natural line bundle $\calO_Z(l')$ with respect to the image of $c^{l'}:\eca^{l'}(Z)\to
\pic^{l'}(Z)$ we can consider
three cases.

(a) {\it  $\calO_Z(l')\in \im (c^{l'})$, or, equivalently, $0\in \im(\widetilde{c}^{l'})$.} The set of cycles $l'$ satisfying this property is denoted by $\calS'_{im}$. 
Clearly $0\in \calS'_{im}$ and
by the first paragraphs of \ref{ss:addDiv} it is a  sub-monoids of $\calS'$. (In the literature, this monoid ---
defined for bundles over $Z\gg 0$, or over $\tX$ ---, is called the
{\em analytic monoid } of $(X,o)$, in contrast with the
{\em topological monoid } $\calS'$, since it  indexes  the restrictions to $E$
of the divisors of different holomorhic sections of
the natural line bundles of $\widetilde{X}$, or
divisors of fuctions of the universal
abelian  covering of $(X,o)$, cf. \cite{Nfive}.)

(b)  $0 \in V_Z(l')$.
This, via Theorem \ref{th:mult},  reads as $\calO_Z(nl')\in \im (c^{nl'})$, or
 $0\in \im (\widetilde{c}^{nl'})$,  for $n\gg 1$.
 The cycles $l'$ satisfying this property are indexed by $\overline{\calS'_{im}}$.

(c)  $0 \not \in V_Z(l')$.
Such cycles $l'$ are indexed by $\calS'\setminus\overline{\calS'_{im}}$.

\begin{example}\label{ex:submonoids} In general, $\calS'_{im} \varsubsetneq \overline{\calS'_{im}}$. E.g. in Example
\ref{ex:notequidim}, $\calO_Z(-Z_{min})\not\in \im (c)$, however
 $\calO_Z(-2Z_{min})\in \im (c)$.
Furthermore, in general,
 $\overline{\calS'_{im}}\varsubsetneq \calS'$ either. Indeed, take e.g. a situation when $\im (c^{l'})$ is a point different than
$\calO_Z(l')$. Then  $\calO_Z(nl')\not \in \im (c^{nl'})$ for $n\geq 1$, hence $nl'\not\in \calS'_{im}$ for $n\geq 1$.
In such cases $\calS'\setminus \overline{\calS'_{im}}$ is even infinite.
For a concrete example see the last case of \ref{ex:dimim}.

\end{example}
\begin{lemma}\label{lem:finite}
Let $Z\geq E$ be an arbitrary cycle as above.

(a) Fix $l'\in -\calS'$ as above, and assume that $n\geq 1$ satisfies the next assumptions:

(i) $\im(\widetilde{c}^{nl'})=A(nl')$
(automatically satisfied if $n$ is sufficiently large, cf. Theorem \ref{th:mult}),

(ii) $0\in \im (\widetilde{c}^{nl'})$.

\noindent Then $0\in A(l')$ and
$\im (\widetilde{c}^{ml'})=A(l') $ for any $m\geq n$.

(b)  $\overline{\calS'_{im}}=\calS'$ if and only if  $\calS'\setminus  \calS'_{im}$ is finite.
\end{lemma}
\begin{proof}
(a) Since  $0\in A(nl')$, by Theorem
\ref{th:mult}{\it (a)} necessarily
$A(kl')=A(l')=V(l')$  for any $k\geq 1$.
Fix $\calL\in \im(\widetilde{c}^{kl'})$. Then, $\calL\in A(kl')$ and
by (\ref{eq:addclose}) and Lemma \ref{lem:propAZ},
$A(l')=A(l')+\calL\subset \im (\widetilde{c}^{nl'})
+\im (\widetilde{c}^{kl'})\subset
\im(\widetilde{c}^{(n+k)l'})\subset A((n+k)l')=A(l')$.
Part (b) follows from (a).
\end{proof}

\subsection{}
In the remaining part of this subsection we will work with line bundles defined over $Z\gg 0$.

\begin{definition}\label{def:ECC}
(a) Following Neumann and Wahl \cite{NWECTh},
we say that $(X,o)$ and its resolution $\phi$ satisfy the {\it End Curve Condition} (ECC)
if $E^*_v\in \calS'_{im}$ for any {\it end vertex} $v\in\calv $ (i.e. for $\delta_v=1$).

(b)  We say that $(X,o)$ and its resolution $\phi$ satisfies the {\it Weak End Curve Condition} (WECC)
if $E^*_v\in \overline{\calS'_{im}}$ for any end vertex $v\in\calv $.
\end{definition}

If  we restrict ourselves to singularities with rational homology sphere links,
by End Curve Theorem \cite{NWECTh} (see also \cite{OECTh}) singularities which satisfy
 ECC  are exactly the splice quotient singularities of Neumann and Wahl \cite{NWsq}.
 The WECC terminology is new in the literature, however
 its necessity and importance appeared in many private discussions
 of the second author with T. Okuma in the last decade.
 The main question regarding
  singularities satisfying WECC is how can one generalize the results valid for splice quotient singularities to this larger family. The present article shows that
  e.g.  the $p_g$--additivity formula of
Okuma extends. Indeed,  the general additivity formula (\ref{eq:HD2}) provides an additivity
with correction term $\dim V_{\tX}(I)$. Furthermore, as we will see
in the next discussions, the correction term $\dim V_{\tX}(I)$ has different
reinterpretations in terms of certain Hilbert polynomials or
Poincar\'e series (similarly as in the splice quotient case)
whenever WECC is satisfied.

\begin{proposition}\label{prop:ConvexTh1}
(a) ({\bf Convexity property of $\overline{\calS'_{im}}$}) \
Fix $u,v\in\calv$, $u\not=v$. If $E_u^*$, $E_v^*\in \overline{\calS'_{im}}$ then for any vertex $w$ on the geodesic path
in the graph connecting $u$ and $v$ one has $E^*_w\in \overline{\calS'_{im}}$ too.

(b) $(X,o)$ satisfies WECC if and only if  $\overline{\calS'_{im}}=\calS'$.
\end{proposition}
\begin{proof}
Fix integers $n_u, \, n_v, \, n_w$ sufficiently large such that (i) $n_uE_u^*$, $n_vE_v^*$, $n_wE_w^*$ belong to $L$,
(ii) the $E_w$--multiplicities of these three cycles are equal, and (iii) $n_uE_u^*$ and $n_vE_v^*$ belong to $\calS'_{im}$.
Set $l:=n_uE_u^*-n_wE_w^*$, and let the connected components of $\Gamma\setminus w$ be $\cup_i\Gamma_i$. We distinguish
$\Gamma_{i_0}$, which contains $u$. Then
$l$ is supported on $\cup_i\Gamma_i$. Since $(l,E_z)=0$ for any $z\in\calv(\cup_{i\not=i_0}\Gamma_i)$, $l|\Gamma_i=0$ for all
 $i\not=i_0$. Since $(l,E_z)\leq 0$ for any $z\in\calv(\Gamma_{i_0})$, and
$(l,E_u)<0$, all the entries of $l|\Gamma_{i_0}$ are strict positive. We have similar property for
$n_vE_v^*-n_wE_w^*$ too. Hence $\min\{n_uE_u^*,n_vE^*_v\}=n_wE^*_w$.
Since, by assumption there exist  functions
$f_u$ and $f_v$, which can be regarded as
sections of $\calO(-n_uE^*_u)$ and $\calO(-n_vE^*_v)$ without fixed components,
 the generic linear combination
$af_u+bf_v$ is a section of $\calO(-n_wE^*_w)$
without fixed components.
 For {\it (b)} use part {\it (a)} and the fact that $\Gamma$ is a tree.
\end{proof}

\subsection{Different reinterpretations of $\dim(V_{\tX}(l'))$ when  $l'\in \overline{\calS'_{im}}$.}\label{ss:APP1} \
In the sequel we apply the results of the previous section (e.g. Theorem \ref{th:mult})
for natural line bundles. This will also include the `classical'  cases $\calL=\calO_{\widetilde{X}}(-l)$,
where $l$ is an effective integral cycle. In order to do this we will need
additional assumptions of type $\calL\in \im (c^{nl'})$.

We fix the following setup.
We  consider line bundles over
$\widetilde{X}$, or over $Z\gg 0$.
We write $V_{\widetilde{X}}(l')$ for the stabilized $V_Z(l')$ with  $Z\gg 0$.
We fix  $l'\in\calS'$ from $ \overline{\calS'_{im}}$, this means that there exists
$n\gg 1$ such that $\calO(-nl')$ admits sections
 without fixed components. Let $o\in\Z_{>0}$ be the order of $[l']$ in $L'/L$.
We also write $ol'=l\in L$. Note that $V_{\widetilde{X}}(l')=V_{\widetilde{X}}(ol')$, cf. Lemma \ref{lem:propAZ}.

\subsection{\bf  $\dim(V_{\tX}(l'))$ as the last coefficient of a Hilbert polynomial.} \label{ss:HILB}
Consider the situation of subsection \ref{ss:APP1}.
For $n\gg 1$ from the exact sequence of sheaves $0\to \calO_{\widetilde{X}}(-nl)\to
\calO_{\widetilde{X}}\to \calO_{nl}\to 0$, we get
$$\dim H^0(\calO)/H^0(\calO(-nl))=\chi(nl)-h^1(\calO(-nl))+p_g(X,o),$$
which combined with Theorem \ref{th:mult} gives
\begin{equation}\label{eq:APP1}
\dim H^0(\calO)/H^0(\calO(-nl))=\chi(nl)+\dim\, V_{\widetilde{X}}(l).
\end{equation}
This already shows that  $V_{\widetilde{X}}(l)$ is the free term of
the Hilbert polynomial associated with
$n\mapsto \dim H^0(\calO)/H^0(\calO(-nl))$. This fact can be reorganized even more.
 Note that by Theorem \ref{th:mult}{\it (d)} $\calO(-nl)$ is generated by global sections for
 all $n\geq n_0$ for some $n_0$.
 Therefore, if we denote the ideal $H^0(\widetilde{X},\calO(-n_0l))\subset \calO_{X,o}$ by $\calj$, then
 the integral closure of its powers satisfy $\overline{\calj^m}=H^0(\widetilde{X},\calO(-mn_0l))$
 \cite[????]{Lipman}. In particular, $\dim(\calO_{X,o}/\overline {\calj^m})=\chi(mn_0l)+\dim\, V_{\widetilde{X}}(l)$.

 Recall that there exist integral coefficients $\overline {e}_i(\calj)$
 (where $i=1,2,3$) such that $\dim(\calO_{X,o}/\overline {\calj^m})=
 \overline{e}_0(\calj)\binom{m+1}{2}- \overline{e}_1(\calj)\binom{m}{1} +\overline{e}_2(\calj)$ for $m\gg 1$.
 Here, the polynomial from the right hand side is called the
 {\it normal Hilbert polynomial of $\calj$}.
 One  verifies that $\overline{e}_2(\calj)$ is independent of the choice of $n_0$.
 Then, the two identities combined provide
 $\dim\, V_{\widetilde{X}}(l)=\overline{e}_2(\calj)$.

If in our general identities from
Theorem \ref{th:mult} we insert $\overline{e}_2(\calj)$ for $\dim\, V_{\widetilde{X}}(l)$, then we  recover
e.g. the results from \cite[\S 3]{OWY15a}; or the additivity statement from \cite[Cor. 4.5]{Ok}.


\subsection{$\dim(V_{\tX}(l'))$  in terms of the multivariable series  $P_{h=0}(\bt)$.}
\label{ss:PERCONSTANT} \
Assume again  that $l'\in  \overline{\calS'_{im}}$, and let $I$ be the $E^*_v$--support of $l'$, that is,
$l'=\sum_{v\in I} a_vE^*_v$ with $a_v\in\bZ_{>0}$.
Then with the notations of \ref{ss:APP1}, for $n$ sufficiently large $\calO(-nol')$ has no fixed components and
$h^1(\tX,\calO(-nl))=p_g-\dim V_{\tX}(I)$. This combined with
(\ref{eq:HPol2}) gives that for cycles  of type $nl$ ($n\gg 1$)
\begin{equation}\label{eq:PSUM}\sum_{\tilde{l}\in L, \ \tilde{l}\not\geq nl} p_{\calO(-\tilde{l})}\,=\, \chi(nl)+\dim V_{\tX}(I);\end{equation}
that is, the counting function $nl\mapsto \sum_{\tilde{l}\in L, \, \tilde{l}\not\geq nl} p_{\calO(-\tilde{l})}$
of the coefficients of  $P_{h=0}(\bt)$ is (for $n\gg 1$)  the multivariable quadratic polynomial
$\chi(nl)+\dim V_{\tX}(I)$ in $nl$, whose free term is exactly $\dim V_{\tX}(I)$.

The above counting function can be simplified  even more: we will reduce the variables of $P_0$ to the variables
indexed by $I$. For this
 we define  the projection (along the $E$--coordinates)  $\pi_{\cali}:\mathbb{R}\langle E_v\rangle_{v\in\mathcal{V}}\to \mathbb{R}\langle E_v\rangle_{v\in\cali}$, denoted also as
 $x\mapsto x|_{\cali}$, by $\sum_{v\in \calv}l_vE_v\mapsto \sum_{v\in\cali}l_vE_v$.

For further motivations and topological analogues of the next statements
see also \cite{LNN} (where $Z(\bt)$ plays the role of $P(\bt)$).

\begin{lemma}\label{lem:REDUCTION}
Assume that $l'= \sum_{v\in I}a_vE^*_v$ with $a_v>0$, and $l''\in \calS'$ too. Then
$l''\geq l'$ if and only if $l''|_I\geq l'|_I$.
\end{lemma}
\begin{proof} We prove the $\Leftarrow$ part. Write $l''-l'$ as $x+y$, where $x$ (resp $y$)
is supported on $E_I$ (resp. on $E_{\calv\setminus I}$). By assumption, $x\geq 0$.
For any $u\in \calv\setminus I$ one has $0\geq (l'',E_u)=(l',E_u)+(x,E_u)+(y,E_u)$. But
$(l',E_u)=0$ and $(x,E_u)\geq 0$. Hence $(y,E_u)\leq 0$ for any $u $ in the support of $y$.
Since $(\,,\,)$ is negative definite, $y\geq 0$.
\end{proof}
According to the $\pi_I$ projection, we also
define the series $P_{I,h}(\bt_I)$ (for any $h\in H$), in variables $\{t_v\}_{v\in I}$ by
$P_{I,h}(\bt_I):= P_{h}(\bt)|_{t_v=1, v\not\in I}$.

Note that the series $P_{I,0}(\bt_I)$ has the form $\sum _{l_I\in \pi_I(\calS'\cap L)} p_I(l_I)\bt_I^{l_I}$.
By Lemma \ref{lem:REDUCTION} one has
$$\sum_{\tilde{l}\in L, \ \tilde{l}\not\geq nl} p_{\calO(-\tilde{l})}\,=\,
\sum_{l_I\in \pi_I(L), \ l_I\not\geq nl|_I} p_I(l_I).$$
Therefore,  for $n\gg 1$, one also has that the counting function of the coefficients of the reduced series
$P_{I,0}$ provides  the same expression
\begin{equation}\label{eq:PSUM2}
\sum_{l_I\in \pi_I(L), \ l_I\not\geq nl|_I} p_I(l_I)\,=\, \chi(nl)+\dim V_{\tX}(I).
\end{equation}
(Note that if the $E^*$--support of $nl$ is $I$, then $nl|_I$ determines uniquely $nl$.)

E.g., if $I=\{v\}$ (under the assumption $E^*_v\in  \overline{\calS'_{im}}$), $P_{I,0}=\sum_{m\geq 0}p_{v}(m)t_v^m$
has only one variable, and
$\sum_{m\geq nl|_v} p_v(m)\,=\, \chi(nl)+\dim V_{\tX}(I)$ for $n\gg 1$.

\begin{theorem} Assume that $(X,o)$ is a splice quotient singularity
associated with the graph $\Gamma$
(or, equivalently, $\phi:\tX\to X$ satisfies the ECC, cf. Definition \ref{def:ECC}). Then for any
$I$ the dimension $\dim V_{\tX}(I)$ is topological, computable from $\Gamma$.
\end{theorem}
\begin{proof}
For splice quotient singularities $P(\bt)$ equals the topological series $Z(\bt)$, cf. \cite{NCL}. Hence,
in (\ref{eq:PSUM}) the left hand side can be replaced by the corresponding
sum of the coefficients of $Z(\bt)$.
\end{proof}
\begin{remark}
Let us denote the Seiberg--Witten invariant of the link $M(\Gamma)$,
associate with the canonical $spin^c$--structure of
$M(\Gamma)$ with $\frsw_{can}(M(\Gamma))$, and the corresponding normalized
 Seiberg--Witten invariant
by $\overline{\frsw}_{can}(M(\Gamma)):=\frsw_{can}(M(\Gamma))+(Z_K^2+|\calv(\Gamma)|)/8$, see e.g. \cite{LNN}.
Recall also that in the splice quotient case $P(\bt)=Z(\bt)$
(cf. \cite{NCL}). Therefore, if we replace in (\ref{eq:PSUM2}) $P(\bt)$ by $Z(\bt)$,
in the terminology of \cite{LNN} (\ref{eq:PSUM2}) reads as follows: $\dim V_{\tX}(I)$ is the
periodic constant of the $I$--reduction $Z_{I,0}(\bt_I)$ of $Z_0(\bt)$, and
by Theorem 3.1.1 of \cite{LNN} it equals $ -\overline{\frsw}_{can}(M(\Gamma))+
\overline{\frsw}_{can}(M(\Gamma\setminus I))$.
\end{remark}
\subsection{The equivariant version of \ref{ss:PERCONSTANT}.}\label{ss:EQUIVAR} \
Note that the identity $(\dagger) \ h^1(\tX, \calO(-nl'))=p_g-\dim V_{\tX}(I)$ holds
uniformly for any $n\gg 1$, though $[nl']\in H$ might have different $H$--classes.
Such stability usually cannot be proved via cohomology exact sequence of type
$0\to \calL(-l)\to \calL\to \calL|_l\to 0$, $l\in L_{>0}$ (since in this situation
$c_1(\calL(-l))-c_1(\calL)\in L$), or by eigenspace decomposition of some sheaf
associated with the universal abelian cover $(X_{ab},o)$. Maybe one should emphasize that
in the  above identity $(\dagger)$ the contribution $p_g$ comes from the dimension of
$\pic^{l'}$, which is independent of the class $[l']\in H$, and not from the
$p_g(X_{ab},o)_h$ for $h=0$.

Now, if we apply (\ref{eq:HPol2}) for $(\dagger)$ for different
classes we obtain the following fact.
Let us fix, as above $l'\in  \overline{\calS'_{im}}$ with  $E^*$--support $I$,
and let us fix also some $k\in \bZ_{\geq 0}$,  $h:=[kl']\in H$,
and write $kl'=r_h+l_k$ for some $l_k\in L$.
Let $o$ be the order of $[l']$ in $H$ as above. Then from (\ref{eq:HPol2}) one has
$$h^1(\calO(-r_h-l_k-nol')=-
\sum_{a\in L, \, a\not\geq 0} p_{\calO(-r_h-l_k-nol')}
+p_g(X_{ab},o)_h+\chi(l_k+nol')-(l_k+nol', r_h).$$
or, for any $k$ and any  $n\gg 1$,
$$\sum_{a\in L, \, a\not\geq 0} p_{\calO(-r_h-l_k-nol')}
=\chi(l_k+nol')-(l_k+nol', r_h)+p_g(X_{ab},o)_h-p_g+ \dim V_{\tX}(I).$$
Hence $\dim V_{\tX}(I)$ connects  the asymptotic behaviour of {\it different} $h$--components of $P(\bt)$
of the form $h=[kl'], \ k\in\Z$.

\section{The `non--stable' $\dim \im (c^{l'})$ and differential forms.}\label{s:ALG}

\subsection{} The first  theorem of this section  is a generalization of that statement
of section \ref{s:DIFFFORMS}, which says that for $Z\geq E$ the dual of the
vector  subspace $V_Z(nl')\subset H^1(\calO_Z)$,  the `stable image affine subspace'
$\im (\widetilde{c}^{nl'})=A_Z(nl')$ ($n\gg 1$) shifted to the origin,  agrees
 with the subspace of forms
 $\Omega_Z(I)$, where $I$ is the $E^*$--support of $l'$ (see
 Theorem \ref{th:DUALVO} and subsection \ref{ss:red8}).
 $V_Z(nl')$ can also be interpreted (up to a shift) as the tangent space
 at any $\calL\in A_Z(nl')$ of $A_Z(nl')$.
Hence, $\calL+V_Z(nl')$ is the intersection of all the kernels of
linear maps $T_{\calL}\omega$, where $\omega \in \Omega_Z(I)$
(that is, for all $\omega$ without  pole along those $E_v$'s   which support the divisors from $\eca^{nl'}(Z)$). For the explicit description of the duality see
\ref{ss:LD}.

The new setup is the following.
 Consider a divisor
  $D \in \eca^{l'}(Z) $, which is a union of $(l', E)$ disjoint divisors  $\{D_i\}_i $,
 each of them $\cO_Z$--reduction of divisors $\{\widetilde{D}_i\}_i $ from  $\eca^{l'}(\tX)$
 intersecting  $E$  transversally.
Set  $\widetilde{D}=\cup_i\widetilde{D}_i $ and  $\calL:=\widetilde{c}^{l'}(D)\in H^1(Z,\calO_Z)$.
Set also $Z=\sum_vm_vE_v$.

We introduce a subsheaf $\Omega_{\tX}^2(Z)^{{\rm regRes}_{\widetilde{D}}}$
of  $\Omega_{\tX}^2(Z)$ consisting of those forms $\omega$
which have the property that the residue ${\rm Res}_{\widetilde{D}_i}(\omega)$
has no poles along $\widetilde{D}_i$ for all $i$. This means that the restrictions of  $\Omega_{\tX}^2(Z)^{{\rm regRes}_{\widetilde{D}}}$ and    $\Omega_{\tX}^2(Z)$
on the complement of the support of $\widetilde{D}$ coincide, however along $\widetilde{D}$
is satisfies the following requirement.
If $p=E\cap \widetilde{D}_i=E_{v_i}\cap \widetilde{D}_i$ has  local coordinates
$(u,v)$ with $\{u=0\}=E$ and  $\widetilde{D}_i$ with  local equation $v$, then a local section of $\Omega_{\tX}^2(Z)$      near $p$
has the form $\omega=\sum_{i\geq -m_{v_i}, j\geq 0} a_{i,j}u^i v^j du\wedge dv$. Then
 the residue ${\rm Res}_{\widetilde{D}_i}(\omega)$ is
$(\omega/dv)|_{v=0}=\sum_i a_{i,0}u^idu$, hence the pole--vanishing reads as $a_{i,0}=0$ for all $i<0$.
Note that $\Omega_{\tX}^2(Z-\widetilde{D}) $ and the sheaf of regular forms  $\Omega_{\tX}^2$ are
subsheaves of $\Omega_{\tX}^2(Z)^{{\rm regRes}_{\widetilde{D}}}$.

\begin{theorem}\label{th:Formsres} In the above situation one has the following facts.

(a) The sheaves $\Omega_{\tX}^2(Z)^{{\rm regRes}_{\widetilde{D}}}/\Omega_{\tX}^2$ and $\calO_{Z}(K_{\tX}+Z-D)$ are isomorphic.

(b) $H^0(\tX,\Omega_{\tX}^2(Z)^{{\rm regRes}_{\widetilde{D}}})/
 H^0(\tX,\Omega_{\tX}^2)\simeq H^1(Z,\calL)^*$.
(The left hand side can be regarded as a subspace of $H^0(\tX,\Omega^2_{\tX}(Z))/H^0(\tX,\Omega_{\tX}^2)\simeq H^1(Z,\calO_Z)^*$.)

(c) The image $T_D\widetilde{c}^{l'} (T_{D} \eca^{l'}(Z))$ of the tangent map $T_D\widetilde{c}^{l'}$
at $D$ of
$\widetilde{c}^{l'}:\eca^{l'}(Z)\to H^1(Z,\calO_Z)=H^1(\tX,\calO_{\tX})$
is the intersection of kernels of
linear maps $T_{\calL}\omega:T_{\calL}H^1(\tX,\calO_{\tX})\to\bC$, where $\omega\in
H^0(\tX,\Omega_{\tX}^2(Z)^{{\rm regRes}_{\widetilde{D}}})$.
\end{theorem}
\begin{proof} {\it (a)} Consider the following diagram:
$$\begin{array}{ccccccccc}
0 & \to & \Omega^2_{\tX}(-\widetilde{D}) & \longrightarrow &
 \Omega^2_{\tX}(Z-\widetilde{D}) & \longrightarrow
& \calO_Z(K_{\tX}+Z-D) & \to & 0 \\
 & & \Big\downarrow \vcenter{%
 \rlap{$\alpha$}} & &  \Big\downarrow \vcenter{%
 \rlap{$\beta$}} &  & \Big\downarrow \vcenter{%
 \rlap{$\gamma$}}& & \\
0 & \to & \Omega^2_{\tX} & \longrightarrow &
\Omega^2_{\tX}(Z)^{{\rm regRes}_{\widetilde{D}}} & \longrightarrow
& \Omega^2_{\tX}(Z)^{{\rm regRes}_{\widetilde{D}}}/\Omega^2_{\tX} &
\to & 0 \end{array} $$

\vspace{1mm}

Above $\alpha$ and $\beta$ are the natural inclusions. We claim that
their cokernels are isomorphic.
Indeed, with the notation $M_{i,j}=u^iv^jdu\wedge dv$ one has
${\rm coker}(\alpha)=\{\sum_{j\geq 0,i\geq 0}a_{i,j}M_{i,j}\}/
\{\sum_{j\geq 1,i\geq 0}a_{i,j}M_{i,j}\}$ and
${\rm coker}(\beta)=
\{\sum_{j\geq 0,i\geq -m_{v_i}}a_{i,j}M_{i,j}\,|\, a_{i<0,0}=0\}/\{\sum_{j\geq 1,i\in\Z}a_{i,j}M_{i,j}\}$.
Hence $\gamma$ is an isomorphism.

{\it (b)}  Since $H^1(\tX,\Omega^2_{\tX})=0$, by part {\it (a)} we have
$H^0(\tX,\Omega_{\tX}^2(Z)^{{\rm regRes}_{\widetilde{D}}})/ H^0(\tX,\Omega_{\tX}^2)=H^0(\calO_Z(K_{\tX}+Z-D))$.
But, this last one equals $H^1(Z,\calO_Z(D))^*$ by Serre duality.

{\it (c)}
We prove the statement in the case  $ (l', E) = 1$, the general case follows
similarly. Hence,
set $l' = -E_v^*$ for some vertex $v \in \calv$, that is, $\widetilde{D}$
 is a transversal cut at the point $p$ of the exceptional divisor $E_v$.
 Consider local coordinates $(u,v)$
 around  $p$ as above. Recall that the local equation of $D$ is $v$.
Let $\{\widetilde{D}_t\}_{t\in\bC, \, |t|\ll 1}$ be a path
 in $\eca^{l'}$ at $D$  whose local equation is $v+tu^{o-1}$ for some
$o\geq 1$.

Consider also an arbitrary form
$\omega\in H^0(\tX, \Omega^2_{\tX}(Z)) $ (with local equation as above).
Then (the class of) $\omega$ is in the dual
space of the image $T_D\widetilde{c}^{l'} (T_{D} \eca^{l'}(Z))$ if and only if
$(T_{\calL}\omega)(T_D\widetilde{c}^{l'}(\delta))=0$  for all tangent vectors
$\delta$, the tangent vectors of paths of type $D_t$ at $D$. But
$T_{\calL}\omega(T_D\widetilde{c}^{l'}(\delta))=\lambda\cdot a_{-o,0}$ ($\lambda\not=0$) by  \ref{eq:Tomega}.
Therefore, the dual space of forms  is exactly the class of forms from
 $ H^0(\tX, \Omega^2_{\tX}(Z)^{{\rm regRes}_{\widetilde{D}}}) $.

 In fact, one also sees that the dimensions of these two spaces
 $\im( T_D\widetilde{c})$ and $ \cap_{\omega}\, T_{\calL}\omega$ agree.
Indeed, $\dim \im (T_D\widetilde{c})=h^1(\calO_Z)-h^1(Z,\calL)$ by (\ref{eq:dimfiber}). But,
$\dim \cap_{\omega}\, T_{\calL}\omega$ is the same by {\it (b)}.
\end{proof}
\begin{corollary}\label{cor:Res2} Assume that $\{\omega_1,\ldots, \omega_{h}\}$ is a basis of $H^0(\tX,\Omega^2_{\tX}(Z))/H^0(\tX,\Omega^2 _{\tX})$. Then
\begin{equation*}\begin{split}
H^0(\tX,\Omega_{\tX}^2(Z)^{{\rm regRes}_{\widetilde{D}}})/ H^0(\tX,\Omega_{\tX}^2)=\hspace{7cm}\\
\{(a_1,\ldots,a_{h})\in \bC^{h}: {\rm Res}_{\widetilde{D}_i}(\textstyle{\sum}_\alpha a_{\alpha}\omega_{\alpha}) \ \mbox{has no pole along $\widetilde{D}_i$ for all  $i$}\}.
\end{split}\end{equation*}
Hence, by Theorem \ref{th:Formsres}, the dimension of the right hand side is
$h^1(Z,\calL)$, and the number of independent relations between $(a_1,\ldots, a_{h})$, $h^1(\calO_Z)-h^1(Z,\calL)$,
is the dimension of $\im T_Dc^{l'}(T_D\eca^{l'}(Z))$.

In particular, $\dim (\im (c^{l'}(Z)))$ is the number of independent
relations for $\{\widetilde{D}_i\}_i$ generic.
\end{corollary}
\subsection{} The above theorem can be applied rather directly in several situations,
when we can provide a bases for
$H^1(Z,\calO_Z)^*=H^0(\tX,\Omega^2_{\tX}(Z))/H^0(\tX,\Omega_{\tX}^2)$,
and verify directly for certain (or for all) divisors $D$
the above pole--vanishing property.
In the next subsections we provide such applications.

\subsection{The Gorenstein case.}\label{ss:Gor}
Assume that $(X,o)$ is Gorenstein, fix a resolution $\tX\to X$ as above, and let
$\omega_0\in H^0(\tX,\Omega^2_{\tX}(Z_K))$ be the pullback of the Gorenstein form,
well defined up to a non--zero constant.
Its pole is $Z_K$, the (anti)canonical cycle.
Since $\Omega^2_{\tX}=\calO_{\tX}(-Z_K)$,
$H^0(\tX,\Omega^2_{\tX}(Z_K))/H^0(\tX,\Omega_{\tX}^2)$ is isomorphic
 with $H^0(\tX,\calO_{\tX})/H^0(\tX,\calO_{\tX}(-Z_K))$,
 hence if we fix a basis of
 $H^0(\tX,\calO_{\tX})/H^0(\tX,\calO_{\tX}(-Z_K))$
consisting of classes of functions $\{f_1,\ldots, f_{p_g}\}\subset H^0(\tX,\calO_{\tX})$
with divisors ${\rm div}_E{f_\alpha}\not\geq Z_K$
 then in $H^0(\tX,\Omega^2_{\tX}(Z))/H^0(\tX,\Omega_{\tX}^2)$ the classes of forms
 $\{f_1 \omega_0, \ldots, f_{p_g}\omega_0\}$ form a basis.

 Therefore, for any fixed $I\subset \calv$,
 \begin{equation}\label{eq:OmegaI}
 \Omega(I)=\{(a_1,\ldots,a_{p_g})\in \bC^{p_g}: m_{E_v}(\textstyle{\sum}_\alpha a_{\alpha}
 f_\alpha)\geq m_{E_v}(Z_K) \ \, \mbox{for any $v\in I$},
 \end{equation}
 where  $m_{E_v}(\cdot)$
denotes the coefficient of a cycle along $E_v$.

By Theorem \ref{th:mult}  $\dim \Omega(I)=h^1(\tX, \calL)$ for any $\calL$
 with $c_1(\calL)=nl'$ with $n\gg 1$ and where $I:=\{\mbox{$E^*$--support of $l'$}\}$.
Furthermore, the number of independent relations between $(a_1,\ldots, a_{p_g})$, $p_g-\dim\Omega(I)$, is
the dimension of the {\em stable} $\im (c^{nl'})$  ($n\gg 1$).

According to Theorem \ref{th:Formsres}, these facts have the
following generalizations.
Set $\widetilde{D}=\cup_i\widetilde{D}_i$ be a divisor as in \ref{s:ALG}:
 each $\widetilde{D}_i$ is a transversal cut intersecting $E_{v(i)}$. Let $\gamma_i:(\bC,0)\to
  (\widetilde{D}_i,\widetilde{D}_i\cap E_{v(i)})$, $t\mapsto \gamma_i(t)$,
 be a parametrization (local diffeomorphism). Set $\calL=\calO_{\tX}(D)$ and $c_1(\calL)=l'$.

\begin{theorem}\label{th:Res2} With the above notations one has
$$H^0(\tX,\Omega_{\tX}^2(Z)^{{\rm regRes}_{\widetilde{D}}})/ H^0(\tX,\Omega_{\tX}^2)=
\{(a_1,\ldots,a_{p_g})\in \bC^{p_g}: {\rm ord}_t(\textstyle{\sum}_\alpha a_{\alpha}
f_\alpha\circ \gamma_i)\geq m_{E_{v(i)}}(Z_K) \ \, \mbox{for all  $i$}\}.$$
Similarly as in Corollary \ref{cor:Res2},
 the dimension of the right hand side is
$h^1(\tX,\calL)$, and the number of independent relations between $(a_1,\ldots, a_{p_g})$, $p_g-h^1(\tX,\calL)$,
is the dimension of $\im T_Dc^{l'}(T_D\eca^{l'}(Z))$ ($Z\gg 0$), and
$\dim (\im (c^{l'}))$ is the number of independent relations for
$\{\widetilde{D}_i\}_i$ generic.
\end{theorem}

We will apply this theorem
in section \ref{s:SI} for superisolated (hypersurface, hence Gorenstein)
germs. 

The general non--Gorenstein case (that is, Corollary \ref{cor:Res2})
will be exemplified  in section
\ref{s:WH} on the case of  weighted homogeneous germs, in which case we construct
a concrete basis
$\{\omega_1,\ldots, \omega_{p_g}\}$.


\section{Superisolated singularities}\label{s:SI}

\subsection{The setup.}
We  will exemplify the Gorenstein case on a special family of isolated
hypersurface singularities.
The family of superisolated singularities creates a bridge between the
theory of projective plane curves and the theory of surface singularities.
 This bridge will be present  in the next discussions as well.
For details and results regarding such  germs see e.g. \cite{Ignacio,LMNsi}.

Assume that $(X,o)$ is a hypersurface superisolated singularity.
This means that $(X,o)$
is a hypersurface singularity $\{F(x_1,x_2,x_3)=0\}$, where the
homogeneous terms $F_d+F_{d+1}+\cdots$ of $F$ satisfy the following properties:
$\{F_d=0\}$ is reduced and it defines in $\bC\bP^2$ an irreducible
rational cuspidal curve $C$; furthermore, the intersection
$\{F_{d+1}=0\}\cap {\rm Sing}\{F_d=0\}$ in $\bC\bP^2$ is empty.
The restrictions regarding $F_d$ implies that the link of $(X,o)$
is a rational homology sphere
(this fact motivates partly the presence of these restrictions).
With $F_d$ fixed,  all the possible choices for $\{F_i\}_{i>d}$ define an
equisingular family of singularities with fixed topology and fixed
$p_g=d(d-1)(d-2)/6$.
For simplicity, here we will take for $F_{d+1}$ the $(d+1)^{th}$--power
of some linear function and $F_{i}=0$ for $i>d+1$.
Moreover, by linear change of variables, we can assume $F_{d+1}=-x_3^{d+1}$.
(Note that in our treatment the analytic type of the singularity plays a crucial role,
hence, by the choice $F_{d+1}=-x_3^{d+1}$ we restrict ourselves to
a special analytic family. We do this since in this case  the presentation
of the next subsections are
more transparent.
However, it would be  interesting to analyse the
stability/non-stability of the  Abel map
in the whole equisingular family  when we vary $F_{i}, \
i\geq  d+1$.)

If we blow up the origin of $\bC^3$ then the strict transform
$X'$ of $X$ is already smooth (this property is responsible for the name `superisolated'),
the exceptional curve
$C'\subset X'$ is irreducible and  it can be identified with $C$ \cite{Ignacio}.
Hence, resolving the plane curve singularities of $C'$ we get a
minimal resolution of $X$;  for the precise resolution graph see
e.g.  \cite{Ignacio,LMNsi}. In the minimal (or, in the partial)
resolution the exceptional curve corresponding to $C'$ will be denoted by $E_0$.
 In the chart $x_1=uw, \ x_2=vw, \ x_3=w$ the total transform
 has equation $ w^d(w-F_d(u,v,1))=0$, $X'=\{w=F_d(u,v,1)\}$,
$C'=\{w=F_d(u,v,1)=0\}$.

We wish to discuss the Abel map associated with several choices of $l'$ and $Z$.

\subsection{The case $l'=-kE_0^*$ $(k\geq 1)$, $Z=Z_K$ (and generic divisor on $\eca^{l'}(Z)$).} \label{ss:firstchoice} \

In this case a generic point $D$ of $\eca^{l'}(Z)$ consists of $k$ transversal cuts
of $E_0$ at generic points. In order to determine $\dim \im (c^{l'})$, which equals
$\dim \im T_D\widetilde{c}^{l'}(T_D\eca^{l'}(Z))$, we will apply Theorem \ref{th:Res2}.
Hence, we need to
analyse the restriction of forms on the components of the divisor $D$.
Note that Theorem \ref{th:Res2} automatically provides
$h^1(Z_K,\calO(D))$ too.
Furthermore, by Grauert--Riemenschneider vanishing
$h^1(\tX,\calO(\widetilde{D}-Z_K))=0$, one also has $h^1(Z_K,\calO(D))=
h^1(\tX,\calO(\widetilde{D}))$.

Since the first blow up already creates the exceptional divisor
$C'=E_0$, all the computation can be done in this partial resolution
$\phi:X'\to X$, and we can even
assume that $D$ is in the chart considered above.
First, we find  $\{f_\alpha\}_{\alpha=1}^{p_g}$ such that
 $\{f_1\omega_0,\cdots,f_{p_g}\omega_0\}$ induces
 a basis in $H^0(\tX,\Omega^2_{\tX}(Z))/H^0(\tX,\Omega_{\tX}^2)$.
 Notice that the pullback of any monomial
 ${\bf x}^{\bf m}=x_1^{m_1}x_2^{m_2}x_3^{m_3}$ has vanishing order
 ${\rm deg}({\bf x}^{\bf m})=\sum_im_i=|{\bf m}|$
along $E_0$. Moreover, the multiplicity of $Z_K$ along $C'$ is $d-2$.
Since the number of monomials of degree strict less than $d-2$
is $p_g=d(d-1)d-2)/6$, the set $\{{\bf x}^{\bf m}\,:\, {\rm deg}({\bf x}^{\bf m})
\leq  d-3\}$ serve as a bases for
$H^0(\tX,\calO_{\tX})/H^0(\tX,\calO_{\tX}(-Z_K))$.

Next, we consider parametrizations of each component $\{\widetilde{D}_i\}_{i=1}^k$
(the liftings of the divisors $\{D_i\}_i$),
$t\mapsto \gamma_i(t)=(u_i(t),v_i(t),w_i(t))\subset X'$.
In fact, we can start with a parametrization
$t\mapsto (u_i(t),v_i(t))$  of a transversal cut of $\{F_d(u,v,1)=0\}\subset \bC^2$
at some smooth point. Then we lift it to $X'$ by  setting $w_i(t):= f(u_i(t),v_i(t),1)$.
The tranversality implies that $w_i(t)$ has the form
$c_1t+c_2t^2+\cdots$ with $c_1\not=0$, hence after a
reparametrization with $t':=w_i(t)$, we can assume that $w_i(t)=t$.

We denote the point $(u_i(0),v_i(0))\in \{F_d(u,v,1)=0\}\subset \bC^2$ by $p_i$.
We abridge $(u,v)^{\bf m}(p_i):=u_i(0)^{m_1}v_i(0)^{m_2}$.
 Then, the restriction of a monomial ${\bf x}^{\bf m}$ to $\widetilde{D}_i$ is
$$u(t)^{m_1}v(t)^{m_2}t^{|{\bf m}|}=t^{|{\bf m}|}\big( (u,v)^{\bf m}(p_i)+ H_{{\bf m}}(t)\big), $$
where $H_{{\bf m}}(t)$ denotes the `higher order terms' with $H_{{\bf m}}(0)=0$.
Hence, by  Theorem \ref{th:Res2},
\begin{equation*}
h^1(Z_K,\calO(D))=\dim \,
\Big\{(a_{\bf m})_{{\bf m}}  \in \bC^{p_g}:
\sum_{{\bf m}} a_{\bf m}\cdot
\frac{(u,v)^{\bf m}(p_i)+H_{{\bf m}}}{t^{d-2-|{\bf m}|}}
\ \mbox{has no pole for all $i$}\Big\}.
\end{equation*}
Expanding the sum into its Laurent series in $t$, and separating
the coefficients of $\{t^{-d+2+j}\}_{0\leq j\leq d-3}$, we get for each $D_i$
a linear system with $d-2$ equations for the variable
$(a_{\bf m})_{{\bf m}}$. We need to determine the rank of the corresponding matrix.
This matrix has a natural block decomposition, a block is indexed by
$j$ and the set ${\bf m}$ with fixed $|{\bf m}|$.
We prefer to order the rows by $t^{-d+2},\ t^{-d+3},\ldots, t^{-1}$.

E.g., for fixed $D_i$,
the first row has its first entry 1 (corresponding to the block $t^{-d+2}$ and
$|{\bf m}|=0$) and all other entries zero.
The second raw has some entry in the first place, the second block
corresponding to  $t^{-d+3}$ and
$|{\bf m}|=1$) has three entries, namely $u(p_i), v(p_i), 1$ (which are the
evaluations of the degree $\leq 1$ $(u,v)$--monomials at $p_i$), and the blocks
corresponding to $|{\bf m}|>1$ are zero.
More generally,
above the diagonal all the  blocks are zero, the
diagonal block indexed by  $t^{-d+2+j}$ and
$|{\bf m}|=j$ contains the evaluation of the $(u,v)$--monomials of degree
$\leq j$ at $p_i$.

E.g., if $k=1$, then the matrix has $d-2$ rows and $p_g$ columns, and each
diagonal block contains one entry 1, hence its rank of the linear
system  is $d-2$. In particular,
 $\dim\im  (c^{-E_0^*})=d-2$.

For $k\geq 2$, we have to put together all the linear equation corresponding to
all $D_i$. A block indexed by
$t^{-d+2+j}$ and  $|{\bf m}|=j'$ will have $k$ rows. Again, all
the blocks above the diagonal are zero. On the other hand, the rank of the diagonal
block indexed by  $t^{-d+2+j}$ and
$|{\bf m}|=j$ is as large as possible, it is $\min\{k,\binom{j+2}{2}\}$.
Indeed, its rows consists of the evaluation of $(u,v)$--monomials of
degree $\leq j$ at points $p_i$: since the points $p_i$ are generic they impose
independent conditions on the corresponding (homogeneous) linear system
(in variable $(x_1,x_2,x_3)$) of degree $j$. Hence, the rank of the matrix is
$\sum_{j=0}^{d-3}\min\{k,\binom{j+2}{2}\}$.
\begin{theorem}\label{th:sumSI}
For any  $k\geq 1$ the dimension of $\im (c^{-kE^*_0})$ is
$\sum_{j=0}^{d-3}\min\{k,\binom{j+2}{2}\}$.
The first $k$ when $c^{-kE^*_0}$ is dominant is $k=\binom{d-1}{2}$.
$\im (c^{-kE^*_0})$ has codimension 1 for $k=\binom{d-1}{2}-1$.

Accordingly, for a generic $\calL\in \im(c^{-kE^*_0})$, $h^1(Z_K, \calL)=p_g-
\dim(\im (c^{-kE^*_0}))$.
\end{theorem}

\subsection{The case $l'=-kE_0^*$ $(k\geq 1)$, $Z=Z_K$ (and special  divisor on $\eca^{l'}(Z)$).} \label{ss:secondchoice} \

In the previous subsection we considered {\it generic}
points ${\mathcal P}:=\{p_1,\ldots, p_k\}$ on $C$, in particular, for all
$j$ ($0\leq j\leq d-3$) they imposed independent conditions on the linear system
$\calO_{\bP^2}(j)$ (or, on the $(u,v)$--monomials of degree $\leq j$).
However, taking special points they might
 fail to impose independent conditions on some $\calO_{\bP^2}(j)$.
The discussion will show that $\im (c^{l'})$ has several (rather complicated)
$h^1$--stratification, (some of them) imposed by special divisors.

Here we will indicate such possibilities; nevertheless, for simplicity we will
restrict ourselves only to certain cases when only one block
degenerates and the rang of the total linear system is determined again
by the diagonal blocks.
Even under this restriction we find the situation extremely rich,
since it accumulates the classical plane curve geometry.
However, the reader is invited to work out  cases when
the global rank depends
on certain entries from the sub--diagonal blocks as well, covering even more sophisticated
$h^1$--strata.

Recall that in the diagonal block of $(t^{-d+2+j}, |{\bf m}|=j)$ we test if
$\calP$ impose independent conditions on $\calO_{\bP^2}(j)$ or not.
In the sequel we will assume that there exits exactly on $j$, say $j_0$,
when $\calP$ fails to impose
independent conditions. Clearly $j_0>0$. Furthermore, we will also assume that
$\binom{j_0+1}{2}\leq k\leq \binom{j_0+3}{2} $.
This means that in all the diagonal blocks with $j<j_0$ the number
$k$ of rows is $\geq $ than the number $\binom{j+2}{2}$ of columns,
hence the $j$--blocks has rank
$\binom{j+2}{2}$. Symmetrically, in all the $j$--diagonal blocks
with $j>j_0$ the number $k$ or rows is $\leq $ than the number
$\binom{j+2}{2}$ of columns, hence the
rank is $k$. Therefore, if the $j_0$--block is degenerated with rank
$\min\{k, \binom{j_0+2}{2}\}-\Delta $ for some $\Delta >0$,
then independently of the sub--diagonal entries, the rank of the matrix of the system is
$\sum_{j=0}^{d-3}\min\{k,\binom{j+2}{2}\}-\Delta$. In particular,
$h^1(Z_K,\calO(D))$ increases by $\Delta$ compared with the generic situation
of \ref{ss:firstchoice}.

Let us list some cases when such a degeneration can occur.
Take e.g. $j_0=1$ and $k=3$ and $\{p_1,p_2,p_3\}$ are collinear.
For $j_0=2$ we give two possibilities: either $k=4$ and the
four points are collinear, or $k=6$ and the six points are contained in a conic.

We recall here two classical theorems of plane curve geometry, which can be used to
produce similar examples; for more see the article \cite{EH} and the citations therein.

(a) \cite[Prop. 1]{EH} For $j_0\geq 1$ and $k\leq 2j_0+2$ the points $\calP$ fail to
impose independent conditions on $\calO_{\bP^2}(j_0)$ if and only if either
$j_0+2$ points of $\calP$ are collinear or $k=2j_0+2$ and $\calP$ is contained in a conic.

(b) \cite[Th. Cayley-Bacharach4]{EH}
Assume that $\calP$ consists of $k=e\cdot f$ poinst which are the
intersection points of two curves of degree $e$ and $f$.
Then if a plane curve of degree $j_0=e+f-3$
contains all but one point of $\calP$ then it contains all of $\calP$.


\section{Weighted homogeneous singularities.} \label{s:WH}

\subsection{Preliminaries}\label{ss:WHprel} Assume that $(X,o)$
is a weighted homogeneous normal surface singularity, that is, there exists a
a normal affine variety $X^{a}$, which admits  a good $\C^*$--action
and singular point $o\in X^{a}$ such that
$(X,o)$ is analytically isomorphic with $(X^{a},o)$.
This implies that the minimal good resolution graph $\Gamma$
is star shaped. As above, we assume that the link is  a rational
homology sphere, hence all  the vertex--genera  are zero.
We write $v_0$ for the central vertex, hence
 $\Gamma\setminus v_0$ consists of $\nu$ strings.
We assume that $\nu\geq 3$ (otherwise $p_g=0$, an uninteresting situation for
the Abel map).
  Let $-b_0$ be the  Euler number of $v_0$.
The Euler numbers of the vertices $v_{ji}$ of the $j^{th}$  string $(1\leq j\leq \nu)$
are $-b_{j1},\ldots, -b_{js_j}$, with $b_{ji}\geq 2$,
determined by the continued fraction
$\alpha_j/\omega_j=[b_{j1}, \dots, b_{js_j}]$,  where
$\gcd(\alpha_j,\omega_j)=1, \ 0<\omega_j<\alpha_j$. For each $j$,
$v_0$  is connected  with  $v_{j1}$ by one edge.
The link is a Seifert fibered 3--manifold with Seifert invariants
 $(b_0,g=0;\{(\alpha_j,\omega_j)\}_j)$. In particular, the Seifert invariants
 characterize the topological type uniquely, see e.g. \cite{neumann}.

We denote by $E_{ji}$ the irreducible exceptional curves indexed by vertices
$v_{ji}$. Let $P_j$ ($1\leq j\leq \nu$) be
$E_{v_0}\cap  E_{j1}$. One has the following result:
\begin{theorem}\label{th:ACTh} {\bf  (Analytic Classification
Theorem)} \cite{CR,Dol2,Dol3,OW3,pinkham,NeuG} \
The analytic isomorphism type of a normal surface weighted
homogeneous singularity (with rational homology sphere link)
with fixed Seifert invariants
is determined by the analytic type of
$(E_{v_0},\{P_j\}_j)$
 modulo an action of ${\rm Aut}(E_{v_0},\{P_j\}_j)$.
 (This is the same as the analytic classification of  Seifert
 line bundles over the projective line.)
 \end{theorem}

Next we show that the minimal resolution of any weighted homogeneous singularity can be
constructed by a special   `analytic plumbing'.

First we construct an analytic space $\widetilde{X^a}$ (the candidate for the resolution of $X^a$).
 Basically we mimic
the analytic plumbing construction of the resolution of
cyclic quotient singularities from e.g. \cite{BPV,Lauferbook}.
 Corresponding to the legs
we fix distinct complex numbers $p_j\in \C$;  the
 affine coordinates of the
points $P_j$. Each leg,  with divisors $\{E_{ji}\}_{i=1}^{s_j}$,
$1\leq j\leq \nu$,
will be covered by open sets $\{U_{j,i}\}_{i=0}^{s_j}$, copies of $\C^2$
with coordinates $(u_{j,i},v_{j,i})$.
For each $1\leq i\leq s_j$ we glue $U_{j,i-1}\setminus \{u_{j,i-1}=0\}$ with
$U_{j,i}\setminus \{v_{j,i}=0\}$.
The gluing maps are
$v_{j,i}=u^{-1}_{j,i-1}$ $(1\leq i\leq s_j)$ and
$u_{j,i}$ equals $u^{b_{ji}}_{j,i-1}v_{j,i-1}$ for $2\leq i\leq s_j$ and
$u^{b_{j1}}_{j,0}(v_{j,0}-p_j)$ for $i=1$.

Furthermore, all $U_{j,0}$ charts will be identified to each other:
$u_{j,0}=u_{k,0},\ v_{j,0}=v_{k,0}$; denoted simply by $U_0$, with coordinates
$(u_0,v_0)$. Till now, the curve $E_{v_0}$ appears only in $U_0$, it has equation $u_0=0$.
To cover $E_{v_0}$ completely we need another copy $U_{-1}$ of $\C^2$ with coordinates
$(u_{-1},v_{-1})$ as well; the gluing of $U_0\setminus \{v_0=0\}$ with $U_{-1}
\setminus \{u_{-1}=0\}$ is  $v_0=u_{-1}^{-1}$, $u_0=u_{-1}^{b_0}v_{-1}$.

\begin{picture}(300,135)(-70,-65)
\put(144,35){\makebox(0,0)[r]{$E_{j,3}$}}\put(260,35){\makebox(0,0)[r]{$E_{j,s_j}$}}
\put(263,45){\makebox(0,0){$v_{j,s_j}$}}\put(292,39){\makebox(0,0){$u_{j,s_j}$}}
\put(280,50){\vector(-1,-1){10}}\put(270,50){\vector(1,-1){10}}
\put(50,50){\line(1,-1){40}}
\put(80,10){\line(1,1){40}}\put(110,50){\line(1,-1){40}}
\put(140,10){\line(1,1){20}}\put(200,30){\line(1,1){20}}
\put(180,30){\makebox(0,0){$\ldots$}}\put(210,50){\line(1,-1){40}} \put(240,10){\line(1,1){40}}
\put(60,50){\vector(-1,-1){10}}
\put(60,50){\line(-1,-1){100}}
\put(50,50){\vector(1,-1){10}}
\put(29,41){\makebox(0,0){$v_{j,0}-p_j$}}\put(69,43){\makebox(0,0){$u_{j,0}$}}
\put(90,10){\vector(-1,1){10}}\put(80,10){\vector(1,1){10}}
\put(73,17){\makebox(0,0){$v_{j,1}$}}\put(99,15){\makebox(0,0){$u_{j,1}$}}
\put(120,50){\vector(-1,-1){10}}
\put(102,43){\makebox(0,0){$v_{2,j}$}}

\put(20,20){\makebox(0,0){$\ldots$}}

\put(95,-15){\makebox(0,0)[r]{$E_{k,3}$}}\put(210,-15){\makebox(0,0)[r]{$E_{k,s_k}$}}
\put(212,-5){\makebox(0,0){$v_{k,s_k}$}}\put(242,-11){\makebox(0,0){$u_{k,s_k}$}}
\put(230,0){\vector(-1,-1){10}}\put(220,0){\vector(1,-1){10}}
\put(0,0){\line(1,-1){40}}
\put(30,-40){\line(1,1){40}}\put(60,0){\line(1,-1){40}}
\put(130,-20){\makebox(0,0){$\ldots$}}\put(160,0){\line(1,-1){40}} \put(190,-40){\line(1,1){40}}
\put(10,0){\vector(-1,-1){10}}
\put(90,-40){\line(1,1){20}}\put(150,-20){\line(1,1){20}}
\put(0,0){\vector(1,-1){10}}
\put(-21,-9){\makebox(0,0){$v_{k,0}-p_k$}}\put(19,-7){\makebox(0,0){$u_{k,0}$}}
\put(40,-40){\vector(-1,1){10}}\put(30,-40){\vector(1,1){10}}
\put(23,-33){\makebox(0,0){$v_{k,1}$}}\put(49,-35){\makebox(0,0){$u_{k,1}$}}
\put(70,0){\vector(-1,-1){10}}
\put(52,-7){\makebox(0,0){$v_{k,j}$}}

\put(-40,-50){\vector(1,1){10}}
\put(-30,-50){\vector(-1,1){10}}
\put(-50,-45){\makebox(0,0){$v_{-1}$}}\put(-20,-45){\makebox(0,0){$u_{-1}$}}
\end{picture}

We call the output space $\widetilde{X^a}$. If we contract (analytically)
$E=E_{v_0}\cup (\cup_{j,i}E_{ji})$ we get a space $X^a$ whose  germ at its singular point
is a  normal surface singularity $(X_{pl},o)$. In this context, a resolution $\widetilde{X_{pl}}$
of  $(X_{pl},o)$
(as a subset of $\widetilde{X^a}$) is   the pullback of a small Stein neighbourhood of $o$.
The following statement is proved in \cite{NBOOK}; basically it
follows from the Analytic Classification Theorem \ref{th:ACTh} and
from the fact that if we blow down the legs the obtained space carries naturally a Seifert
line bundle structure over the projective line.

\begin{proposition}\label{lem:SpPl}
 The analytic structure on $(X_{pl},o)$ carries  a weighted homogeneous
 structure.
 Moreover, the minimal good resolution of any
weighted homogeneous singularity with
Seifert invariants  $(b_0,g=0;\{(\alpha_j,\omega_j)\}_j)$
admits such an analytic plumbing representation for
certain constants $\{p_j\}_j$ (that is, it can be embedded in some
$\widetilde{X^a}$ constructed above via plumbing).
By Theorem \ref{th:ACTh} we can even assume that each $p_j$ is non--zero
(what we will assume below).
\end{proposition}

The $\C^*$ orbits lifted to $\widetilde{X^a}$ and closed are as follows:
the generic ones, which intersect $E_{v_0}$ sit in $U_0\cup U_{-1}$ and
are given by $\{v_0=c\}$,
$c\in(\C\setminus \{\cup_j\{p_j\}\})\cup \infty$. The special
Seifert orbit for each $j$ in $U_{j,s_j}$ is
given by $\{v_{j,s_j}=0\}$.

In the sequel we will identify our weighted homogeneosu germ $(X,o)$ with such $(X_{pl},0)$.

For each $j$ we also introduce $0<\omega_j'<\alpha_j$ such that
$\omega_j\omega_j'-1= \alpha_j\tau_j$ for some $\tau_j$.

\subsection{A basis for $H^0(\tX\setminus E,\Omega^2_{\tX})/
H^0(\tX,\Omega^2_{\tX})$} \label{ss:whbasis}
For $\ell,\, \{m_j\}_j\in \Z$, $n\in\Z_{\geq 0}$,
let  $\omega_{\ell,n}^{0}:= u_0^{-\ell-1}\prod_j
(v_0-p_j)^{-m_j}v_0^ndv_0\wedge du_0$ be a section of
$\Omega^2_{\tX}$ over $U_0$, with possible poles over $E\cap U_0$.
This under the
transformation $v_0=u_{-1}^{-1}$, $u_0=u_{-1}^{b_0}v_{-1}$ transforms
into the following form on $U_{-1}$:
$$\pm\,u_{-1}^{-b_0\ell+\sum m_j-n-2}
v_{-1}^{-\ell-1}\textstyle{\prod_j} (1-u_{-1}p_j)^{-m_j}du_{-1}
\wedge dv_{-1}.$$
The regularity over $\tX\setminus E$
requires that the exponent of $u_{-1}$ should be non-negative:
\begin{equation}\label{eq:whforms1}
n\leq - b_0\ell-2 +\textstyle{\sum_j}m_j.
\end{equation}
Let is fix one of the legs, say $j$.
By induction using substrings of the legs and the corresponding continued
fraction identities (facts used intensively in cyclic quotient
invariants computations) one gets that the transformation between chart $U_0$ and
$U_{j,s_j}$ is $u_0=u_{j,s_j}^{-\tau_j} v_{j,s_j}^{-\omega_j}$,
\marginpar{{\bf CITE}}
$v_0=u_{j,s_j}^{\omega_j'}v_{j,s_j}^{\alpha_j}$.
Then, $\omega_{\ell,n}^{0}$ in the chart $U_{j,s_j}$ under this  transformation
 becomes
$$
u_{j,s_j}^{\tau_j\ell -\omega_j'm_j+\omega_j'-1}
v_{j,s_j}^{\omega_j\ell -\alpha_jm_j+\alpha_j-1}
(u_{j,s_j}^{\omega_j'} v_{j,s_j}^{\alpha_j}+p_j)^n \cdot
\textstyle{\prod _{j'\not=j}}
(u_{j,s_j}^{\omega_j'} v_{j,\s_j}^{\alpha_j}+p_{j'}-p_j)^{-m_j}\, dv_{j,s_j}
\wedge du_{j,s_j}.
$$
Again, by the regularity along $\tX\setminus E$, the exponent of
$v_{j,s_j}$ should be non-negative, hence
$\omega_j\ell -\alpha_jm_j+\alpha_j-1\geq 0$.
The largest solution for $m_j$ is
\begin{equation}\label{eq:whforms2}
m_j=\ce { \omega_j\ell/\alpha_j}.
\end{equation}
Hence, the form $\omega_{\ell,n}^{0}$ extends to a form $\omega_{\ell,n}$
on $\tX$,
regular on $\tX\setminus E$, if for $m_j:=\ce{ \omega_j\ell/\alpha_j}$
as in (\ref{eq:whforms2}) (for all $j$)
the inequality (\ref{eq:whforms1}) holds.
If $\ell< 0$ then $m_j=\ce{ \omega_j\ell/\alpha_j}\leq 0$, hence
the form $\omega_{\ell,n}$ is regular on  $\tX$, and in
$H^0(\tX\setminus E,\Omega^2_{\tX})/
H^0(\tX,\Omega^2_{\tX})$ it is zero. Hence, we can consider only the values
$\ell \geq 0$. For them  we set  as a combination of the right hand
side of (\ref{eq:whforms1}) and (\ref{eq:whforms2})
$$n_{\ell}:= - b_0\ell-2 +\textstyle{\sum_j}\ce{ \omega_j\ell/\alpha_j}.$$
If $n_{\ell}<0$ then there is no such form with pole $\ell+1$
along $E_{v_0}$, cf. (\ref{eq:whforms1}).
Set $\calw:=\{\ell\geq 0\,:\, n_{\ell}\geq 0\}$.

\begin{lemma}\label{lem:whforms} \ \cite{NBOOK}
The forms $\omega_{\ell,n}$ ($\ell\in\calw$, $0\leq n\leq n_{\ell}$)
form a basis of $H^0(\tX\setminus E,\Omega^2_{\tX})/
H^0(\tX,\Omega^2_{\tX})$.
\end{lemma}
\begin{proof}
By their construction, the forms generate $H^0(\tX\setminus E,\Omega^2_{\tX})/
H^0(\tX,\Omega^2_{\tX})$. But,
by the $p_g$--formula of Pinkham \cite{pinkham},
namely  $p_g=\sum_{\ell\in\calw} (n_{\ell}+1)$,
their number is exactly $p_g$, the dimension of this quotient space.
\end{proof}
\begin{remark}\label{rem:POLE}
$n_0=-2$, hence the $E_{v_0}$--pole of any $\omega_{\ell,n}$ is $\geq 2$.
\end{remark}

\subsection{Natural line bundles}\label{ss:whNLB}
Let $\widetilde{X^a}$ be  as above, let $O^a $ be the
closure of a  lifted $\C^*$ orbit into $\widetilde{X^a}$, and set
 $O:=O^a \cap \tX\subset \tX$.
\begin{theorem}\label{th:whNLB}
$\calO_{\tX}(O)$ is a natural line bundle in $\pic(\tX)$.
\end{theorem}
\begin{proof}
For each $O=O^a\cap \tX$ we find a local analytic function $f_O:(X,o)\to (\C,0)$
such that the divisor in $\tX$ of $f_O\circ \phi$ has the form
$n_OO+\sum_vn_vE_v$ for some $n_O,n_v\in \Z_{>0}$. This implies that
$\calO_{\tX}(n_OO)\simeq \calO_{\tX}(-\sum_vn_vE_v)$.
In order to find $f_O$  we use the fact that
a weighted homogeneous germ is splice quotient \cite{NWsqI,NWsq}.
In fact, by \cite{Neu}, the universal abelian cover (UAC) of $(X,o)$
is a Brieskorn complete intersection and
certain powers of the  coordinate functions of this complete intersection
are the end curve functions of $(X,o)$ which have  the wished
properties for the orbits supported by the
end vertices.

A similar argument clarifies the case of the other (generic) orbits as well.
 Let the UAC complete intersection equations be
 $\sum_{i,j}a_{i,j}z^{\alpha_j}=0$, $1\leq j\leq \nu$, and
$1\leq i\leq \nu-2$, and where $\{a_{i,j}\}_{i,j} $ has full rank,
cf. \cite{Neu}. Then we add one more equation
of type $\sum_ib_iz^{\alpha_i}+w=0$, such that the new larger matrix
has again full rank. The new system corresponds
to a splice quotient equations of the graph $\Gamma'$ obtained from
$\Gamma$ by blowing up the central vertex.
The point is that the resolution with dual graph $\Gamma'$ of this
splice quotient singularity associated with
$\Gamma'$ can be obtained from $\tX$ by blowing up a  certain point
$P\in E_{v_0}\setminus \cup_jP_j$.
$P$ is determined by the choice of the coefficients $\{b_i\}_i$, and
modification of the  $\{b_i\}_i$'s  provides different points $P$.
By the theory of splice quotient singularities, the end curve function
$w=0$ cuts out an end curve
in the UAC, which projected on $(X,o)$ is irreducible. Hence
$\sum_ib_iz^{\alpha_i}$ is a weighted homogeneous
function on the UAC of $(X,o)$, one of its powers is a homogeneous
function on $(X,o)$ whose reduced
zero set is irreducible. Its strict transform is some $O$,
where $O\cap E_{v_0}$ depends on the choice of
$\{b_i\}_i$. In particular all such orbits define the same line bundle,
the natural line bundle.
\end{proof}

\subsection{The Abel maps, $h^1(Z,\calL)$ and $\dim \im (c^{l'})$
for different line bundles and $l'$ }\label{ss:whAbel1} \

In the sequel we fix  a cycle $Z$: for simplicity we assume that
$Z\gg 0$, e.g. in the numerical Gorenstein case we can take $Z=Z_K$, or,
in general, $Z\in Z_K+\calS'$ (in which cases for any $\calL\in \pic(\tX)$ with
$c_1(\calL)\in -\calS'$ one has $h^1(\tX,\calL)=h^1(Z,\calL|_Z)$).
In this case we will use all the differential forms $\omega_{\ell,n}$
constructed above.
The interested reader might rewrite the statements  and proofs below
for smaller cycles (using forms a similar system of forms  with poles
$\leq Z$).

\subsection{The value $h^1(Z,\calO_Z(-kE^*_{v_0}))$, $k\geq 1$.}\label{ss:h1egy}\
Consider the natural line bundle $\calO_Z(-E^*_{v_0})$. If $O_q$ denotes the
intersection of the generic $\C^*$--orbit with $\tX$, $O_q\cap E_{v_0}=\{q\}$
(where the intersection point $q$ can be identified with the $v_0$--
affine coordinate in $U_0$), then by Theorem \ref{th:whNLB}
$\calO_Z(-E^*_{v_0})=\calO_Z(O_q)$ for any $q\in E_{v_0}\setminus \cup_j E_{j,1}$.
Recall that $O_q$ in
the chart $U_0$ is given by $\{v_0=q\}$.
For $k$ distinct orbits $O_{q_1},\ldots, O_{q_k}$  we apply
Corollary \ref{cor:Res2}. The restrictions are of type
$$\sum_{n=0}^{n_\ell}a_{\ell,n}q_i^n=0, \ \ \mbox{for all $\ell\in\calw$ and
$1\leq i\leq k$}.$$
\begin{proposition}\label{prop:whh1a} With the above notations,
the number of independent relations
(or,  $p_g-h^1(Z, \calO(-kE^*_{v_0}))$, cf. Corollary \ref{cor:Res2})
is $\sum_{\ell\in\calw} \min\{n_{\ell}+1,  k\}$. Hence
$$h^1(Z, \calO(-kE^*_{v_0}))=\sum_{\ell\in \calw} \max\{ 0, n_{\ell}+1-k\}.$$
\end{proposition}
\begin{proof} Use the previous discussion and
 $p_g=\sum_{\ell\in \calw} n_{\ell}+1$. \end{proof}

E.g., for $k=1$  the number of independent relations is
$\#\calw$ and $h^1(Z, \calO(-E^*_{v_0}))=p_g-\#\calw$.

\subsection{The Zariski tangent space of $\im (c^{l'})$ at
$\calO_Z(l')$, for $l'=-kE^*_{v_0}$, $k\geq 1$.}
\label{ss:Zariski} \

Take first $k=1$, $\calL =\calO_Z(l')=\calO_Z(-E^*_{v_0})$, and let
$T_{\calL}\im (c^{l'})$ be the Zariski tangent space of $\im (c^{l'})$ at $\calL$.
By Theorem \ref{th:whNLB} $O_q\in (c^{l'})^{-1}(\calL)$ for
any $q\in E_{v_0}\setminus \cup_jP_j$,
and (cf. Corollary \ref{cor:Res2} and \ref{ss:h1egy})
$\im T_{\calO_q}(c^{l'})$ is the kernel of forms $\sum a_{\ell,n} \omega_{\ell, n}$
with $\sum_n a_{\ell, n}q^n=0$ for all $\ell\in \calw$.
 We wish to describe the space
generated by all subspaces
$\im T_{\calO_q}(c^{l'})\subset T_{\calL}\im (c^{l'})$ when we move $q$. By taking
$(n_{\ell}+1)$ different values $q_r$ we get that the vectors
$(q^0_r, q^1_r, \ldots, q^{n_{\ell}})$
(dual to the hyperplane  $\sum_n a_{\ell,n}q_r^n=0$) are linearly independent
(since their Vandermonde determinant is non--zero), hence
$\sum_q\im T_{\calO_q}(c^{l'})=T_{\calL} \pic^{l'}(Z)$, the whole
tangent space of $\pic^{l'}(Z)$
at $\calL$. Hence we proved the following statement for $k=1$.
\begin{theorem}\label{th:Zariski}
$T_{\calL}\im (c^{l'})=T_{\calL} \pic^{l'}(Z)$ for any $l'=-kE^*_{v_0}$, $k\geq 1$.
\end{theorem}
The general case $k\geq 1$ follows from the case $k=1$ and (\ref{eq:addclose}).

\subsection{The value $h^1(Z,\calO_Z(-E^*_{j,s_j}))$}\label{ss:h1ketto}\
Fix some leg $j$ and
consider the corresponding end--vertex $E_{j,s_j}$ and
the natural line bundle $\calO_Z(-E^*_{j,s_j})$. If $O_j$ denotes the
intersection of the special  $\C^*$--orbit with $\tX$,
then in $U_{j,s_j}$ it is given by $\{v_{j,s_j}=0\}$ and
 by Theorem \ref{th:whNLB} $\calO_Z(-E^*_{j,s_j})=\calO_Z(O_j)$.
We apply again
Corollary \ref{cor:Res2} for the forms $\omega_{\ell,n}$ in
 $U_{j,s_j}$ (cf. \ref{ss:whbasis})
$$
u_{j,s_j}^{\tau_j\ell -\omega_j'm_j+\omega_j'-1}
v_{j,s_j}^{\omega_j\ell -\alpha_jm_j+\alpha_j-1}
(u_{j,s_j}^{\omega_j'} v_{j,s_j}^{\alpha_j}+p_j)^n \cdot
\textstyle{\prod _{j'\not=j}}
(u_{j,s_j}^{\omega_j'} v_{j,\s_j}^{\alpha_j}+p_{j'}-p_j)^{-m_j}\, dv_{j,s_j}
\wedge du_{j,s_j}.
$$
Some of the forms have no poles along $\{u_{j,s_j}=0\}$, hence they
 determine no restrictions. That is, any restriction appears only if
$\tau_j\ell -\omega_j'm_j+\omega_j'-1<0$.

We recall that the $v_{j,s_j}$ exponent
$\omega_j\ell -\alpha_jm_j+\alpha_j-1$ is non--negative. However, if this
exponent is strict positive, then the restriction to
$\{v_{j,s_j}=0\}$ is zero. Hence, restriction appears only if this exponent
 is exactly zero. Note that $\omega_j\ell -\alpha_jm_j+\alpha_j-1=0$
if and only if $\alpha_j| \omega_j\ell-1$.

\begin{proposition}\label{prop:whh1b}
The number of independent relations, $p_g-h^1(Z,\calO_Z(-E^*_{j,s_j}))$, is
$$\textstyle{\sum_{\ell\in\calw}} n_{\ell}\cdot \#\{\ell\,:\,
\tau_j\ell -\omega_j'\ce{\ell\omega_j/\alpha_j}+\omega_j'-1<0, \ \
 \alpha_j|\omega_j\ell-1\}.$$
\end{proposition}

\subsection{The dimension of $\im(c^{l'})$ for $l'=-E^*_{v_0}$} \label{ss:whDim} \
Let $d:=(Z,l')$, this is the dimension of
$\eca^{l'}(Z)$ (cf. Theorem \ref{th:smooth}). In fact,
$\eca^{l'}(Z)$ projects onto $E_{v_0}\setminus \cup_j P_j$ with fibers $\simeq  \C^{d-1}$.
We are again in chart $U_0$, and  we
 simplify the coordinate notations $(u_0,v_0)$ into $(u,v)$.
 We have to restrict the forms $\omega_{\ell,n}$ to the generic
 transversal cut $\widetilde{D}_{gen}$ given by
$\{v=c_0+c_1u+\cdots + c_{d-1}u^{d-1}\}$. 
In this generic
case the linear system is more complicated, the rank is much harder to compute.

Recall $p_g=\sum_{\ell\in\calw} (n_{\ell}+1)$.
We define the function $s:\Z_{\geq 0}\to \Z_{\geq 0}$ by decreasing induction.
First, set $s(\ell)=0$ for any $\ell\gg0$ (e.g. for any $\ell $
larger than any element of $\calw$).
Then define
 \begin{equation}\label{eq:FORMULAdef}
 s(\ell) :=\left \{\begin{array}{ll} \max\{0, s(\ell+1)-1\} \ \ &
 \mbox{if $\ell\not\in\calw$}, \\
 s(\ell+1)+n_{\ell}   & \mbox{if $\ell\in\calw$}.
 \end{array}\right.\end{equation}
 \begin{lemma}\label{lem:ineqs0} (a) $s(0)\leq p_g-\#\calw$.

(b)  The following facts are equivalent:
 (i) $s(0)=p_g-\#\calw$, (ii) $s(\ell)=0$ for all $\ell\geq 0$,
 (iii) $n_{\ell}=0$ for all $\ell\in\calw$, (iv) $p_g=\#\calw$ and $s(0)=0$.
 \end{lemma}
 \begin{proof} {\it (a)}
 By a decreasing induction one gets
$$s(\ell)=\sum _{\ell'\geq \ell,\, \ell\in \calw}n_{\ell'}  -\#\{
 \ell'\,:\, \ell'\geq \ell,\ \ell'\not\in\calw,\ s(\ell'+1)>0\}.$$
 In particular,
 \begin{equation}\label{eq:sell}
 s(0)=p_g-\#\calw-\#\{
 \ell\,:\, \ell\geq 0,\ \ell\not\in\calw,\ s(\ell+1)>0\}\leq p_g-\#\calw.
 \end{equation}
 {\it (b)} \
 {\it (ii)$\Rightarrow$(iii)$\Rightarrow$(iv)$\Rightarrow$(i)} are easy, we prove  {\it (i)$\Rightarrow$(ii)}. We use from (\ref{eq:sell}) that
 ($\dag$) $\{\ell\geq 0,\ \ell\not\in\calw,\ s(\ell+1)>0\}=\emptyset$.
 Recall that $0\not\in\calw$ (cf. \ref{rem:POLE}). Hence $s(1)=0$
 (and necessarily $s(0)=0$ too). If $1\in \calw$ then $s(2)=0$
 from the definition of $s(1)$, if $1\not\in \calw$ then $s(2)=0$ from ($\dag$).
Repeating the argument, we get {\it (ii)}.
 \end{proof}

 \begin{theorem}\label{th:whh1c}
 $h^1(Z,\calO_Z(\widetilde{D}_{gen}))=s(0)$, hence
  $$\mbox{number of independent relations}=
\dim\im (c^{l'})=p_g-s(0).$$
\end{theorem}
Recall that if $O_{gen}$ is the intersection of the generic $\C^*$--orbit
with $\tX$, then the natural line bundle $\calO_Z(-E^*_{v_0})$ is
$\calO_Z(O_{gen})$ and $h^1(Z, \calO_Z(O_{gen}))=p_g-\#\calw$.
This identity and the one from Theorem \ref{th:whh1c} show that  the
inequality from (\ref{eq:sell}) is compatible with the semicontinuity
of $h^1$, cf.   \ref{lem:semicont}.
\begin{corollary}
$h^1(Z,\calO_Z(\widetilde{D}_{gen}))=h^1(Z, \calO_Z(O_{gen}))$ if and only if
$n_{\ell}=0$ for all $\ell\in \calw$, and in this case in fact
$h^1(Z, \calO_Z(O_{gen}))=0$ and
$c^{-E^*_{v_0}}$
is dominant.
\end{corollary}

In the next proof we
write  $\calw=\{\ell_1, \ldots, \ell_{max}\}$ with $\ell_1< \cdots < \ell_{max}$.
\bekezdes\label{bek:proof} {\it Proof of Theorem \ref{th:whh1c}.}
According to Corollary \ref{cor:Res2} we have to find the dimension
\begin{equation*}
\{\{a_{\ell, n}\}_{\ell\in \calw, \, 0\leq n \leq n_{\ell}}
\,:\, {\rm Res}_{\widetilde{D}_{gen}}(\textstyle{\sum}_{\ell,n}
a_{\ell, n}\omega_{\ell, n}) \
 \mbox{has no pole along $\widetilde{D}_{gen}$} \}.
\end{equation*}
Recall that $\widetilde{D}_{gen}=\{v=c_0+c_1u+\cdots + c_{d-1}u^{d-1}\}$ with
$c_i$ generic. In fact, what we will need is $c_1\not=0$.
If we set $v':=v-(c_0+c_1u+\cdots + c_{d-1}u^{d-1})$ then
$dv\wedge du=dv'\wedge du$, hence in the form $dv\wedge du$ can be
replaced by $dv'\wedge du$,
and ${\rm Res}_{\widetilde{D}_{gen}}(\textstyle{\sum}_{\ell,n}
a_{\ell, n}\omega_{\ell, n})=(\textstyle{\sum}_{\ell,n}
a_{\ell, n}\omega_{\ell, n}/dv)|_{\widetilde{D}_{gen}}$.

The vanishing of poles provides a linear system whose matrix will be labelled
as follows. The columns are indexed by $\{\ell, n\}_{\ell\in\calw, n}$.
We subdivide  this in vertical blocks. The first one,  $B(\ell_1)$,
contains $n_{\ell_1}+1$ columns, the second one,  $B(\ell_2)$,
contains $n_{\ell_2}+1$ columns, etc. Their columns are indexed by the
corresponding $n$'s.
The rows are indexed by the pole orders: the first row corresponds to
$u^{-\ell_{max}-1}$, the last one to $u^{-1}$.

We define also the sub--block
$B'(\ell_i)$ ($\ell_i\in\calw$) of  $B(\ell_i)$, it is obtained from
$B(\ell_i)$ by deleting the rows corresponding to pole orders
strict higher than $\ell_i+1$ (rows $u^{-\ell-1}$, $l > l_i$).
In fact, all the entries of the deleted part are zero, and the
highest  row kept in $B'(\ell_i)$ contains non--zero entries.

We proceed by induction: we fix some $\ell$, $0\leq \ell\leq \ell_{max}$,
and assume that $\ell_{i-1}<\ell\leq \ell_i$ for some $\ell_i\in\calw$.
Then we consider the submatrix $M(\ell)$ (in the up--right corner)
containing all the entries from the columns contained in the
vertical blocks $B(\ell_j)$ with $\ell_j\geq \ell_i$, and
from the rows $u^{-l'-1}$ with $l\leq l'\leq l_{max}$.
It is the matrix of a linear system
with $l_{max}-\ell+1$ equations and with variables
$\{\{a_{\ell_j, n}\}_{\ell_j\in \calw, \, \ell_j\geq \ell_i, \,
0\leq n \leq n_{\ell_j}}$, what we formulate next.
By decreasing induction we prove that $s(\ell)$ is exactly the dimension
$\dim(\ell)$ of
\begin{equation*}\{\{a_{\ell_j, n}\}_{\ell_j\geq \ell_i,\,
0\leq n \leq n_{\ell_j}}\,:\,
(\textstyle{\sum}_{\ell_j\geq \ell_i,n}
a_{\ell_j, n}\omega_{\ell_j, n}/dv)|_{\widetilde{D}_{gen}} \
 \mbox{has no pole of order $\geq \ell+1$}\}.
\end{equation*}
If $\ell=\ell_{max}$,
then the system contains $n_{\ell_{max}}+1$ variables and
a nontrivial equation (one checks that  at least one entry
of the system is non--zero), hence $\dim(\ell_{max})=n_{\ell_{max}}=s(\ell_{max})$.

When we step from $\ell+1$ to $\ell$ ($0\leq \ell<\ell_{max}$),
we have to consider two cases.

First assume that $\ell\in\calw$ (say $\ell=\ell_i$).
Then   we add $n_{\ell}+1$ new variables and one new equation.
In the columns corresponding to the new variables only the last row
contains non--zero entries, but this part indeed contains at least one
non--zero entry. Hence the new equation is linearly independent
from   the old ones,
and $\dim(l)=\dim(l+1)+n_{\ell}+1-1$; this is the inductive step for $s(\ell)$ too.

If $\ell\not \in \calw$, say $\ell_{i-1}<\ell < \ell_i$, then in the
new system
one has the same number of variables, but there is one more equation
corresponding to the new row $\ell$. We divide this case into two subcases.
First, assume that the rank of $M(\ell+1)$ equals the number of
columns $\sum_{\ell_j\geq \ell_i}(n_{\ell_j}+1)$. Hence, adding a
new row we cannot increase the rank, hence $\dim (\ell+1)=\dim(\ell)$. In fact,
$\dim(\ell+1)=0$, and the new equation (even if it is  `generic')
cannot decrease the dimension of the system.

In the second case we assume that the rank of $M(\ell+1)$ is strict smaller
 than the number of
columns $\sum_{\ell_j\geq \ell_i}(n_{\ell_j}+1)$.
In this case we claim ${\rm rank}(M(\ell))={\rm rank}(M(\ell+1)+1)$,
 hence $\dim (\ell)=\dim (\ell+1)-1$. This again agrees with
 (\ref{eq:FORMULAdef}). The claim follows  from the next lemma
via standard  linear algebra.
 \begin{lemma}\label{lem:PROOFlemma1} Fix $\ell_i\in\calw$.
 Assume that the hight $\ell_i+1$ of the sub--block $B'(\ell_i)$
 is not smaller then its width $n_{\ell_i}+1$. Then the top
 $(n_{\ell_i}+1)\times (n_{\ell_i}+1)$ minor $M'$ of $B'(\ell_i)$
 has non--zero determinant. Furthermore, if
 the hight $\ell_i+1$ of the sub--block $B'(\ell_i)$ is smaller then
 its width $n_{\ell_i}+1$ then the  $\ell_i+1$ rows of
 $B'(\ell_i)$ are linearly independent.
 \end{lemma}
We prove this lemma in two steps. The first step is the next statement.
 \begin{lemma}\label{lem:PROOFlemma2} For any $m\in \Z_{>0}$
 we construct the $m\times m$--matrix $M(c)$ as follows.
 Its $ n^{th}$--column consists of the
 first $m$ coefficients of the series $(\sum_{k\geq 0} c_ku^k)^{n-1}$.
 E.g., the first column has entries $(1, 0,  \dots )$, the second one
 $(c_0,c_1,\ldots)$, the third one $(c_0^2, 2c_0c_1, \dots)$.
 Then $\det\, M(c)=c_1^{m(m-1)/2}$.
\end{lemma}
\begin{proof}
Assume that
 $\{C_n\}_{1\leq n\leq m}$ are the columns of $M(c)$. Consider the matrix
$M(c)'$ consisting of columns $\{C'_n\}_{1\leq n\leq m}$, $C_1'=C_1$,
$C_2'=C_2-c_0C_1$, $C_3'=C_3-2c_0C_2+c_0^2C_1$, etc.
Then $det\, M(c)=\det\, M(c)'$. But $C_n'$ consists of the coefficients of
$(-c_0+\sum_{k\geq 0} c_ku^k)^{n-1}$, hence the entries of $M(c)'$ above the diagonal
are zero, and on the diagonal one has the entries  $1, c_1, c_1^2, \ldots$.
\end{proof}

Finally we prove Lemma \ref{lem:PROOFlemma2}. We apply Lemma \ref{lem:PROOFlemma1}
for $m=\ell_i+1$. Note that we can consider only the first situation
(when the height is sufficiently large, since in the other case
the matrix can be completed to a square matrix of size $n_{\ell_i}+1$ with
non--zero determinant.

For fixed $\ell_i$ the forms have the form $\omega_{\ell_i,n}=u^{-\ell_i-1}f_{\ell_i}(v_0)v_0^n dv_0\wedge du_0$, $0\leq n
\leq n_{\ell_i}$, where  $f_{\ell_i}(v_0):=\prod_j (v_0-p_j)^{-m_j}$.
Let us we substitute $v_0\mapsto \sum_{k\geq 0} c_ku^k$ in this
function, and consider its Taylor expansion $P(u)=b_0+b_1u+\cdots$.
Then the columns of the top minor $M'$ of $B'(\ell_i)$ are the coefficients of the product
$(\sum_{k\geq 0} c_ku^k)^n\cdot P(u)$. Since $b_0\not=0$, this means that
$M'$ is obtained from $M(c)$ by multiplication with an invertible matrix.
Hence $\det\, M'\not=0$.

This ends the proof of Theorem \ref{th:whh1c}.

\subsection{The Abel map  $c^{l'}$ in a neighborhood of some $0_q$ supported by $E_{v_0}$}\label{ss:whAbel} \
Since we have a basis of differential forms,  using the results and the
notations of subsection \ref{ss:abelforms} we are able to give
the `complete Abel map'. Indeed, assume that $O$ is the intersection of a  generic $\C^*$--orbit
with $\tX$, and in some local chart  it is given by $v=0$. Consider the
parametrization of its neighbourhood  in $\eca^{l'}(Z)$ in the form
$D(c)=\{v=c_0+c_1u+\cdots+ c_{d-1}u^{d-1}\}$ ($|c_0|\ll 1$),
where $d=(Z,l')$ is the dimension of
 $\eca^{l'}(Z)$. Above we constructed $p_g$  differential forms
 having in this chart the  expressions
 $\omega^0_{\ell,n}=u^{-\ell-1}f_{\ell,n}v^ndv\wedge du$, where $f_{\ell,n}:=
 \prod_j (v-p_j)^{-m_j}$, $\ell\in \calw$   and $0\leq n\leq n_{\ell}$.
  Then the Abel map restricted to this chart is
  $c\mapsto
(\langle\langle D(c),
\omega_{\ell,n}\rangle\rangle )_{\ell,n}$, where each coordinate
$\langle\langle D(c),
\omega_{\ell,n}\rangle\rangle$ is determined explicitly in (\ref{eq:SYM5}).

The reader is invited to take his/her favorite star--shaped graph,
determine explicitly the forms and the corresponding Abel map.
Here we will exemplify the general description by an example when
the image of the Abel map is a singular hypersurface.

\begin{example}\label{ex:whSING} Consider the star--shaped graph with $b_0=4$,
$\nu=8$, $(\alpha_j,\omega_j)=(8,1)$ for all $j$. Then $p_g=3$.
By a computation  $\calw =\{1\}$, and the three forms are
$u^{-2}f(v)v^n$, where $n=0,1,2$, and $f(v)=\prod_{j}(v-p_j)^{-1}$.
By a computation (using Laufer's algorithm) $Z_{min}=E+E_{v_0}$,
and $h^1(\calO_{Z_{min}})=3$. In fact, $\lfloor Z_K \rfloor=Z_{min}$.
We determine the Abel map $c^{l'}$ for $Z=Z_{min}$ and $l'=- E^*_{v_0}$.
Hence $d=(Z,l')=2$, and $D(c)=\{v=c_0+c_1u\}$. By (\ref{eq:SYM5}) the Abel map is
$$(c_0,c_1)\mapsto (-c_1f(c_0), -c_1c_0f(c_0), -c_1c_0^2f(c_0)).$$
If $(A,B,C)$ are the coordinates in the target, then $\im(c)=\{AC=B^2\}$. It is surprising
that $\im(c)$ is independent of the choice of the points $\{p_j\}_j$ (that is, of the analytic structure of $(X,o)$).
\end{example}

\end{document}